\documentclass[10pt]{report}
\usepackage[utf8]{inputenc}
\usepackage{amsmath}
\usepackage{amsthm}
\usepackage{amsfonts}
\usepackage{amssymb}
\usepackage{color}
\usepackage[all]{xy}
\usepackage{pdfpages}
\usepackage{url}

\newtheorem{teo}{Theorem}[section]
\newtheorem{lema}[teo]{Lemma}
\newtheorem{cor}[teo]{Corollary}

\newtheorem{de}[teo]{Definition}

\newtheorem{pr}[teo]{Proposition}
\newtheorem{rk}[teo]{Remark}
\newtheorem{ex}[teo]{Example}
\newtheorem{say}[teo]{}
\newtheorem{defteo}[teo]{Definition/Proposition}

\def\frontpage{%
  \cleardoublepage
  \thispagestyle{empty}
  \begin{center}
       {\bfseries\MakeUppercase{{\rm \bf Birational Geometry of Toric Varieties}}}
  \vskip 1cm
  \begin{center}
    
    {\rm Edilaine Ervilha Nobili} 
  \end{center}
     \end{center}
  \vskip 1cm   
\begin{center} 

\textbf{ Abstract }    
\end{center} 

Toric geometry provides a bridge between algebraic geometry and combinatorics of fans and polytopes. For each polarized toric variety $(X,L)$ we have associated a polytope $P$. In this thesis we use this correspondence to study birational geometry for toric varieties. To this end, we address subjects such as Minimal Model Program, Mori fiber spaces, and chamber structures on the cone of effective divisors. We translate some results from these theories to the combinatorics of polytopes and use them to get structure theorems on space of polytopes. In particular, we treat toric varieties known as 2-Fano, and we classify them in low dimensions.  \\ \\

\textbf{Acknowledgements.} This work is a Ph.D. thesis under the supervision of
Carolina Araujo. I thank Carolina Araujo for proposing me the study of this interesting subject, for many hours of discussions, for important ideas, and for always encouraging me.
 I also thank Ana-Maria Castravet, Alicia Dickenstein, Sandra Di Rocco, and Diane Maclagan  for their suggestions and fruitful discussions.}

\begin{document}

\frontpage
\tableofcontents

\chapter{Introduction}
A toric variety is a normal algebraic variety $X$ containing an algebraic torus $T\simeq (\mathbb{C}^*)^n$ as an open dense subset, together with an action $T\times X\rightarrow X$ extending the natural action of $T$ on itself. To each toric variety $X$ one associates a fan $\Sigma_X$. Many geometric proprieties of $X$ are encoded as combinatorial proprieties of $\Sigma_X$. There are also connections between projective toric varieties and polytopes. Each polarized toric variety $(X,L)$, where $L$ is an ample $\mathbb{Q}$-divisor on $X$, defines a rational polytope $P$. This correspondence is one to one and again there are links between geometric proprieties of $(X,L)$ and combinatorial aspects of the associated polytope $P$.

Toric varieties play an important role in the algebraic geometry. Thanks to their combinatorial description, they provide several examples and have been a natural place to test general conjectures and theories. The main reference for an introduction to toric varieties is Fulton's book \cite{fulton}. \\

One of the most important achievements in birational classification of algebraic varieties is the so called
Minimal Model Program (MMP for short).  
The aim of the MMP is to run a succession of special birational transformations on $X$  in order to achieve a variety 
$X'$ that is birationally equivalent to $X$ satisfying one of the following:
\begin{enumerate}
\item $K_{X'}$ is nef (i.e., $K_{X'}\geq 0$), or
\item $X'$ admits a structure of Mori 
fiber space (i.e., there exists an elementary fibration $f:X'\rightarrow Y$ such that $-K_{X'}$ is $f$-ample).
\end{enumerate}

The birational transformations allowed in MMP are very special: they are either divisorial contractions or flips. 
Since a projective toric variety is birationally equivalent to $\mathbb{P}^n$, its canonical class can never be 
made nef,  so the MMP for projective toric varieties always ends with a Mori fiber space. \\

The MMP was proved by Mori for threefolds in \cite{mori} and for toric varieties in \cite{reid}. 
Recently, a special instance of the Minimal Model Program, the MMP with scaling, was established for arbitrary dimension in \cite{scaling}.

One of the problems that we are interested in is the study of MMP from the viewpoint of polytope theory. In \cite{reid} Reid has
established the MMP for toric varieties by interpreting it as a sequence of operations on the associated fans. In this thesis, we consider a similar
problem. Given a polarized toric variety $(X,L)$ we define operations on the associated polytope that describe each
step of the MMP with scaling for $(X,L)$. Our first task is to give a complete description of polytopes associated to Mori fiber spaces. We call these Cayley-Mori polytopes, and provide an explicit facet presentation for them.  \\

Next, we summarize the results established in this thesis:
\begin{enumerate}

\item We introduce a new class of polytopes, called Cayley-Mori polytopes that correspond precisely to Mori fiber spaces.
\item We describe the structure of spaces of polytopes.
\item We describe the Minimal Model Program with scaling as an operation on polytopes.
\item We investigate toric 2-Fano varieties, providing a classification in low dimension.
\end{enumerate}

\section{Cayley-Mori Polytopes and Mori Fiber Spaces}

We say that a simple polytope is a Cayley-Mori polytope if it is isomorphic to a polytope of the form 
$P_0*P_1*...*P_k=conv((P_0\times {w_0}),...,(P_k\times{w_k}))\subset\mathbb{R}^n\times\mathbb{R}^k$, where $P_0,...,P_k$ are $n$-dimensional strictly
 combinatorially equivalent polytopes, $\{w_1,...,w_k\}$ is a basis for $\mathbb{R}^k$, and $w_0=0$. We prove that these polytopes correspond precisely to Mori fiber spaces obtained from $\mathbb{Q}$-factorial projective toric varieties by running MMP.

\section{Spaces of Polytopes}
Let $v_i\in \mathbb{Z}^n$,  $1\leq i\leq r$, be distinct primitive vectors such that 
$\mbox{cone}(v_1,...,v_r)=\mathbb{R}^n$.
Set $\mathcal{H}=\big(v_1, \cdots, v_r\big)$.
For each $a=(a_1,..., a_r)\in \mathbb{R}^r$ define the polytope:

$$P_a=\Big\{x\in\mathbb{R}^n\Big| \langle v_i,x\rangle \geq -a_i, 1\leq i\leq r \Big\}.$$

We define the space of polytope presentations $\mathcal{PP}_{\mathcal{H}}$ as 
$$
\mathcal{PP}_{\mathcal{H}}=\Big\{a\in \mathbb{R}^r \ \Big| \ P_a \ \text{is a nonempty polytope} \ 
\Big\}\subset \mathbb{R}^r.
$$
Since two distinct element of $\mathcal{PP}_{\mathcal{H}}$ can define the same polytope, one is led to consider the quotient
$\mathcal{PP}_{\mathcal{H}}\ /\sim$, where $\sim$ is the equivalence relation that 
identifies elements $a, b\in \mathcal{PP}_{\mathcal{H}}$ such that $P_a=P_b$.  

We prove that $\mathcal{PP}_{\mathcal{H}}\ /\sim$ can be realized
as an $r$-dimensional closed convex polyhedral subcone 
of $\mathcal{PP}_{\mathcal{H}}\subset \mathbb{R}^r$, denoted by $\mathcal{P}_{\mathcal{H}}$. We show that there is a point $a_0\in\mathcal{PP}_{\mathcal{H}}$ such that $P_{a_0}$ is a simple polytope (i.e. each vertex is contained in exactly $n$ edges) which has exactly $r$ facets. Let $X$ be the toric variety defined by $P_{a_0}$. The cone $Eff(X)$ of effective divisors on $X$ admits a decomposition in convex cones called GKZ decomposition of $X$. We use this decomposition to get structure theorems for spaces of polytopes. We define a fan supported on $\mathcal{PP}_{\mathcal{H}}\subset \mathbb{R}^r$ satisfying the following conditions: Polytopes associated to elements in the relative interior of the same cone of this fan are strictly combinatorially isomorphic. Moreover, $\mathcal{P}_{\mathcal{H}}$ is the union of some of the maximal cones in this fan.

\section{Polytope MMP}
Let $(X,L)$ be a polarized $\mathbb{Q}$-factorial toric variety, where $L$ is an ample $\mathbb{Q}$-divisor on $X$. 
Let $P:=P_L$ be the polytope associated to $L$. For each $s\geq 0$ we define $P^{(s)}$ as the set of those points in $P$ 
whose lattice distance to every facet of $P$ is at least $s$. These polytopes are called  \textit{adjoint polytopes} in 
\cite{sandra}. Let~$\sigma(P):=sup\{s\in \mathbb{R}_{\geq 0} \mid P^{(s)}\neq\emptyset\}$. The polytope $P^{(\sigma(P))}$ 
is called the \textit{core} of P. When we increase $s$ from $0$ to $\sigma(P)$, $P^{(s)}$ will change its combinatorial 
type at some critical values. The first one is $$\lambda_1:=sup\{s\in \mathbb{R}_{\geq 0} \mid P \ and \ P^{(s)} \ have \ the \ same \ normal \ fan\}=$$ $$=sup\{s\in \mathbb{R}_{\geq 0} \mid L+sK_X \ is \ nef\}$$ the \textit{nef value} of $P$. 

Our aim is to describe the family of polytopes $P^{(s)}$ for values of $s$ between 0 and $\sigma(P)$. We will prove the following result. See Definition \ref{general} for the precise notion of a general polytope.\\
 
\textbf{Theorem.} Let $(X,L)$ be a polarized $n$-dimensional $\mathbb{Q}$-factorial toric variety associated to a ``general'' rational polytope $P\subset\mathbb{R}^n$.
Then there exist sequences 
$$0=\lambda_0<\lambda_1<...<\lambda_k=\sigma(P), \ \ \ \ \ X=X_1 \stackrel{f_1}{\dashrightarrow} X_2\stackrel{f_2}{\dashrightarrow} ...\stackrel{f_k}{\dashrightarrow} X_{k+1}$$ 

of rational numbers and rational maps, such that: 
\begin{enumerate}
\item For $i\in\{1,...,k-1\}$, $f_i$ is either a divisorial contraction or a flip. 
\item For $\lambda_i <s,t<\lambda_{i+1}, \ P^{(s)}$ and $P^{(t)}$ are $n$-dimensional simple polytopes with the same normal fan. 
\item At $s=\lambda_i$, one of the following occurs. 
\begin{enumerate}

\item Either $P^{(\lambda_i)}$ is simple and $P^{(\lambda_i)}$ has one less facet than $P^{(\lambda_{i-1})}$ (equivalently, $f_i$ is a divisorial contraction), or 
\item $P^{(\lambda_i)}$ is not simple and $P^{(\lambda_i)}$ has the same number of facets as $P^{(\lambda_{i-1})}$ (equivalently, $f_i$ is a flip). 
\item For each $i\in\{1,...,k\}$, denote by $m_i$ the dimension of the locus where $f_i$ is not an isomorphism. Then, for $\lambda_i<s<\lambda_{i+1}$, the polytope $P^{(s)}$ has exactly one more $m_i$-dimensional face than $P^{(\lambda_{i+1})}$, and this face is a Cayley-Mori polytope.  
\end{enumerate}
\item For $\lambda_{k-1}<s<\lambda_k=\sigma(P)$, $P^{(s)}$ is a Cayley-Mori polytope, (equivalently $f_k$ is a Mori fiber space, and $X_{k+1}$ is the toric variety associated to $P^{(\sigma(P))}$).

\item Let $K(P)$ be the linear space parallel to $Aff(Core(P))$ and consider the natural projection $\pi_P:\mathbb{R}^n\rightarrow \mathbb{R}^n/K(P)$ associated to $P$. The toric variety associated to the polytope $Q:=\pi_P(P)$ is the closure of the general fiber of the rational map $f:=f_k\circ...\circ f_1:X\dashrightarrow X_{k+1}$.
\end{enumerate}
Moreover, if $P^{(\lambda_i)}$ is simple then $X_{i+1}$ is the toric variety associated to it.
  Otherwise, if $P^{(\lambda_i)}$ is not simple, the toric variety associated to $P^{(\lambda_i)}$ is the image of the small contraction corresponding to the flip $f_i$ and $X_{i+1}$ is associated to $P^{(s)}$ for $\lambda_i<s<\lambda_{i+1}$.

\section{2-Fano Toric Varieties}
A $\mathbb{Q}$-factorial projective variety $X$ is said to be Fano if has ample anti-canonical divisor. One important aspect of Fano varieties is that they appear in the MMP as fibers of Mori fiber spaces.
In addition to their role in the MMP, Fano varieties are important for their own sake, and have been very much studied. 
Fano varieties are quite rare. It was 
proved by Kollár, Miyaoka and Mori that, for a fixed dimension, there exist only finitely many smooth Fano varieties up to 
deformation  (see \cite{kollar}, \cite{kol}). Further, in the toric case, there exist only finitely many isomorphism classes of 
them.\\

A smooth Fano variety $X$ is said to be 2-Fano if its second Chern character is positive (i.e., $ch_2(T_X)\cdot S >0$ for every
 surface $S\subset X$). These varieties were introduced by de Jong and Starr in \cite{starr} and \cite{jason} in connection 
with rationally simply connected varieties, which in turn are linked with the problem of finding rational sections for 
fibrations over surfaces.  2-Fano varieties are even more scarce than Fano varieties. Few examples of 2-Fano varieties are known. 
First de Jong and Starr gave some examples in \cite{jason}, then in \cite{ana} Araujo and Castravet found some more examples. 
Among all known examples, the only smooth toric 2-Fano varieties are projective spaces. So, it is natural to pose the following question:\\
   
   \textbf{Question 1:} Is $\mathbb{P}^n$ the only n-dimensional smooth projective toric 2-Fano variety? \\

In \cite{eu} we have answered this question positively when $n\leq 4$ by using the classification of toric Fano 4-folds given by Batyrev. Then, we used a database provided by  Øbro to answered the question positively in dimension 5 and 6. In order to approach this question in the general case, we investigate what happens with the 
second Chern character when we run the Minimal Model Program. Since in the toric case  the MMP ends with a Mori fiber space, we start investigating the second Chern character of a Mori fiber space. 
   
   Since Mori fiber spaces are associated to Cayley-Mori polytopes, we study
combinatorial proprieties of these polytopes, and translate them into geometric proprieties about Mori fiber spaces. 
We will show that is possible to find a birational model $X'$ of $X$ with structure of Mori fiber space, such that the general fibers are 
projective spaces. Then, we will show that such a variety  
 $X'$ cannot be 2-Fano.
 In particular, we will show that if $X$ is a smooth toric variety which is a Mori fiber space then $X$ cannot be a 2-Fano variety. As a corollary, the only $n$-dimensional smooth projective toric 2-Fano variety with Picard number $\leq 2$ is $\mathbb{P}^n$. On the other hand, if we allow singularities on 
$X$, we can give examples of Mori fiber spaces with Picard number $2$ that are 2-Fano. We also prove that a 2-Fano toric variety cannot admit certain types of divisorial contractions.\\ \\

This thesis is structured as follows. In Chapter 2 we give an overview of general theory for toric varieties and Mori theory for toric varieties. We warn that this review is rather concise. For more details we indicate \cite{cox}, \cite{fulton} and \cite{reid}.

In Chapter 3 we introduce the class of Cayley-Mori polytopes. We give a combinatorial description for these polytopes and prove that they are associated to Mori fiber spaces.

In Chapter 4 we present spaces of polytopes and give structure theorems on this space. Next, we define an operation on polytopes and relate this with the MMP with scaling.

In Chapter 5 we approach the problem of classification of Fano varieties having positive second Chern character. These varieties are called 2-Fano. We provide a classification of toric 2-Fano varieties in low dimension. We describe a strategy to classify these varieties in arbitrary dimension and give some partial results.
 We finalize this work in Chapter 6 giving the Maple code used in Chapter 5 to compute the second Chern character of a smooth projective toric variety.

\chapter{Preliminaries on Toric Varieties}
\section{Constructing Toric Varieties}
Throughout this thesis, a  \emph{variety} means an algebraic integral separated scheme of finite type over $\mathbb{C}$. A \emph{subvariety} of a variety is a closed subscheme which is a variety, and by a \emph{point} on a variety we mean a closed point.

In this chapter we will give an overview of basic facts about toric varieties. The definitions and statements of the theorems can be found in \cite{cox} and \cite{fulton}, unless otherwise noted.

 \begin{de} A \emph{toric variety} is a normal variety $X$ containing an algebraic torus $T\simeq (\mathbb{C}^*)^n$ as an open dense subset, together with an action $T\times X\rightarrow X$ extending the natural action of $T$ on itself. 
\end{de}

There is a combinatorial way to obtain toric varieties. Let $N\simeq \mathbb{Z}^n$ be a lattice and 
$M:=Hom_{\mathbb{Z}}(N,\mathbb{Z})\simeq \mathbb{Z}^n$ its dual lattice. Set $N_{\mathbb{R}}:=N\otimes_{\mathbb{Z}}\mathbb{R}$. Given $u\in M$ and $v\in N$ we denote 
$u(v)$ by $\langle u,v\rangle$. Each element $u=(u_1,...,u_n)\in M$ defines a character of the torus denoted by 
$\chi^u:=x_1^{u_1}\cdot ... \cdot x_n^{u_n}\in k[T]=k[x_1^{\pm 1},...,x_n^{\pm 1}]$. 

\begin{de} A convex rational polyhedral cone in $N_{\mathbb{R}}$ is a set of the form:
$$\sigma:=\left\lbrace\displaystyle\sum_{i=1}^{k}a_iv_i\in N_{\mathbb{R}} \ \Big| \ a_i\geq 0\right\rbrace,$$
for some finite collection of elements $\{v_1,...,v_k\}\subset N$. If $\sigma$ contains no line we say that it is strongly convex. We will call it a ``cone'' for short.
\end{de}
Let $\sigma\subset N_{\mathbb{R}}$ be a cone of dimension n and consider the dual cone of $\sigma$ given by $\sigma^{\vee}:=\{u\in M_{\mathbb{R}}:=M\otimes_{\mathbb{Z}}\mathbb{R}\mid \langle  u,v\rangle\geq 0 \ \text{for all} \ v\in\sigma\}$. 

The semigroup $S_{\sigma}:=\sigma^{\vee}\cap M$ is finitely generated. Then $A_{\sigma}:=\mathbb{C}[S_{\sigma}]=\mathbb{C}[\chi^u]_{u\in S_{\sigma}}$ is a finitely generated $\mathbb{C}$-algebra defining the affine toric variety                                                     
$U_{\sigma}:=Spec(A_{\sigma})$ of dimension $n$. 
Recall that a fan $\Sigma\subset N_{\mathbb{R}}\simeq\mathbb{R}^n$ is a finite collection of rational polyhedral strongly convex cones $\sigma\subset N_{\mathbb{R}}$ such that:

\begin{enumerate}

\item If $\sigma$ and $\tau$ belong to $\Sigma$ then $\sigma\cap\tau$ is a face of $\sigma$ and $\tau$;
\item If $\sigma\in\Sigma$ and $\tau$ is a face of $\sigma$ then $\tau\in\Sigma$.
\end{enumerate}
The set of the $m$-dimensional cones of $\Sigma$ will be denoted by $\Sigma(m)$.
A cone $\tau\in\Sigma(n-1)$ is called a wall when it is the intersection of two $n$-dimensional cones of $\Sigma$.
 
If $\Sigma\subset N_{\mathbb{R}}$ is a fan then the affine toric varieties $U_{\sigma}$ for $\sigma\in\Sigma$ glue together to a toric variety $X_{\Sigma}$.

A collection $\Sigma$ of convex cones  is called a degenerate fan if it satisfies the two conditions above and  there is a nontrivial rational linear subspace $U\subset N_{\mathbb{R}}$ such that, for every cone $\sigma\in\Sigma$, we have $\sigma\cap -\sigma=U$. In this case, $\Sigma /U$ defines a fan with respect to the quotient lattice $N/(U\cap N)$ whose associated toric variety is also denoted by $X_{\Sigma}$. 
It is a classical result that every toric variety is obtained from a fan (see for instance \cite[Corollary 3.1.8]{cox}). \\

There is a one to one correspondence between the points of an affine toric variety $U_{\sigma}$ and the semigroup homomorphisms $\varphi: S_{\sigma} \rightarrow \mathbb{C}$. For each cone $\sigma$ we have a distinguished point $x_{\sigma}\in U_{\sigma}$ that corresponds to the following semigroup homomorphism 
$$m\in S_{\sigma} \longmapsto
\left\{
\begin{array}{ll}
1 & \mbox{if } m\in S_{\sigma}\cap \sigma^{\perp} \\
0 & \textrm{otherwise}. \\
\end{array}
\right.
$$

Let $X_{\Sigma}$ be the toric variety associated to a fan $\Sigma\subset N_{\mathbb{R}}$. The distinguished points of $X_{\Sigma}$ determine $T_N$-invariant subvarieties of $X_{\Sigma}$. We have the following correspondence
(The Orbit-Cone Correspondence, see for instance \cite[3.2.6]{cox}):
\begin{enumerate}
\item There is a bijective correspondence
$$
\begin{array}{rl}
\{\sigma\in\Sigma\}\longleftrightarrow & \{T_N\text{-orbits in} \ X_{\Sigma}\}. \\
\sigma\longleftrightarrow & O_{\sigma}:=T_N\cdot x_{\sigma} \\
\end{array}
$$
\item Let $n=$dim $N_{\mathbb{R}}$. Then dim $O_{\sigma}=n-$dim $\sigma$.
\item $U_{\sigma}=\displaystyle\bigsqcup_{\tau\prec\sigma}O_{\tau}$.
\item $\tau\prec\sigma\Leftrightarrow O_{\sigma}\subseteq \overline{O_{\tau}}$, and $V(\sigma):=\overline{O_{\sigma}}=\displaystyle\bigsqcup_{\sigma\prec\gamma}O_{\gamma}$,
where $\overline{O_{\sigma}}$ denotes the closure in both the classical and Zariski topologies.
\end{enumerate}

 The $T$-invariant subvariety $V(\sigma)$ of $X_{\Sigma}$  has the structure of a toric variety given by the following fan:
 
 Consider the sublattice $N_{\sigma}:=span(\sigma)\cap N$ of $N$ and let $N(\sigma)=N/N_{\sigma}$. The collection of cones $Star(\sigma):=\{\bar{\tau}\subset N(\sigma)_{\mathbb{R}}\mid \sigma\prec\tau\in\Sigma\}$ is a fan, where $\bar{\tau}$ is the image of $\tau\in\Sigma$ in $N(\sigma)$, and $X_{Star(\sigma),N(\sigma)}\simeq V(\sigma)$.

Let $\Sigma$ and $\Sigma'$ be fans with respect to lattices $N$ and $N'$ respectively. Consider a lattice homomorphism $\Phi: N\rightarrow N'$. We say that $\Phi$ is \textit{compatible} with $\Sigma$ and $\Sigma'$ or, $\Sigma$ and $\Sigma'$ are compatible with $\Phi$, if for each cone $\sigma\in\Sigma$ there exists a cone $\sigma'\in\Sigma'$ with $\Phi_{\mathbb{R}}(\sigma)\subset\sigma'$. In this case, $\Phi$ induces an equivariant morphism $\phi:X_{\Sigma} \rightarrow X_{\Sigma'}$. Moreover, this is a toric morphism. This means that $\phi$ maps the torus $T_N$ of $X_{\Sigma}$ on the torus $T_{N'}$ of $X_{\Sigma'}$ and $\phi\mid_{T_N}$ is a group homomorphism. In fact, every toric morphism arises in this way.

We recall an important result involving distinguished points and toric morphisms (see for instance \cite[3.3.21]{cox}).

\begin{pr}\label{pr1} Let $\phi: X_{\Sigma}\rightarrow X_{\Sigma'}$ be a toric morphism induced by a map $\Phi:N\rightarrow N'$ that is compatible with $\Sigma$ and $\Sigma'$. Given $\sigma\in\Sigma$, let $\sigma'$ be the minimal cone of $\Sigma'$ such that $\Phi_{\mathbb{R}}(\sigma)\subset \sigma'$. Then:
\begin{enumerate}
\item $\phi(x_{\sigma})=x_{\sigma'}$.
\item $\phi(O_{\sigma})\subseteq O_{\sigma'}$ and $\phi(V(\sigma))\subseteq V(\sigma')$.
\item The induced map $\phi\mid_{V(\sigma)}:V(\sigma)\rightarrow V(\sigma')$ is a toric morphism.\\

\end{enumerate}

\end{pr}

There are deep connections between toric varieties and polytopes. We can construct toric varieties from rational polytopes. Recall that a set $P\subset M_{\mathbb{R}}$ is a rational polytope if $P$ is a convex hull of a finite set $S\subset M_{\mathbb{Q}}$. When $S\subset M$ we call $P$ a lattice polytope. When $P$ is full dimensional we can write a polytope $P$, in a minimal way, as intersection of finitely many closed half spaces $$P:=\Big\{m\in M_{\mathbb{R}}\ \Big| \ \langle m,u_F\rangle\geq -a_F, \text{for all facets } F\prec P\Big\},$$
where $u_F\in N$ is a primitive vector normal to the facet $F$ and $a_F\in\mathbb{Q}$. This is called a \textit{facet presentation} of $P$. When $P$ is full dimensional, it  has a unique facet presentation. Note that if $P$ is a lattice polytope then $a_F\in\mathbb{Z}$.
Each face $Q$ of $P$ defines a cone $\sigma_Q:=\mbox{Cone}(u_F)_{Q\prec F}\subset N_{\mathbb{R}}$. The collection of these cones forms a fan (also known as the normal fan of $P$) $\Sigma_P$ and then we get a toric variety $X_P$ associated to $P$. When $P$ is not a full dimensional polytope we set $X_P$ to be the toric variety associated to $R:=Aff(P)\cap P$ with respect to the lattice $Aff(P)\cap M$, where $Aff(P)$ denotes the smallest affine space containing $P$. In this case, $\Sigma_P$ will denote the degenerate fan each of whose cones is generated by a cone of $\Sigma_R$ and the linear space $span(P)^{\perp}$. \\

A fan encodes many algebraic proprieties of the variety associated to it. Let $\Sigma\subset N_{\mathbb{R}}$ be a fan. The support of $\Sigma$ is $|\Sigma|=\displaystyle\bigcup_{\sigma\in\Sigma}\sigma$.  We say that $\Sigma$  is \textit{complete} if $|\Sigma|=N_{\mathbb{R}}$. A cone $\sigma\in\Sigma$ is \textit{smooth (resp. simplicial)} if its minimal generators form part of a $\mathbb{Z}$-basis of $N$ (resp. $\mathbb{R}$-basis of $N_{\mathbb{R}}$). We say that $\Sigma$ is  smooth (resp. simplicial) if every cone of $\Sigma$ is smooth (resp. simplicial). We say that a polytope $P$ is smooth (resp. simplicial) if the same holds for $\Sigma_P$. 

Next proposition is a standard result for toric varieties. For the proof see for instance \cite[3.1.19 and 4.2.7]{cox}.
\begin{pr}  The toric variety $X_\Sigma$ is smooth, complete or $\mathbb{Q}$-factorial if and only if $\Sigma$ is, respectively, smooth, complete or simplicial. 
\end{pr}

 A polytope $P$ will be called smooth (resp. simple) if $\Sigma_P$ is smooth (resp. simplicial).
Note that if $P\subset M_{\mathbb{R}}$ is a rational polytope then for every $k\in\mathbb{Q}$ and $v\in M_{\mathbb{Q}}$, the poytopes $P, kP$ and $P+v$ define the same fan. For this reason, we will frequently suppose that $P$ is a lattice polytope.

Given a fan $\Sigma\subset N_{\mathbb{R}}$, a fan $\Sigma'$  in $N_{\mathbb{R}}$ \textit{refines} $\Sigma$ if $|\Sigma'|=|\Sigma|$ and every cone of $\Sigma'$ is contained in a cone of $\Sigma$. In this case, the identity map on $N$ induces a toric birational morphism $\phi:X_{\Sigma'}\rightarrow X_{\Sigma}$. 

There is a special type of refinement called \textit{star subdivision} which we describe below.

Given a fan $\Sigma\subset N_{\mathbb{R}}$ and a nonzero primitive element $v\in |\Sigma|\cap N$, let $\Sigma (v)$ be the set of the following cones:
\begin{itemize}
\item $\sigma$, where $v\notin\sigma\in\Sigma$.
\item cone$(\tau,v)$, where $v\notin\tau\in\Sigma$ and $\{v\}\cup\tau\subset\sigma\in\Sigma$.
\end{itemize} 
The collection $\Sigma (v)$ is a fan which refines $\Sigma$, called by \textit{star subdivision} of $\Sigma$ at $v$. When $\Sigma$ is smooth and there is a maximal cone $\sigma=\mbox{Cone}(v_1,...,v_n)\in\Sigma$ such that $v=v_1+...+v_n$, the refinement $\Sigma (v)$ induces a morphism $X_{\Sigma (v)}\rightarrow X_{\Sigma}$ that corresponds the blow up of $X_{\Sigma}$ at the point $V(\sigma)$. For details we refer to \cite[§11.1]{cox} and \cite{sato}. 

\section{Toric Varieties and Divisors}

Let $\Sigma$ be a fan in $N_{\mathbb{R}}$. We denote by $\Sigma(1):=\{v_1,...,v_r\}$ the set of all minimal generators of $\Sigma$. Given $\sigma\in\Sigma$ we set $\sigma(1):=\sigma\cap\Sigma(1)$. By the orbit-cone correspondence $D_i:=V(v_i)$ is a $T$-invariant subvariety of $X_{\Sigma}$ of codimension 1. Hence $D_i$ is an invariant prime divisor on $X_{\Sigma}$. By abuse of notation, we often write $i\in\Sigma(1)$  meaning $v_i\in\Sigma(1)$.  \\
For $m\in M$, the character $\chi^m$ is a rational function on $X_{\Sigma}$. We have the following nice description of its divisor.
\begin{pr} \label{div} $div(\chi^m)=\displaystyle\sum_{v_i\in\Sigma(1)}\langle m,v_i\rangle D_i$ for every $m\in M$.

\end{pr}

\begin{de} A Weil divisor $D:=\sum a_iD_i$ on the toric variety $X_{\Sigma}$ is said to be $T$-invariant (or simply invariant) if every prime divisor $D_i$ such that $a_i\neq 0$ is invariant under the action of the torus. This happens exactly when $D_i=V(v_i)$ with $v_i\in\Sigma(1)$.
\end{de}

\begin{pr}\label{chow} Let  $X:=X_{\Sigma}$ be a toric variety of dimension $n$. Then the Chow group $A_k(X)$ of $X_{\Sigma}$ is generated by the classes of $k$-dimensional invariant subvarieties. In other words, $A_k(X)$ is generated by $[V(\sigma)]$ as $\sigma$ ranges over the $(n-k)$-dimensional cones from $\Sigma$.  In particular, the set $T$-$Div(X_{\Sigma})$ of invariant divisors by the torus action generate the class group $\operatorname{Cl}(X_{\Sigma})$.
\end{pr}

Now, we will describe intersection products between invariant divisors and invariant varieties (see for instance \cite[Section 5.1]{fulton}).

\begin{pr}\label{inter} Let $X_{\Sigma}$ be a toric variety, $D=\displaystyle\sum_{i\in\Sigma(1)}a_iD_i$ an invariant Cartier divisor on $X_{\Sigma}$ and $V(\sigma)$ a $k$-dimensional subvariety of $X_{\Sigma}$ that is not contained in the support of $D$. Then $D\cdot V(\sigma)=\displaystyle\sum b_{\gamma}V(\gamma)$, where the sum is over all $(k+1)$-dimensional cones $\gamma\in\Sigma$ containing $\sigma$, and $b_{\gamma}$ is obtained in the following way:\\

Let $v_i\in\Sigma(1)$ be any primitive ray of $\gamma$ that is not a primitive ray of $\sigma$.  Let $e$ be the generator of the one-dimensional lattice $N_{\gamma}/N_{\sigma}$ such that the image of $v_i$ in $N_{\gamma}/N_{\sigma}$ is $s_i\cdot e$, with $s_i$ a positive integer. Then $b_{\gamma}=\displaystyle\frac{a_i}{s_i}$.
\end{pr}

\begin{rk} \label{rkinter} \em{ When $X_{\Sigma}$ is smooth, $D_i$ is a Cartier divisor for every $i\in\Sigma(1)$ and we have:}
 $$D_i \cdot V(\sigma)  =
\left\{
\begin{array}{ll}
  V(\gamma)  & \text{if} \ v_i \ \text{and} \ \sigma \ \text{span a cone } \gamma \text{ of}\  \Sigma\\
  0 & \text{if} \ v_i \ and \ \sigma \ \text{do not span a cone of}\  \Sigma \ \ . \\
\end{array}
\right.
$$
\end{rk}

The following proposition gives a criterion for a Weil invariant divisor to be a Cartier invariant divisor.

\begin{pr} Let $\sigma\in\Sigma$ be a cone and $D$ a Weil invariant divisor on $X_{\Sigma}$. Then:
\begin{enumerate}
\item $D|_{U_{\sigma}}$ is a $T$-invariant Cartier divisor on $U_{\sigma}$ $\Leftrightarrow$ there exists $u\in M$ such that $D|_{U_{\sigma}}=div(\chi^u)$.
\item $Pic(U_{\sigma})=0$.
\item $D:=\displaystyle\sum_{v_i\in\Sigma(1)}a_iD_i$ is Cartier $\Leftrightarrow D$ is principal on the affine open subset $U_{\sigma}$ for all $\sigma\in\Sigma$ $\Leftrightarrow$ for each $\sigma\in\Sigma$, there is $u_{\sigma}\in M$ such that $\langle u_{\sigma},v_i\rangle =a_i$ $\forall \ v_i\in\sigma(1)$.
\end{enumerate}
\end{pr}

Let $X_P$ be the toric variety of a full dimensional rational polytope $P\subset M_{\mathbb{R}}$. The polytope $P$ has a facet presentation $$P:=\Big\{m\in M_{\mathbb{R}}\ \Big| \ \langle m,u_F\rangle\geq -a_F, \text{for all facets } F\prec P\Big\}.$$ The  fan $\Sigma_P$ tells us the facet normal $u_F$ but it does not give any information about the numbers $a_F$. In fact, we have already noted that $P$ and $kP$ give the same fan for all $k\in\mathbb{Q}$. We define a special $\mathbb{Q}$-divisor $D_P:=\sum a_FD_F$ on $X_P$, where $D_F=V(u_F)$.

\begin{pr} If $P\subset M_{\mathbb{R}}$ is a full dimensional rational polytope, then                                                                                                         $D_P$ is an ample $\mathbb{Q}$-Cartier invariant divisor on $X_P$.
\end{pr}

A divisor $D\in T$-$Div(X_{\Sigma})$ on the toric variety $X_{\Sigma}$ defines a sheaf $\mathcal{O}_{X_{\Sigma}}(D)$. We give a description of the global sections $\Gamma (X_{\Sigma},\mathcal{O}_{X_{\Sigma}}(D))$.

\begin{pr} \label{p.global} If $D\in T$-$Div(X_{\Sigma})$ then $$\Gamma (X_{\Sigma},\mathcal{O}_{X_{\Sigma}}(D))\simeq\displaystyle\bigoplus_{div(\chi^m)+D\geq 0}\mathbb{C}\cdot\chi^m.$$

\end{pr}

\begin{de} Let $X_{\Sigma}$ be a toric variety and let $D:=\displaystyle\sum_{v_i\in\Sigma(1)}a_iD_i$ be an invariant Weil divisor on $X_{\Sigma}$. We define a polyhedron associated to $D$ as
$$P_D:=\Big\{m\in M_{\mathbb{R}} \ \Big| \  \langle m,v_i\rangle \geq -a_i \text{ for all }v_i\in\Sigma(1)\Big\}.$$
\end{de}

Note that $$\Gamma(X_{\Sigma},\mathcal{O}_{X_{\Sigma}}(D))\simeq\displaystyle\bigoplus_{m\in P_D\cap M}\chi^m.$$

 When  $\Sigma$ is a complete fan, $P_D$ is a polytope.  Furthermore, if $D$ is an ample $\mathbb{Q}$-divisor on  $X_{\Sigma}$ then the toric variety associated to $P_D$ is isomorphic to  $X_{\Sigma}$ and $D_{P_D}=D$.
In other words, a full dimensional rational polytope $ P\subset M_{\mathbb{R}}$ defines a polarized toric variety, that is, a pair $(X_P,D_P)$ where $D_P$ is an ample invariant $\mathbb{Q}$-divisor on $X_P$. Conversely, a pair $(X_{\Sigma},D)$ where D is an ample invariant $\mathbb{Q}$-divisor on $X_{\Sigma}$, defines a rational polytope $P_D$. These constructions are inverses to each other. In particular, a toric variety is projective if and only if it is the variety associated to a full dimensional rational polytope.\\

\begin{rk}\label{r.global} \em{ It follows from Propositions \ref{div} and \ref{p.global} that  if $D=D'+div(\chi^m)$ for some $m\in M$ then $P_D=P_{D'}+m$, and conversely. In particular, $D$ is linearly equivalent to an effective
divisor if and only if $P_D\neq~\emptyset$.}
\end{rk}

\begin{de} Let $P\subset M_{\mathbb{R}}$ and $P'\subset M'_{\mathbb{R}}$ be full dimensional rational polytopes. We say that $P$ and $P'$ are isomorphic if there exists an isomorphism between the corresponding polarized toric varieties $(X_P,D_P)$ and $(X_{P'},D_{P'})$. Equivalently, if there is a bijective affine transformation $\rho:M_{\mathbb{R}}\rightarrow M'_{\mathbb{R}}$ 
induced by a bijective affine transformation on the lattices, satisfying $\rho(P)=P'$.
\end{de}

\section{Mori Theory for Toric Varieties}

\subsection{Definitions and Standard Results}
\begin{de} Let $X$ be a complete variety and $D$ a Cartier divisor on $X$. We say that $D$ is nef (numerically effective) if $D\cdot C\geqslant 0$ for every irreducible complete curve $C$.
\end{de}

\begin{pr} Let $X$ be a complete toric variety and $D$ a Cartier divisor on $X$. The following are equivalent:
\begin{enumerate}
\item $D$ is nef
\item $\mathcal{O}_X(D)$ is generated by global sections.
\item $D\cdot C\geqslant 0$ for every invariant irreducible curve $C\subset X$.
\end{enumerate} 
\end{pr}
 
The theorem below is a well known criterion of ampleness for a Cartier divisor. We will enunciate it in the toric case.
\begin{teo} \label{kleiman} (Toric Kleiman Criterion). Let $D$ be a Cartier Divisor on a complete toric variety $X_\Sigma$. Then $D$ is ample if and only if 
 $D\cdot C>0$ for every invariant curve $C\subset X_{\Sigma}$, that is, $D\cdot V(\sigma)>0$ for every cone $\sigma\in\Sigma$ of dimension equal to $(rank(\Sigma)-1)$.
\end{teo}

\begin{de} Let $X$ be a normal variety, $D$ and $F$ be Cartier divisors on $X$. We say that $D$ and $F$ are numerically equivalent, and we write $D\equiv F$, if $D\cdot C=F\cdot C$ for every irreducible complete curve $C\subset X$.
\end{de}

\begin{pr} \cite[6.3.15]{cox} For complete toric varieties numerical equivalence and linear equivalence coincide. 
\end{pr}

We also have a numerical equivalence on curves. 

\begin{de} Let $X$ be a toric variety and denote by $Z_1(X)$ the free abelian group generated by irreducible complete curves contained in $X$. Given $C$ and $C'$ cycles in $Z_1(X)$, we say that $C$ and $C'$ are numerically equivalent, written $C\equiv C'$, if $D\cdot C=D\cdot C'$ for every Cartier divisor $D$ on X. 
 
\end{de}

\begin{de} For a toric variety $X$ we define the vector spaces $N^1(X):=(CDiv(X)/\equiv)\otimes_{\mathbb{Z}}\mathbb{R}=Pic(X)_{\mathbb{R}}$ and $N_1(X):=(Z_1(X)/\equiv)\otimes_{\mathbb{Z}}\mathbb{R}$. 
\end{de}

The linear map $$N^1(X)\times N_1(X)\rightarrow \mathbb{R}$$ induced by the intersection product makes $N_1(X)$ and $N^1(X)$ dual vector spaces to each other.\\

In $N^1(X)$ and  $N_1(X)$ there are some important cones for Mori theory.

\begin{defteo} Let $X$ be an $n$-dimensional complete toric variety.
\begin{enumerate}

\item The effective cone $Eff(X)$ of $X$ is the cone in $\operatorname{Cl}(X)\otimes\mathbb{R}$ generated by effective divisors. Equivalently, by Remark \ref{r.global}, the effective cone is generated by the divisors $D$ such that $P_D$ is not empty.
\item When $X$ is $\mathbb{Q}$-factorial, $Eff(X)\subset N^1(X)$ and the interior of $Eff(X)$ is the big cone $Big(X)$ (see \cite[2.2.26]{robert}). 
\item The nef cone $Nef(X)$ of $X$ is the cone in $N^1(X)$ generated by nef Cartier divisors.
\item The interior of $Nef(X)$ is the ample cone $Amp(X)$ (see \ref{kleiman}). Equivalently, the ample cone of $X$ is generated by the Cartier divisors $D$ such that the toric variety associated to $P_D$ is isomorphic to $X$.
\item The Mori cone $\mbox{NE}(X)$ of $X$ is the cone in $N_1(X)$ dual to $Nef(X)$. Equivalently, it is the cone generated by irreducible complete curves.
\end{enumerate}
All these cones are  strongly convex polyhedral cones.
\end{defteo}

Now we state a toric version of one of the first results in the Mori theory.

\begin{teo}\label{teo cone}(Toric Cone Theorem)\cite[6.3.20]{cox} Let $X:=X_{\Sigma}$ be a complete  toric variety of dimension $n$. Then $$\mbox{NE}(X)=\displaystyle\sum_{\tau\in\Sigma(n-1)}\mathbb{R}_{\geq 0}[V(\tau)].$$\\ 
\end{teo}

\begin{pr}\label{seq}(See for instance \cite[6.4.1]{cox}) Let  $\Sigma$ be a simplicial fan  and $X:=X_{\Sigma}$. Then there is an exact sequence:

$$
\begin{array}{ccccccccc}
 0 &\longrightarrow& N_1(X) &\longrightarrow& \mathbb{R}^r&\longrightarrow&  N_{\mathbb{R}} &\longrightarrow &0 \ \ ,\\
  &&\xi&\longmapsto &\displaystyle\sum_{v_i\in\Sigma(1)}(V(\langle v_i \rangle)\cdot \xi)e_i  \\
&&&&e_i&\longmapsto & v_i \\
\end{array}
$$

 \noindent where $r:=|\Sigma(1)|$, and ${e_1,...,e_r}$ is the standard basis for $\mathbb{R}^r$.\\
\end{pr}
Proposition \ref{seq} says that $N_1(X)$ can be interpreted as the space of linear relations among the $v_i's$.

Now, consider a projective $\mathbb{Q}$-factorial toric variety $X$. Let $Z_2(X)$ and $Z^2(X)$ be the free $\mathbb{Z}$-modulo of 2-cycles and 2-cocycles respectively on $X$. We define the vector space $N_2(X):=(Z_2(X)/\equiv)\otimes_{\mathbb{Z}}\mathbb{R}$, where $``\equiv "$ denotes numerical equivalence on $Z_2(X)$: A 2-cycle $S\in Z_2(X)$ is numerically  equivalent to $0$ if $E\cdot S=0$ for every 2-cocycle $E\in N^2(X)$. 
The cone of effective 2-cycles is defined as:  $$NE_2(X):=\Big\{\Big[\displaystyle\sum_i a_iS_i\Big]\in N_2(X)\ \Big| \ a_i\geq 0\Big\},$$
\noindent where the $S_i's$ are irreducible surfaces in $X$.

 Proposition \ref{dm} gives a version of the ``Toric Cone Theorem" for the cone of effective 2-cycles. 

\subsection{Description of Contractions of Extremal Rays}\label{cont}

Throughout this section we will suppose that $X:=X_{\Sigma}$ is a projective $\mathbb{Q}$-factorial $n$-dimensional toric variety.\\

Now we discuss the connection between the faces of the nef cone and fiber spaces. We recall that if $X$ is projective 
%veja cox 7.0.6 e criterio de propriedade
then a fiber space is a morphism $f:X\rightarrow Y$ onto a normal variety with connected fibers. From the Stein factorization theorem \cite[11.5]{hartshorne} it follows that a fiber space is determined by the irreducible complete curves that are contracted by $f$.\\

Let $D$ be a nef divisor on $X$. By replacing it with a linearly equivalent divisor if necessary, we may suppose that $D$ is effective. Denote by $\Sigma_D$ the fan defined by the polytope $P_D$ and $X_{\Sigma_D}$ the toric variety defined by $\Sigma_D$.
From \cite[6.2.5]{cox} it follows that $\Sigma$ is a refinement of  $\Sigma_D$. Then we have a morphism $f_D:X\rightarrow X_{\Sigma_D}$ such that $D=f^*(A)$ for some ample divisor $A$ on $X_{\Sigma_D}$. This means that an irreducible complete  curve $C\in X$ is contracted to a point by $f_D$ if and only if $D\cdot C=0$. Therefore, $C$ is contracted by $f_D$ if and only if the class $[C]$ belongs to the face of the Mori cone defined by $\mbox{NE}(X)\cap D^{\perp}$. From the duality of the nef cone and the Mori cone, we conclude that $f_D$ only depends on the face $\tau\prec Nef(X)$ which contains $D$ in its interior.
Moreover, $\tau =f_D^*(Nef(X_{\Sigma_D}))$. 
On the other hand, if $g:X\rightarrow Y$ is a fiber space with $Y$ projective, then, up to isomorphism, $g=f_{f^*A}$ for any ample divisor $A$ on $Y$. 

We summarize our discussion in the following theorem. \\

\begin{teo}(Toric Contraction Theorem) \label{contracao} There is a bijective correspondence between the faces $\tau$ of the nef cone $Nef(X)$ and the fiber spaces $f:X\rightarrow Y$ with $Y$ projective, given by  $\tau =f^*(Nef(Y))$. Conversely, the irreducible complete curves contracted by $f$ are the curves in $\tau ^{\perp}$. Each such fiber space is a toric morphism.
\end{teo}

Let $\tau$ be a facet of $Nef(X)$ and $f:X\rightarrow Y$ be the corresponding fiber space. Then an irreducible complete curve $C\subset X$ is contracted by $f$ if and only if the class $[C]$ belongs the  1-dimensional face $\mbox{NE}(X)\cap\tau ^{\perp}$ of the Mori cone. By \ref{teo cone} this face is of the form $\mathbb{R}_{\geq 0}V(\sigma)$ for some cone $\sigma\in\Sigma(n-1)$. In this case, we say that $f$ is a contraction of a ray of $\mbox{NE}(X)$ or simply an extremal contraction.\\

Given a contraction $f_R:X\rightarrow Y$ of a ray $R\subset \mbox{NE}(X)$ we will describe how to obtain the  fan of the variety $Y$. We give an overview of the results in \cite{reid,matsuki}.\\

 Let $V(W)$ be an invariant curve which generates the ray $R$. Then we can write the wall $W:=\mbox{cone}(v_1,...,v_{n-1})$ as intersection of two maximal cones: $\mbox{cone}(v_1,...,v_{n-1},v_n)$ and $\mbox{cone}(v_1,...,v_{n-1},v_{n+1})$. Since $\Sigma_X$ is a simplicial fan there are uniquely determined rational numbers $a_1,...,a_n$  such that $v_{n+1}=-a_1v_1-...-a_nv_n$. 
Reordering the vectors $v_i's$ if necessary we may suppose that:
$$a_i
\left\{
\begin{array}{ll}
  <0  & \text{for } 1\leq i\leq\alpha ,\\
    =0  & \text{for } \alpha +1\leq i\leq\beta ,\\
      >0  & \text{for } \beta +1\leq i\leq n+1 ;\\
\end{array}
\right.
$$
with $0\leq\alpha\leq\beta\leq n-1.$ \\

For each $j\in\{1,...,n+1\}$, set $\Sigma_j:=\mbox{cone}(v_1,...,\widehat{v_j},...,v_{n+1})$ and $\Sigma_W:=\mbox{cone}(v_1,...,v_{n+1})$. One can show that $\Sigma_W=\displaystyle\cup_{j=1}^{\alpha}\Sigma_j=\displaystyle\cup_{j=\beta +1}^{n+1}\Sigma_j$.

The fan $\Sigma_Y$ is determined by the collection of maximal cones formed by the maximal cones of $\Sigma_X$ which contain no wall $W\in\Sigma_X$ with $[V(W)]\in R$ and the cones $\Sigma_W$ for every wall $W\in\Sigma_X$ such that $[V(W)]\in R$. This fan is degenerate if and only if $\alpha=0$.\\

\begin{pr}\label{fibrapicard} Let $f_R:X\rightarrow Y$ be the contraction of a ray $R$ of the Mori cone $\mbox{NE}(X)$. Denote by $\mbox{Exc}(f_R)$ the exceptional locus of $f_R$. This morphism is classified into one of the following three types:

\begin{enumerate}
\item Contraction of Fibering Type (or 
Mori Fiber Space): dim $Y<$ dim $X$.

In this case, we have

$$
\begin{array}{l}
  \alpha=0,  \\
\beta=\mbox{dim } Y. \\
\end{array}
$$

The variety $Y$ is $\mathbb{Q}$-factorial (equivalently the non degenerate fan corresponding to $\Sigma_Y$ is simplicial).

\item Divisorial Contraction: The morphism $f_R$ is birational and codim $(Exc(f_R))=\alpha=1$. In this case $Y$ is $\mathbb{Q}$-factorial and $Exc(f_R)$ is the prime $T$-invariant divisor $V(v_1)$.

\item Contraction of Flipping type (or small contraction): The morphism $f_R$ is birational and codim $Exc(f_R)=\alpha\geq 2$. In this case $Y$ is necessarily non $\mathbb{Q}$-factorial. The exceptional locus is the irreducible $T$-invariant variety $V(\mbox{cone}(v_1,...,v_{\alpha}))$.

In all cases, $Y$ is projective, the exceptional locus is an invariant irreducible variety and $f_R$ restricted to $Exc(f_R)$ is a flat equivariant morphism. 
There exists a toric $\mathbb{Q}$-factorial Fano variety $G$ with Picard number one such that $(f_R)^{-1}(p)_{red}\simeq G$ for every $p\in f_R(Exc(f_R))$.  
\end{enumerate} 
\end{pr}

\begin{rk} \em{ When $f_R:X\rightarrow Y$ is a contraction of fibering type, that is, $\alpha=0$, we can use the linear relation $v_{n+1}=-a_1v_1-...-a_nv_n$ and Proposition \ref{inter} to show that $-K_X$ is $f_R$-ample, that is, $-K_X\cdot R>0$.} 
\end{rk}

\begin{pr} If $f_R:X\rightarrow Y$ is a contraction of flipping type, then there exists a birational morphism $f_R^+:X^+\rightarrow Y$ called the flip of $f_R$ such that:
\begin{enumerate}
\item $X^+$ is projective and $\mathbb{Q}$-factorial.
\item $f_R^+$ is also a contraction of flipping type.
\item $X$ and $X^+$ are isomorphic in codimension 1 and so we have an isomorphism $N^1(X)\cong N^1(X^+)$ defined by taking strict transforms of divisors. Hence, this induces an isomorphism between their dual spaces $N_1(X)\cong N_1(X^+)$.
\item $f_R^+$ is the contraction of the extremal ray $-R$ of $\mbox{NE}(X^+)$, where $-R$ is the image of $R$ by the isomorphism $N_1(X)\cong N_1(X^+)$.
\item The divisor $-K_X$ is $f_R$-ample if and only if $K_{X^+}$ is $f_R^+$-ample. 
\end{enumerate}
The maximal cones of the fan $\Sigma_{X^+}$ are: $$\Sigma_{X^+}(n):=\Sigma_X(n)\setminus\{\Sigma_W\mid [V(W)]\in R\}\cup\{\Sigma_j\mid [V(W)]\in R, j=1,...,\alpha\}.$$ 
\end{pr}

\section{MMP for Toric Varieties}
Let $X$ be a projective $\mathbb{Q}$-factorial variety with klt singularities.
The aim of the Minimal Model Program (MMP for short) is to run a succession of divisorial contractions and flips on $X$  in order to achieve a variety 
$X'$ that is birationally equivalent to $X$ satisfying one of the following:
\begin{enumerate}
\item $K_{X'}$ is nef, or
\item $X'$ admits a structure of Mori 
fiber space (i.e., there exists an extremal contraction $f:X'\rightarrow Y$ such that $-K_{X'}$ is $f$-ample).
\end{enumerate}
In the toric case, we do not need require any condition on singularities since every projective $\mathbb{Q}$-factorial toric variety is klt.

\begin{say} \em{ Now, we recall the steps of the MMP in the toric case. 
We start with a projective $\mathbb{Q}$-factorial toric variety $X$. If $K_X$ is not nef, then there is an extremal ray $R$ of the Mori cone of $X$ such that $K_X\cdot R<0$. We consider the contraction $f_R:X\rightarrow Y$ of the ray $R$. If $f_R$ is a Mori fiber space, then we stop. If $f_R$ is a divisorial contraction we replace $X$ by $Y$ and repeat the process. Finaly, if $f_R$ is a small contraction  we replace $X$ by $X^+$ and repeat the process.
   
 For toric varieties it was proved in \cite{reid} that this process stops. 
Since a projective toric variety is birationally equivalent to $\mathbb{P}^n$ its canonical class can never be 
made nef,  so the MMP for projective toric varieties always ends with a Mori fiber space. 

\section{MMP with scaling} \label{mmp}
In the general case termination of flips is still an open problem for dimension bigger than three. However, an instance of the MMP, called MMP with scaling, was proved in \cite{scaling} for arbitrary dimension.}	 
\end{say}
Next, we recall the steps of the MMP with scaling, established in \cite{scaling}. See also \cite[3.8]{carolina}.\\

\begin{say} \em{ We start with a polarized toric variety $(X,L)$, where $X$ is projective, $\mathbb{Q}$-factorial. Here $L$ is an ample $\mathbb{Q}$-divisor on $X$. Since $X$ is  a $\mathbb{Q}$-factorial projective toric variety, $X$ is klt (see \cite[14.3.2]{matsuki}). Moreover, $-K_X$ is equal to $\displaystyle\sum_{i\in\Sigma(1)}D_i$ and therefore this divisor is always big (see for instance \cite[8.2.3]{cox}).

 Given a big $\mathbb{R}$-divisor $E$ on $X$ set $\sigma(E):=sup\{s\in \mathbb{R}_{\geq 0} \mid [E+sK_X]\in Eff(X)\}$. When $E$ is a $\mathbb{Q}$-divisor the number $\sigma(E)$ is rational (see for instance \cite[5.2]{carolina}. In addition, we consider the nef threshold $$\lambda(E,X):=sup\Big\{s\in \mathbb{R}_{\geq 0} \ \Big| \ [E+sK_X]\in Nef(X)\Big\}\leq \sigma(E).$$
This is defined provided that $[E+sK_X]$ is nef for some $s\geq 0$.
Moreover, if $E$ is a $\mathbb{Q}$-divisor, then by the Rationality Theorem $\lambda(E,X)$ is a rational number (see for instance \cite[1.5.5]{mauro}). Note that $\lambda(E,X)=\sigma(E)$ if and only if the class $[E+\lambda(E,X)K_X]$ is not big.

Set $(X_1,L_1)=(X,L)$, $\lambda_1=\lambda(L,X)$ and $\sigma=\sigma(L).$
 We will define inductively a finite sequence of pairs $(X_i,L_i)$, $1\leq i\leq k$ and rational numbers $0<\lambda_1<...<\lambda_k=\sigma$ such that the following holds. 

\begin{itemize}
\item For each $0\leq i\leq k$, the variety $X_i$ will be projective, $\mathbb{Q}$-factorial. 

\item  There are birational maps $\varphi_i:X\dashrightarrow X_i$ which are compositions of divisorial contractions and flips. 
\item  $L_i=(\varphi_i)_*L$ and $[L_i+\lambda_{i}K_{X_i}]\in\partial Nef(X_i)$. 
\end{itemize}
 Suppose we have constructed $(X_i,L_i)$ and $\lambda_{i}$. Since $[L_i+\lambda_iK_{X_{i}}]\in\partial Nef(X_i)$, there exists an extremal ray $R_i\subset \overline{NE}(X_i)$ such that  $(L_i+\lambda_iK_{X_i})\cdot R_i=0$ and $K_{X_i}\cdot R_i<0$ .\newline  

 Let $f:X_i\rightarrow Y$ be the contraction of $R_i$. 

As in \ref{fibrapicard}, we have three possibilities:

\begin{enumerate} 
 \item If dim $Y<$  dim $X_i$ then $\lambda_i=\sigma$, $Y$ is $\mathbb{Q}$-factorial, and $f$ is a Mori fiber space. In this case we stop.
 
 \item If $f:X_i\rightarrow Y$ is a divisorial contraction, then $Y$ is $\mathbb{Q}$-factorial. We set $X_{i+1}=Y$, $\varphi_{i+1}=f\circ\varphi_i:X\dashrightarrow X_{i+1}$, and $L_{i+1}=f_*L_i=(\varphi_{i+1})_*L$.
 Notice that $$L_i+\lambda_iK_{X_i}=f^*(L_{i+1}+\lambda_iK_{X_{i+1}}).$$ 
 This implies that $L_{i+1}+\lambda_iK_{X_{i+1}}$ is  nef since so is $L_{i}+\lambda_iK_{X_{i}}$. Thus $$\lambda_i\leq\lambda_{i+1}:=\lambda(L_{i+1},X_{i+1})\leq\sigma.$$
 
\item If $f:X_i\rightarrow Y$ is a small contraction, then $Y$ is not $\mathbb{Q}$-factorial. (In fact, $K_Y$ is not a $\mathbb{Q}$-Cartier divisor. Otherwise, $K_{X_i}=f^*(K_Y)$ and therefore $K_{X_i}\cdot R_i=0$). Consider the flip diagram:
$$\xymatrix{
X_i \ar[rd]_f \ar@{-->}[rr]^{\psi} & & X_i^+ \ar[ld]^{f^+} \\ 
& Y &  
}$$ 

\end{enumerate}

\noindent where $\psi$ is the associated flip, $X_i^+$ is $\mathbb{Q}$-factorial and
$f^+$ is the contraction of a $K_{X_i^+}$-positive extremal ray of $\overline{NE}(X_i^+)$.
We set $X_{i+1}=X_i^+$, $\varphi_{i+1}~=\psi\circ\varphi_i:~X\dashrightarrow X_{i+1}$ and $L_{i+1}=\psi_*L_i=(\varphi_{i+1})_*L$.
Since $(L_i+\lambda_iK_X{_i})\cdot R_i=0$, there exists a $\mathbb{Q}$-Cartier $\mathbb{Q}$-divisor $D_Y$ on $Y$ such that $L_i+\lambda_iK_{X_i}=f^*D_Y$. Then $L_{i+1}+\lambda_iK_{X_{i+1}}=(f^+)^*D_Y$. By hypothesis $L_i+\lambda_iK_X{_i}$ is nef. Thus $D_Y$ is nef and so is $L_{i+1}+\lambda_iK_{X_{i+1}}$. Therefore $$\lambda_i\leq\lambda_{i+1}:=\lambda(L_{i+1},X_{i+1})\leq\sigma.$$}
\end{say}

\begin{rk} \em{ If $[L_i+\lambda_iK_{X_i}]$ is in the relative interior of a facet of $Nef(X_i)$, then there is only one extremal ray $R_i$ satisfying $(L_i+\lambda_iK_{X_i})\cdot R_i=0$. 
Moreover, $f:X_i\rightarrow Y$ is the morphism associated to the complete linear system $\mid m(L_i+\lambda_iK_{X_i})\mid$ for $m$ sufficiently large and divisible.}
\end{rk}

\section{GKZ Decomposition}\label{gkzdecomposition}

In this section we describe the GKZ decomposition for a 
complete toric variety, and we recall some of its properties.
We refer to \cite{hering_et_all_GKZ} and \cite{mustata} for details.

Let $X=X_{\Sigma}$ be a complete toric variety and $D_1,...,D_r$ be the $T$-invariant prime divisors. Consider a possibly degenerate complete fan $\Delta$ and $I$ a subset of $\Sigma(1)=\{v_1,...,v_r\}$ such that:
\begin{enumerate}
\item Every cone of $\Delta$ is generated by rays in $\Sigma(1)\setminus I$.
 \item $X_{\Delta}$ is projective.
\end{enumerate}

There is a rational map $f:X\dashrightarrow X_{\Delta}$ induced by the natural projection $\pi:N\rightarrow N/W\cap N$, where $W$ is the maximal linear space contained in every cone of $\Delta$. Since $\Delta$ is complete, $f$ is defined on the toric variety $U=\displaystyle\bigcup_{i=1}^rU_{\langle v_i\rangle}$. The codimension of the complement of $U$ in $X$ is at least 2. Therefore we have a pull-back map $f^*$ which takes $\mathbb{Q}$-Cartier divisors on $X_{\Delta}$ to $\mathbb{Q}$-Weil divisors on $X$.

Let $A$  be an invariant nef  $\mathbb{Q}$-Cartier divisor on $X_{\Delta}$ and $\lambda=(\lambda_i)_i\in\mathbb{R}_{\geq 0}^{|I|}$. There is an injective linear map 

$$
\begin{array}{crcl}
 \Phi_{\Delta,I}: \  & Nef(X_{\Delta})\times\mathbb{R}_{\geq 0}^{|I|} & \rightarrow & \operatorname{Cl}(X)_{\mathbb{R}} \\
 & A\times\lambda & \mapsto & f^*(A)+\displaystyle\sum_{i\in I}\lambda_iD_i. \\
\end{array}
$$

The image of this map is a cone denoted by $GKZ(\Delta,I)$.\\

Let $\mathcal{S}$ be the set of such pairs $(\Delta,I)$.

The set $GKZ(X):=\{GKZ(\Delta,I)\mid (\Delta,I)\in\mathcal{S}\}$ is a fan supported on the effective cone $Eff(X)$ of $X$ satisfying the following proprieties:

\begin{enumerate}

\item The cone $GKZ(\Delta,I)$ is a maximal cone in $GKZ(X)$ 
if and only if $\Delta$ is non degenerate and simplicial, and 
$I= \Sigma(1)\setminus \Delta(1)$. Thus, each maximal cone $GKZ(\Delta,I)$ of $GKZ(X)$ determines a unique variety $X_{\Delta}$.

\item The cone $GKZ(\Delta',I')$ is a face of $GKZ(\Delta,I)$ if and only if $\Delta$ refines $\Delta'$ and $I'$ is contained in $I$.
\item If $[D]$ and $[D']$ belong to the relative interior of the cone $GKZ(\Delta,I)$, then the toric varieties associated to $D$ and $D'$, that is, the varieties corresponding to $P_D$ and $P_{D'}$, are isomorphic.
\item If $D=\displaystyle\sum_{i=1}^{r}a_iD_i$ is an effective $\mathbb{Q}$-divisor on $X$, then $[D]$ lies in the interior of the cone $GKZ(\Delta,I)$ if and only if  the normal fan $\Delta_D$ of $P_D$ is equal to $\Delta$ and $I=\{v_i\mid \langle u,v_i\rangle >-a_i \textit{ for every }u\in P_D\}$. 
\end{enumerate}

The next proposition describes the walls of the fan $GKZ(X)$.
\begin{pr} \label{gkz} Let $X=X_{\Sigma}$ be a complete toric variety. Suppose that $GKZ(\Delta,I)$ is a maximal cone in $GKZ(X)$. Then: 

\begin{enumerate}
\item If $f:X_{\Delta}\rightarrow X_{\Delta'}$ is a divisorial contraction with exceptional divisor corresponding to the primitive vector $v\in\Delta(1)$ then the cone $GKZ(\Delta',I\cup\{v\})$ is a maximal cone in $GKZ(X)$ which intersects $GKZ(\Delta,I)$ along the wall $GKZ(\Delta',I)$.

\item Suppose $v\in I$, and let $\Delta(v)$ be the fan obtained by star subdivision of $\Delta$. Then $GKZ(\Delta(v),I\setminus\{v\})$ is a maximal cone in $GKZ(X)$ which intersects $GKZ(\Delta,I)$ along the wall $GKZ(\Delta,I\setminus\{v\})$. 

\item If $f:X_{\Delta}\rightarrow X_{\Delta'}$ is a small contraction and  $f^+:X_{\Delta^+}\rightarrow X_{\Delta'}$ is the associated flip, then $GKZ(\Delta^+,I)$ is a maximal cone in $GKZ(X)$ whose intersection with $GKZ(\Delta,I)$ is the wall $GKZ(\Delta',I)$.

\item If $f:X_{\Delta}\rightarrow X_{\Delta'}$ is a contraction of fibering type, then $GKZ(\Delta',I)$ is an exterior wall of $GKZ(X)$.
\end{enumerate}  
Moreover, every interior wall in $GKZ(X)$ arises as in 1,2 or 3, and every exterior wall appears as in 4.\\
\end{pr}

\textit{Proof}. This result was proved in \cite{mustata}. For the reader's convenience we give the proof here. First of all, note that by Theorem \ref{contracao} a cone $GKZ(\Delta ',I')$ is a facet of $GKZ(\Delta,I)$ if and only if one of the following conditions holds:
\begin{enumerate}
\item $\Delta=\Delta'$ and $|I\setminus I'|=1$.
\item $I=I'$ and there is a contraction $X_{\Delta}\rightarrow X_{\Delta '}$ of a ray of $\mbox{NE}(X_{\Delta})$.
\end{enumerate}

If $f:X_{\Delta}\rightarrow X_{\Delta '}$ is a divisorial contraction then, as described in  Section \ref{cont}, $X_{\Delta'}$ is a $\mathbb{Q}$-factorial projective variety and $v$ is the only primitive vector in $\Delta (1)\setminus\Delta '(1)$. Thus, $GKZ(\Delta',I\cup\{v\})$ is a maximal cone in $GKZ(X)$, which intersects $GKZ(\Delta,I)$ along the wall $GKZ(\Delta',I)$.

The second claim follows from the first one since we have the divisorial contraction $X_{\Delta(v)}\rightarrow X_{\Delta}$ and the cone $GKZ(\Delta(v),I\setminus\{v\})$ is maximal.

To prove $3.$, note that  the cone $GKZ(\Delta^+,I)$ is maximal in $GKZ(X)$ since $X_{\Delta^+}$ is a $\mathbb{Q}$-factorial projective variety isomorphic to $X_{\Delta}$ in codimension one. So, the inclusion $GKZ(\Delta',I)\subseteq GKZ(\Delta,I)\cap GKZ(\Delta^+,I)$ has to be an equality.

Finally suppose that $f:X_{\Delta}\rightarrow X_{\Delta'}$ is a contraction of fibering type. Assume that there is a maximal cone $GKZ(\Delta'',I'')$ which intersect $GKZ(\Delta,I)$ along the wall $GKZ(\Delta',I)$. We must have $I''=I$. Otherwise, $\Delta''=\Delta'$, $I''\setminus I=\{v\}$ and $2.$ would imply $GKZ(\Delta'(v),I)=GKZ(\Delta,I)$. In this case, $f$ would be a divisorial contraction. Moreover, the morphism $X_{\Delta''}\rightarrow X_{\Delta'}$ is again a contraction of fibering type. Indeed, otherwise it follows from $1,2$ and $3$ that the maximal cone $GKZ(\Delta,I)$, adjacent to $GKZ(\Delta'',I'')$, does not correspond to a contraction of fibering type.  It follows that $\Delta''=\Delta$, a contradiction. 

The last assertion follows from the fact that we already have considered all possible walls in $GKZ(X)$. \qed \\

\chapter{Cayley-Mori Polytopes} \label{Cayley}

\section{Cayley-Mori Polytopes} \label{Cayley}

In this section we characterize the normal fans of a special class of polytopes called Cayley-Mori polytopes. These polytopes play an important role in the spaces of polytopes as described in Chapter \ref{bpg}. From the viewpoint of toric varieties, this is reflected by the fact that these polytopes are associated to Mori fiber spaces.\\

Let $X:=\mathbb{P}_Y(\mathcal{E})$ be a toric variety which is a toric projective bundle over a toric variety $Y$. 
It is well known that in this case $\mathcal{E}$ will be a decomposable bundle, that is, it is a direct sum of line 
bundles (see for instance \cite{oda}, p.41). We adopt the Grothendieck notation, that is, $\mathbb{P}_Y(\mathcal{E})=\mbox{Proj}_Y\mbox{ Sym}(\mathcal{E})$. Let $\Sigma$ be the fan of $Y$ with respect to the lattice $N$. For each $v_j\in\Sigma(1)$ set $D_j:=V(\langle v_j\rangle)$. Let $L_i:=\displaystyle\sum_{j\in\Sigma(1)}a_{i_j}D_j$, $i\in\{0,...,k\}$, be divisors on $Y$ such that
$\mathcal{E}=\mathcal{O}_Y(L_0)\oplus ... \oplus\mathcal{O}_Y(L_k)$. The fan of $X$ can be described as follows.\\

Let $e_1,...,e_k$ be the canonical basis for $\mathbb{Z}^k$ and $e_0:=-(e_1+...+e_k)$. By abuse of notation we also denote by $e_i$  the element $0\times e_i$ in the lattice $N\times \mathbb{Z}^k$. Similarly given $v\in N$ we also denote by $v$ the element $v\times 0$ in $N\times\mathbb{Z}^k$.\\

For each cone $\sigma\in\Sigma$ and $i=0,...,k$ we define $$\sigma_i:=\mbox{Cone}\big(v_j+(a_{1_j}-a_{0_j})e_1+...+(a_{k_j}-a_{0_j})e_k)\mid v_j\in\sigma(1)\big)+\mbox{Cone}(e_0,...,\widehat{e_i},...,e_k).$$\\

The fan of $X$ is built up from the cones $\sigma_i$ and their faces, as $\sigma$ ranges over the cones from $\Sigma$.\\

Now, consider $Y$ a projective toric variety and for $i=1,...,k$ let $L_i:=\displaystyle\sum_{j\in\Sigma(1)}a_{i_j}D_j$ be ample divisors on $Y$. Each $L_i$ defines a polytope $P_i:=P_{L_i}\subset M_{\mathbb{R}}$ whose associated fan is $\Sigma_Y$.

 Consider the polytope $P:=conv\big((P_0\times 0),(P_1\times e_1),...,(P_k\times e_k)\big)\subset M_{\mathbb{R}}\times\mathbb{R}^k$. In \cite[§3]{cox2}  it was proved that $P$ is defined by the following inequalities: \\

$$
\begin{array}{lc}
 \langle \hat{m},\hat{v_j}\rangle \geq-a_{0_j}, & v_j\in \Sigma_Y(1) \\
 \langle\hat{m},e_0\rangle\geq -1 &  \ \ \\
 \langle\hat{m},e_j\rangle\geq 0 &  j=1,...,k  \\
\end{array}
$$

where $\displaystyle \hat{v_j}=v_j+\sum_{i=1}^k(a_{i_j}-a_{0_j})e_i$.\\

Moreover, $\big(\mathbb{P}_Y(\mathcal{O}_Y(L_0)\oplus ... \oplus\mathcal{O}_Y(L_k)),\xi\big)$ is the polarized toric variety associated to $P$, where $\xi$ is the tautological line bundle.\\

The polytope $P$ is called a \textit{Cayley polytope}. In \cite{alicia} a generalization of Cayley polytopes was introduced, namely polytopes of the form $Cayley^ s(P_0,...,P_k):=conv\big((P_0\times 0),(P_1\times se_1),...,(P_k\times se_k)\big)$ for some positive integer $s$. Such a polytope is called an \textit{sth order generalized Cayley polytope} associated to $P_0,...,P_k$.

Now, we will define a further generalization of  Cayley polytopes. 
We will use the inequalities above and the description of the fan $\Sigma_P$ to describe the fan of this new 
object.

 \begin{de} \label{cayley} Let $P_0,...,P_k\subset \mathbb{R}^n$ be $n$-dimensional polytopes, $\{w_1,...,w_k\}$ a basis for $\mathbb{R}^k$ and $w_0=0$. Suppose 
that $P_0,...,P_k$ are strictly combinatorially equivalent polytopes, that is, they have the same normal fan $\Sigma$.
Let $P_0*P_1*...*P_k=conv\big((P_0\times {w_0}),...,(P_k\times  {w_k})\big)\subset\mathbb{R}^n\times\mathbb{R}^k$. Any simple polytope isomorphic to a polytope  of the form
$P_0*P_1*...*P_k$ will be called a Cayley-Mori polytope associated to $P_0,...,P_k$.
\end{de}
\begin{rk}\em{ Note that a Cayley-Mori polytope is defined in the same way as a Cayley polytope, replacing the canonical basis $\{e_1,...,e_k\}\subset\mathbb{R}^k$ by an arbitrary basis $\{w_1,...,w_k\}\subset\mathbb{R}^k$. }
\end{rk}

 We want to describe the fan of a Cayley-Mori rational polytope $P:=P_0*P_1*...*P_k$. We may suppose that $P$ is a lattice polytope. 
 
 Consider a linear isomorphism $A:\mathbb{R}^{n+k}\rightarrow \mathbb{R}^{n+k}$ which 
is the identity on the first $n$ coordinates and maps $e_i$ to $w_i$ for each $i \in \{1,...,t\}$. We write 
$D_{P_i}:=\displaystyle\sum_{j\in\Sigma(1)}a_{i_j}D_j$ for $i=0,...,k$.  \\
Note that the polytope $Q:=conv\big((P_0\times 0),...,(P_k\times e_k)\big)$ is mapped onto $P$. Thus 
$P$ is defined by the inequalities:

$$
\begin{array}{lc}
 \langle \hat{m},(A^*)^{-1}(\hat{v}_j)\rangle \geq-a_{0_j}, & v_j\in \Sigma_{P_i}(1) \\
 \langle \hat{m},(A^*)^{-1}(e_0)\rangle\geq -1 &  \ \ \\
 \langle \hat{m},(A^*)^{-1}(e_j)\rangle\geq 0 &  j=1,...,k  \\
\end{array}
$$

where $\displaystyle \hat{v}_j=v_j+\sum_{i=1}^k(a_{i_j}-a_{0_j})e_i$.\\

It follows that the 1-dimensional cones of $\Sigma_P$ are generated by the vectors:
\begin{itemize}

 \item $\displaystyle \tilde{v}_j:=d_j[v_j+\sum_{i=1}^k(a_{i_j}-a_{0_j})(A^*)^{-1}(e_i)],$ for every $v_j\in \Sigma(1)$.
 
\item $u_i:=s_i(A^*)^{-1}(e_i), i=0,...,k$,

\noindent where $d_j$ and $s_i$ are rational numbers which make  $\tilde{v}_j$ and $u_i$ primitive vectors for every $v_j\in \Sigma(1)$ and $i=0,...,k$.
\end{itemize} 

Hence, the fan of $P$ is given by the cones $$\sigma_i:=\mbox{cone}\big(\tilde{v}_j\mid v_j\in\sigma(1)\big)+\mbox{Cone}(u_0,...,\widehat{u_i},...,u_k)$$ and their faces, as $\sigma$ ranges over the cones $\sigma\in\Sigma$ and $i$ ranges from $0$ to $k$.\\

\section{Cayley-Mori Polytopes and Mori Fiber Spaces}

Let $N$ and $N'$ be lattices  and  $\Phi : N \rightarrow N'$ a surjective $\mathbb{Z}$-linear map.\\
If $\Sigma$ and $\Sigma'$ are fans in $N_{\mathbb{R}} $ and $N'_{\mathbb{R}} $ respectively, compatible with 
$\Phi$, then there is a toric morphism $\phi:X_{\Sigma}\rightarrow X_{\Sigma'}$.\\
Let $N_0:=ker(\Phi)$. The following sequence is exact:
$$0\longrightarrow N_0 \longrightarrow N \stackrel{\Phi}{\longrightarrow} N' \longrightarrow 0 .$$\\
Now, consider the subfan of $\Sigma$: $$\Sigma_0:=\{\sigma \in \Sigma \ | \ \sigma \subset (N_0)_{\mathbb{R}} \subseteq N_{\mathbb{R}}\}.$$
We have that $X_{\Sigma_0,N}\simeq X_{\Sigma_0,N_0}\times T_{\frac{N}{N_0}}\simeq X_{\Sigma_0,N_0}\times T_{N'}$.\\
Furthermore, $\Phi$ is compatible with $\Sigma_0 \subset N_{\mathbb{R}}$ and $\{0\}\subset N'_{\mathbb{R}}$.\\
Thus, there is a toric morphism $\phi_{{|}_{X_{\Sigma_0,N}}}:X_{\Sigma_0,N}\to T_{N'}$ such that $\phi^{-1}(T_{N'})\simeq X_{\Sigma_0,N_0}\times T_{N'}$.\\
Note that if $X_{\Sigma,N}$ is smooth then so is $X_{\Sigma_0,N_0}$.

 \begin{de} \label{d.split} We say $\Sigma$ is \textbf{ weakly split by} $\Sigma_0$ and $\Sigma'$ if there exist a subfan $\hat{\Sigma} \subseteq \Sigma $ such that:
\begin{enumerate}
\item $\Phi_{\mathbb{R}}$ maps each cone $\hat{\sigma} \in \hat{\Sigma}$ bijectively to a cone $\sigma' \in \Sigma'$. Furthermore, the map $\hat{\sigma}\mapsto \sigma'$ defines a bijection between $\hat{\Sigma}$ and $\Sigma'$.
\item For every cone $\sigma\in\Sigma$ we have $\sigma=\hat{\sigma}+\sigma_0$ with $\hat{\sigma}\in\hat{\Sigma}$ and $\sigma_0\in\Sigma_0$.
\end{enumerate}
\end{de}

\begin{de} Let $X$ be a $\mathbb{Q}$-factorial projective toric variety of dimension n. 
A toric fibration on $X$ is a flat, equivariant surjective morphism $f:X\rightarrow Y$ with connected fibers such that 
$Y$ is a projective toric variety and dim $Y<n$.
\end{de}

\begin{rk} \label{fibracaoelementar} \em{ Note that a toric fibration is also a fiber space and if all fibers of $f$ are irreducible and have Picard number one, then $f:X\rightarrow Y$ is a Mori fiber space.}
\end{rk}

\begin{rk} \em{ In \cite[§3.3]{cox} it was defined the notion of a fan $\Sigma$ being split by fans 
$\Sigma_0$ and $\Sigma'$. There, it is required the additional condition $\Phi_{\mathbb{R}}(\hat{\sigma}\cap N)=\sigma'\cap N'$ 
for each $\hat{\sigma}\in\hat{\Sigma}$. It was proved that if $\Sigma$ is split by $\Sigma_0$ and $\Sigma'$, then $\phi:X_{\Sigma}\rightarrow X_{\Sigma'}$ is a locally trivial toric fibration with fiber $X_{\Sigma_0,N_0}$. Next, we give a generalization of this result.}  
\end{rk}

\begin{teo} \label{t.split}  Let $(\Sigma,N)$ and $(\Sigma',N')$ be fans defining  $\mathbb{Q}$-factorial 
projective toric varieties and $\Phi:N\rightarrow N'$ a surjective $\mathbb{Z}$-linear map compatible with $\Sigma$ and $\Sigma'$. Let $\phi:X_{\Sigma}\rightarrow X_{\Sigma'}$ be the associated toric morphism. Then, $\phi$ is a toric fibration with irreducible fibers if and only if 
$\Sigma$ is  weakly split by $\Sigma_0$ and $\Sigma'$ as in Definition \ref{d.split}. In this case, the general 
fiber of $\phi$ is isomorphic to $X_{\Sigma_0,N_0}$ .\\
\end{teo}

\begin{proof} Suppose first that $\Sigma$ is  weakly split by $\Sigma_0$ and $\Sigma'$. 
In order to prove that $\phi: X_{\Sigma}\rightarrow X_{\Sigma'}$  is a toric  fibration with irreducible fibers we will show that all 
fibers are toric varieties of the same dimension. We have seen that the morphism 
$\phi: X_{\Sigma}\rightarrow X_{\Sigma'}$ is a trivial fibration over the torus $T_N'$ with fiber $X_{\Sigma_0,N_0}$.
If $p\in X_{\Sigma'}$ is an invariant point then there exists a maximal cone $\sigma'$ such that $p=V(\sigma')$. 
By proposition \ref{pr1},  $\phi^{-1}(p)_{red}=V(\hat{\sigma})$ is a toric variety of dimension dim$(X_{\Sigma})-$dim$(X_{\Sigma'})$.

 Now, consider an invariant divisor $V(v')$ on  $X_{\Sigma'}$. Since $span(\hat{v})\cap N_0=\{0\}$, the restriction 
$\phi\mid_{V(\hat{v})}:V(\hat{v})\rightarrow V(v')$ is induced by $$0\rightarrow N_0\rightarrow N/ span (\hat{v})\cap N\rightarrow N'/ span(v')\cap N'\rightarrow 0.$$ 
Note that $\Sigma_{V(\hat{v})}$ is split by $\Sigma_{V(v')}$ and the fan 
$\{\sigma\in\Sigma_{V(\hat{v})}\mid\sigma\in \Sigma_{V(\hat{v})}\cap\Sigma_0\}$. 
Hence, the fibers over the torus of $V(v')$ and over fixed points are toric varieties with the same dimension, 
which is dim$(V(\hat{v}))-$dim$(V(v'))=$ dim$(X_{\Sigma})-$dim$(X_{\Sigma'})$. By induction on the rank of $N'$, 
we conclude that all fibers of $\phi$ are irreducible of the same dimension and then $\phi$ has connected fibers. 
Since every toric variety is Cohen-Macaulay (see Theorem 9.2.9, \cite{cox}), this is enough to conclude that $\phi$ is a flat 
morphism (see Theorem 4.1.2, \cite{mauro}). 

Conversely, suppose $\phi$ is a toric fibration with irreducible fibers. Let $\sigma'$ be a maximal cone of $\Sigma'$. Since $\phi$ is a flat surjective morphism with irreducible 
fibers, there exists only one cone $\hat{\sigma}\in\Sigma$ such that $\hat{\sigma}$ has the same 
dimension as $\sigma'$ and $\Phi$ maps $\hat{\sigma}$ bijectively onto $\sigma'$. Let $\tau\in\Sigma$ be a maximal cone 
containing $\hat{\sigma}$. We write $\hat{\sigma}=\mbox{Cone}(v_1,...,v_k)$ and $\tau=\mbox{Cone}(u_1,...,u_m,v_1,...,v_k)$. 
Since $\Sigma$ is simplicial,
for each $i=1,...,m$ and $j=1,...,k$ the cone $\mbox{cone}(v_1,...,\hat{v_j},...,v_k,u_i)$ belongs to $\Sigma$.
If $u_i$ does not belong to $\Sigma_0$ for some $i=1,...,m$, there exists some $j$ such that 
$\Phi(\mbox{cone}(v_1,...,\hat{v_j},...,v_k,u_i))$ has the same dimension as $\sigma'$.  
In this case $V(\hat{\sigma})$ and $V(\mbox{cone}(v_1,...,\hat{v_j},...,v_k,u_i))$ are both contained in $\pi^{-1}(V(\sigma'))$. 
But this is  absurd since  every fiber is irreducible. Thus, $\tau=\hat{\sigma}+\mbox{cone}(u_1,...,u_m)$ and 
$\mbox{cone}(u_1,...,u_m)$ is contained in $\Sigma_0$. Note that since $\Sigma$ is simplicial, $\Phi$ maps each face of $\hat{\sigma}$ bijectively to a face of $\sigma'$. So, the set formed by the cones $\hat{\sigma}$ and their faces is clearly the desired fan $\hat{\Sigma}$.
\end{proof}

\begin{lema}\label{picard} Let $X_{\Sigma}$ be an $n$-dimensional complete $\mathbb{Q}$-factorial toric variety and $d$ the number of 1-dimensional cones of $\Sigma$. Then $d-\rho_X=n$.
 \end{lema}
\textit{Proof}. It follows immediately from the exact sequence (see \cite{cox}, 4.2.1): $$0\rightarrow M\rightarrow CDiv_T(X)\rightarrow Pic(X)\rightarrow 0$$
\noindent where $CDiv_T(X)$ denotes the subgroup of invariants Cartier divisors on $X$.
\qed

\begin{cor} \label{picard1} If $X_{\Sigma}$ is a smooth complete toric variety with Picard number one, then $X\simeq \mathbb{P}^n$.
\end{cor}
\begin{proof} By Lemma  \ref{picard} the complete fan $\Sigma$ has exactly $n+1$ primitive vectors $v_1,...,v_{n+1}$. Therefore, the maximal cones are of the form $cone (v_1,...,\widehat{v_i},...v_{n+1})$ for every $i\in\{1,...,n+1\}$. Since $\Sigma$ is smooth we can suppose that $v_1,...,v_{n}$ is the canonical base of $\mathbb{R}^ n$ and $v_{n+1}=-v_1-...-v_n$. Thus, $X\simeq \mathbb{P}^n$.
\end{proof}

\begin{lema} \label{lema3} Let $X$ be a toric projective variety associated to a rational polytope $P$. Let $Q$ be a face of $P$ defining a cone $\sigma_Q\in \Sigma_X$ and $Z:=V(\sigma_Q)$. Then:

\begin{enumerate}

\item There exists $u\in M_{\mathbb{Q}}$ such that $Aff(Q)+u=(\sigma_Q)^{\bot}$.

\item $Q+u\subset (\sigma_Q)^{\bot}$  is the polytope associated to the pair $(Z,D_P\mid_Z)$ with respect to the lattice $M\cap (\sigma_Q)^{\bot}$.

\end{enumerate}
\end{lema}
\begin{proof} There is no loss in supposing that $P$ is a lattice polytope. Then $P$ has a unique facet presentation $P=\{x\in M_{\mathbb{R}} \mid \langle x,u_F\rangle \geq -a_F$ for all facets $F\prec P\}$. With this notation we have $\sigma_Q=\mbox{cone}(u_F)_{F\succ Q}$ and the smallest affine space containing $Q$ is $Aff(Q)=\{x\in M_{\mathbb{R}}\mid \langle x,u_F \rangle =-a_F$ for all facets $F\succ Q\}$. So, to prove 1 we can take $u$ to be any lattice point such that $-u\in Q\cap M$.

To prove 2 we can make a translation of $P$ and assume that $u=0$. The maximal cones of $\Sigma_X$ containing $\sigma_Q$ are given by $\sigma_v:=\mbox{cone}(u_F)_{v\prec F}$ where $v$ is a vertice of $Q$. Moreover $D_P\mid _{U_{\sigma_v}}=div(\chi^{-v})$. Since $Q\subset (\sigma_Q)^{\bot}$ we conclude that $Z\nsubseteq supp(D_P)$. In this case, by Proposition \ref{inter}, $D_P\mid_ Z=\displaystyle\sum_{\gamma}b_{\gamma}V(\gamma)$,          where $\gamma$ runs through all cones containing $\sigma_Q$ such that dim$(\gamma)=$dim($\sigma_Q)+1$, $b_{\gamma}=\frac{a_{F_i}}{s_i}$ where $u_{F_i}$ is any primitive vector in $\gamma(1)\setminus\sigma_Q(1)$ and $s_i$ is a positive integer such that the image of $u_{F_i}$ in the one-dimensional lattice $N_{\gamma}/N_{\sigma_Q}$ is $u_{F_i}=s_ie$ for $e$ a generator of $N_{\gamma}/N_{\sigma_Q}$. To finish the proof, we note that if $D_Q$ is the divisor associated to $Q\subset M_{\mathbb{R}}\cap (\sigma_Q)^\perp$ and $v$ is a vertice of $Q$ with $\gamma\subset \sigma_v$, then the coefficient of $V(\gamma)$ in $D_Q$ is $\langle-v,e\rangle=\frac{1}{s_i}\langle-v,s_ie\rangle=\frac{1}{s_i}\langle -v,u_{F_i}\rangle=\frac{a_{F_i}}{s_i}=b_{\gamma}$. It follows that $D_P\mid_Z=D_Q$.
\end{proof}

\begin{teo} \label{mori} Let $X$ be the $\mathbb{Q}$-factorial projective toric variety associated to a simple rational polytope 
$P\subset M_{\mathbb{R}}$ and $\Delta$ a sublattice of $N$. Let $\pi:M\rightarrow\Lambda$ be the dual map to the inclusion $j:\Delta\hookrightarrow N$, where $\Lambda=\Delta^{\vee}$. If $j:\Delta\hookrightarrow N$ induces a Mori fiber space $f:X\rightarrow Y$ then:

\begin{enumerate}
 \item $\pi:M\rightarrow\Lambda$ is surjective, and
\item $P\simeq P_0*P_1*...*P_k$ is a Cayley-Mori polytope for some rational polytopes $P_0,...,P_k$ such that $X_{P_i}\simeq Y$ 
and $Q:=\pi_{\mathbb{R}}(P)=conv(w_0,...,w_k)$ is the polytope associated to the general fiber of $f$. 
\end{enumerate}
Conversely, if $P=P_0*P_1*...*P_k$ then the canonical inclusion $j:\mathbb{Z}^ k\hookrightarrow\mathbb{Z}^{n+k}$ induces 
a Mori fiber space $f:X_P\rightarrow Y$, where $Y\simeq X_{P_i}$.
\end{teo}

\begin{proof} Throughout the proof we suppose without loss that $P$ is a lattice polytope. Suppose that $j:\Delta\hookrightarrow N$ induces a Mori fiber space $f:X\rightarrow Y$. The first item follows from \cite[2.4]{reid}.  Let $F$ be the general fiber of $f$.  Then $F$ has the structure of toric variety given by  $\Sigma_F:=\{\sigma\in\Sigma_X\mid \sigma\subset\Delta_{\mathbb{R}}\}$. Since $D_P$ is an ample divisor on $X$, $D_P\mid_F$ is an ample divisor on F. Let $S:=P_{D_P\mid_F}\subset \Lambda_{\mathbb{R}}$ and denote by $w_0,...,w_k$ the vertices of $S$. Every vertex $w_i$ corresponds to a full dimensional cone $\tau_i\in\Sigma_F$ which defines a fixed point $p_i$ of $F$. Let $Y_i:=V(\tau_i)\subset X$ be the invariant section of $f$ passing through $p_i$.

  Let $R_i$ be the face of $P$ corresponding to $Y_i$. Note that $span(\tau_i)=\Delta_{\mathbb{R}}$, which implies $\tau_i^{\perp}=\Delta_{\mathbb{R}}^{\perp}=$ker$(\pi_{\mathbb{R}})$. By Lemma \ref{lema3} there exists $u_i\in M$ 
such that Aff$(R_i)+u_i=$ker$(\pi_{\mathbb{R}})$
and $R_i+u_i$ is the polytope associated to $(Y_i,D_P\mid_{Y_i})$ with respect to the lattice ker$(\pi)\subset M$. 
Since $Y_i\simeq Y$, we conclude that $R_0,...,R_k$ are strictly combinatorially isomorphic. Since $\Sigma_X$ is simplicial, 
the $Y_i$'s are pairwise disjoint and therefore the same holds for the $R_i$'s.\\

Let $D_P=\displaystyle\sum_{v_1i\in\Sigma(1)}a_iV(\langle v_i\rangle)$ be the invariant divisor on $X$ associated to $P$. 
Since $F$ is the a general fiber, $V(\langle v_i\rangle)\cap F\neq\emptyset$ if and only if $v_i\in\Delta$, and we have 
$D_P\mid_F=\displaystyle\sum_{v_i\in\Delta\cap\Sigma(1)}a_iV(\langle v_i\rangle)$. Using the facet presentations of $P$ 
and $S$ we see that $\pi_{\mathbb{R}}(R_i)=w_i$ and $Q=\pi_{\mathbb{R}}(P)=S$.\\

The fixed points of $X$ belong to the fibers of $f$ over the fixed points of $Y$. From \ref{fibrapicard} we have that every fiber has Picard number one. Hence, by lemma \ref{picard} every invariant fiber has $k+1$ fixed points. If $s$ is the number of fixed points of $Y$ then each $R_i$ has $s$ vertices. It follows that $X$ has 
$s(k+1)$ fixed points and therefore $P$ has $s(k+1)$ vertices. Thus, $P$ is the convex hull of $R_0,...,R_k$.
Since $F$ has $1+$dim$(F)$ fixed points, $S=conv(w_0,...,w_k)$ is a simplex in $\Lambda_{\mathbb{R}}$. Thus, we conclude that
$P=P_0*...*P_k$, where $P_i=R_i+u_i\subset \mbox{ker}(\pi_{\mathbb{R}})$.\\

Conversely, assume that $P=P_0*P_1*...*P_k$. We can suppose that $\pi$ is a canonical projection 
$\mathbb{Z}^{n+k}\rightarrow\mathbb{Z}^k$ as definition \ref{cayley}.  Since $\pi_{\mathbb{R}}(P_i)=w_i$ the affine 
subspace $Aff(P_i)$ is a translation of $ker(\pi_{\mathbb{R}})$. Let $Y$ be the projective toric variety associated to 
the polytopes $P_{i}$'s with respect to the lattice $ker(\pi)$. Then the cones of the fan $\Sigma_Y$ are contained in 
$(ker(\pi_{\mathbb{R}}))^{\vee}=\mathbb{R}^n$. We have noted in Section \ref{Cayley} that the fan $\Sigma_P$ consists of the cones 
$$\sigma_i:=\mbox{cone}\big(\tilde{v}_j\mid v_j\in\sigma(1)\big)+\mbox{cone}(u_0,...,\widehat{u_i},...,u_k)$$ and their faces, where 
$\sigma\in\Sigma_{P_j}$.  Denote by $\Sigma_0$ the fan consisting of the cones $\mbox{cone}(u_0,...,\widehat{u_i},...,u_k)$ for $i=0,...,k$ and their faces.
Then $\Sigma_P$ is split by
 $\Sigma_0$ and $\Sigma_Y$. It follows from Theorem \ref{t.split}, Lemma \ref{picard} and Remark \ref{fibracaoelementar} that 
$j:\mathbb{Z}^ k\hookrightarrow\mathbb{Z}^{n+k}$ induces a Mori fiber space $f:X_P\rightarrow Y$. \end{proof}

\begin{rk}\em{ Recall from the end of Section \ref{Cayley} that the 1-dimensional cones of $\Sigma_P$, $P=P_0*...*P_k$, are generated by the vectors $\tilde{v_i}$, $v_i\in\Sigma_{P_j}$, and $u_j$, $j\in\{0,...,k\}$. Notice that the divisors on $X$ of the form $V(\tilde{v_j})$ are pull backs of the invariant divisors $V(v_j)$ on $Y$. Moreover, the invariant divisors of the general fiber $F$ are the restrictions to $F$ of the divisors on $X$ of the form $V(u_i)$.}
\end{rk}

\begin{cor}\label{bundle} Let $X_P$ be a projective toric variety associated to the lattice polytope $P$. Then, there is a Mori fiber space $f:X_P\rightarrow Y$ with general fiber isomorphic to $\mathbb{P}^k$ if and only if there are strictly combinatorially equivalent polytopes $P_0,...,P_k$ and a positive integer $s$, such that $P\simeq Cayley^s(P_0,...,P_k)$.
Moreover, if $X_P$ is smooth we can take $s$ equal to one and therefore $X_P$ is a projective bundle $\mathbb{P}_Y(\mathcal{E})$, where $Y$ is the toric variety associated to the $P_i's$, and $\mathcal{E}$ is a decomposable vector bundle of rank $k+1$ on $Y$.
\end{cor}

\begin{proof} 
By Theorem \ref{mori}, there are strictly combinatorially equivalent polytopes $P_0,...,P_k$ such that $P=P_0*...*P_k$ is a Cayley-Mori polytope. Suppose that the general fiber of $f$ is $\mathbb{P}^k$. Then there exists a positive integer $s$ such that $P=Cayley^s(P_0,...,P_k)=conv\big((P_0\times {0}),(P_1\times {se_1})...,(P_k\times {se_k})\big)$, 
where $\{e_1,...,e_k\}$ is a basis for $\mathbb{Z}^k\subset\mathbb{R}^k$. 

As discussed in Section \ref{Cayley}, if $\sigma$ is a maximal cone of $\Sigma_Y$ then the cone $\sigma_0:=\mbox{cone}(\tilde{v_j}\mid v_j\in\sigma(1))+\mbox{cone}(e_1,...,e_k))$ is a maximal cone of $\Sigma_P$, where $\displaystyle \tilde{v_j}=d_j\Big[v_j+\sum_{i=1}^k(a_{i_j}-a_{0_j})\displaystyle\frac{e_i}{s}\Big]$ and $D_{P_i}:=\displaystyle\sum_{j\in\Sigma_P(1)}a_{i_j}D_j$ for $i=0,...,k$. If $X_P$ is smooth, then $d_j=1$ and $\displaystyle\frac{a_{i_j}-a_{0_j}}{s}$ is an integer for every $i=1,...,k$. It follows that $s$ divides $D_j-D_0$ in $Pic(Y)$ for every $j\in\Sigma_P(1)$, and the fan $\Sigma_P$ is exactly the fan of $\mathbb{P}_Y\big(\mathcal{O}\oplus\mathcal{O}(\frac{D_1-D_0}{s})\oplus ...\mathcal{O}(\frac{D_k-D_0}{s})\big).$ 

Let $D$ be an ample divisor on $Y$ such that $D+\displaystyle\frac{D_i-D_0}{s}$ is ample for every $i\in\{1,...,k\}$. We set $Q_i$ to be the polytope associated to  $D+\displaystyle\frac{D_i-D_0}{s}$. Then, $P\simeq Cayley^1(Q_0,...,Q_k)$.
\end{proof}

\begin{cor} \label{exc} Let $X$ be a smooth projective toric variety and $f_R:X\rightarrow Y$ a contraction of an extremal 
ray $R\in \mbox{NE}(X)$. Then the exceptional locus $Exc(f_R)$ is a projective bundle over a toric variety $Z$.\\
\end{cor} 
\begin{proof} By Proposition \ref{fibrapicard} the restriction $f_R\mid_{Exc(f_R)}: Exc(f_R)\rightarrow Z$ is a flat, equivariant, surjective morphism and the fibers are
invariants varieties with Picard number one. So, this restriction is a Mori Fiber space. Since $X$ is smooth, $Exc(f_R)$ is also a smooth toric variety and therefore the general fiber of $f_R\mid_{Exc(f_R)}$ is smooth. Hence, this general fiber is isomorphic to a projective space $\mathbb{P}^k$. From Corollary \ref{bundle} we conclude that the exceptional locus $Exc(f_R)$ is a projective bundle over a toric variety $Z:=f_R(Exc(f_R))$.
\end{proof}

\chapter{Birational Polytope Geometry}\label{bpg}
In this chapter we relate the Mori theory  for toric varieties with combinatorial properties of polytopes. Part of the content exposed in this chapter comes from a joint
work in progress with Carolina Araujo, Alicia Dickenstein and Sandra Di Rocco.
\section{Spaces of Polytopes} \label{Section:P_H}

Let $v_i\in \mathbb{Z}^n$,  $1\leq i\leq r$, be distinct primitive vectors such that 
$\mbox{cone}(v_1,...,v_r)=\mathbb{R}^n$.
Set $\mathcal{H}=\big(v_1, \cdots, v_r\big)$.
For each $a=(a_1,..., a_r)\in \mathbb{R}^r$ define the set:

$$P_a=\Big\{x\in\mathbb{R}^n\Big| \langle v_i,x\rangle \geq -a_i, 1\leq i\leq r \Big\}.$$

\begin{de}\label{defn:PP_H}
We define the space of polytope presentations $\mathcal{PP}_{\mathcal{H}}$ as 
$$
\mathcal{PP}_{\mathcal{H}}=\Big\{a\in \mathbb{R}^r \ \Big| \ P_a \ \text{is a nonempty polytope} \ 
\Big\}\subset \mathbb{R}^r.
$$
\end{de}

Next we will find a suitable projective $\mathbb{Q}$-factorial
$n$-dimensional toric variety $X$, and associate to each 
$a\in \mathcal{PP}_{\mathcal{H}}$ the linear equivalence class of a real effective 
divisor on $X$.

\begin{lema}\label{PP_H_nonempty}
Let $\mathcal{H}=\big(v_1, \cdots, v_r\big)$ be as above.
Then there is an element $a_0\in \mathbb{Z}^r$ such that 
$P_{a_0}$ is an $n$-dimensional simple lattice polytope having exactly
$r$ distinct facets. 
\end{lema}

\begin{proof}
Consider the polytope $Q\subset \mathbb{R}^n$ defined as the convex hull of the vectors 
$\{ v_i \ | \  1\leq i\leq r\}$. 
Since $\mbox{cone}(v_1,...,v_r)=\mathbb{R}^n$, the polytope $Q$ contains $0\in \mathbb{R}^n$
in its interior.
Hence $Q$ determines a fan $\Sigma$, whose cones are defined to be the cones over 
the faces of $Q$, that is, $\Sigma =\Sigma_{Q^{\vee}}$.
The fan $\Sigma$ defines a projective toric variety.
The primitive vectors of $\Sigma$ form a subset of $\{v_1, \cdots, v_r\}$.
After renumbering, we may assume that $\Sigma(1)=\{v_1, \cdots, v_k\}$,
with $n+1\leq k\leq r$.

By considering subsequent star subdivisions of $\Sigma$ at $v_{k+1},\cdots ,v_r$,
we obtain a fan refinement $\Sigma'$ of $\Sigma$ such that  $\Sigma'(1)=\{v_1, \cdots, v_r\}$. 
By \cite[Proposition 11.1.6]{cox}, $\Sigma'$ still defines a projective toric variety.
By  \cite[Proposition 11.1.7]{cox}, $\Sigma'$ can be further star subdivided to produce
a \emph{simplicial} fan $\Sigma''$ such that  $\Sigma''(1)=\{v_1, \cdots, v_r\}$.
The toric variety $X$ corresponding to $\Sigma''$ is then projective, $\mathbb{Q}$-factorial
and $n$-dimensional.
Let $L$ be an ample divisor on $X$. Then the polytope $P$ associated to $L$ is an 
$n$-dimensional simple lattice polytope having exactly $r$ distinct facets. 
By construction, $P=P_{a_0}$ for some $a_0\in \mathbb{Z}^r$. 
\end{proof} 
\vspace{0.5cm}
\begin{say} \label{escolha} \em{
Let $\mathcal{H}=\big(v_1, \cdots, v_r\big)$ be as above.

By Lemma~\ref{PP_H_nonempty}, there is an element
$a_0\in \mathbb{Z}^r$ such that 
$P=P_{a_0}$ is an $n$-dimensional simple lattice polytope having exactly
$r$ distinct facets. \\

Let $X$ be the $n$-dimensional $\mathbb{Q}$-factorial projective toric variety associated to $P$. Then $\Sigma_X(1)=\{v_1,\cdots, v_r\}$.
As usual, for each $i\in\{1,\cdots, r\}$, let $D_i\subset X$ be the $T$-invariant prime Weil 
divisor associated to $v_i$.

Consider the surjective linear map:}
\end{say}
$$
\begin{array}{cccc}
\mathcal{D}: & \mathbb{R}^r & \longrightarrow & N^1(X). \\
& a=\big(a_1,...,a_r\big) & \longmapsto & \mathcal{D}_a=\Big[\displaystyle\sum_{i=1}^{r}a_iD_i\Big] \\
\end{array}
$$
 
\begin{lema}\label{lemma:PP_H<->Eff}
Let the notation be as in \ref{escolha}. Then
\begin{enumerate}
	\item $\mathcal{PP}_{\mathcal{H}}=\mathcal{D}^{-1} \Big( Eff(X)\Big)$. 
	\item $\mathcal{PP}_{\mathcal{H}}$ is an $r$-dimensional closed convex polyhedral cone in $\mathbb{R}^r$.	
\end{enumerate}
Let $a\in\mathcal{PP}_{\mathcal{H}}$. Then
\begin{enumerate}
\setcounter{enumi}{2}
	\item $P_a$ is $n$-dimensional if and only if 
		$a$ lies in the interior of $\mathcal{PP}_{\mathcal{H}}$.
	\item $P_a$ and $P$ have the same normal fan if and only if $\mathcal{D}_a\in Amp(X)$.
\end{enumerate}
\end{lema}

\begin{proof}
Let $a\in \mathbb{Z}^r$. It follows from Proposition \ref{p.global} that $P_a\neq \emptyset$ if and only if $\mathcal{D}_a\in Eff(X)$, proving 1.

It follows from Propositions \ref{div} and \ref{p.global} that the effective cone of a  toric variety is generated by
the classes of the $T$-invariant prime divisors of $X$. 
So $Eff(X)$ is a $\rho_X$-dimensional closed convex polyhedral cone in $N^1(X)$,
and 2 follows. 

From the definition of bigness, we see that $\mathcal{D}_a$ is big if and only if 
$P_a$ is $n$-dimensional. 
Then item $3$ follows from the fact that $Big(X)$ is the interior of $Eff(X)$.

Item $4$ is also well known. See for instance \cite[Section 3.4]{fulton}.
\end{proof}

Notice that 2 distinct elements of $\mathcal{PP}_{\mathcal{H}}$ may define the same polytope.

\begin{ex} \em{
Set ${\mathcal{H}}=\big((1,0),(0,1),(-1,0),(0,-1),(-1,-1)\big)$, and consider the elements
$a=(1,1,1,1,2)$, and $b=(1,1,1,1,3)$. 
One can check that $P_a=P_b$.  }
\end{ex}

Hence, if one is interested in parameter spaces for polytopes, not just for
polytope presentations, then one is led to consider the quotient
$\mathcal{PP}_{\mathcal{H}}\ /\sim$, where $\sim$ is the equivalence relation that 
identifies elements $a, b\in \mathcal{PP}_{\mathcal{H}}$ such that $P_a=P_b$.  

Our next goal is to show that 
that $\mathcal{PP}_{\mathcal{H}}\ /\sim$ can be realized
as an $r$-dimensional closed convex polyhedral subcone 
of $\mathcal{PP}_{\mathcal{H}}\subset \mathbb{R}^r$.

In \cite{payne}, Payne introduced the cone $Amp^1(X)$
generated by classes of divisors on $X$ whose stable base loci
do not contain any divisorial component. There, it was proved that if $X$ is a complete $\mathbb{Q}$-factorial toric variety then $Amp^1(X)$ is the closed convex polyhedral cone $\displaystyle\bigcap_{i=1}^{r}\mbox{cone}\big(D_1,...,\widehat{D_i},...,D_r\big)$. 

The class of an invariant divisor $D_i$ on $X$ does not belong to $Amp^1(X)$ if and only if $[D_i]\not\in \mbox{cone}\big(D_1,...,\widehat{D_i},...,D_r\big)$. This is equivalent to saying that $|mD_i|=\{mD_i\}$ for any $m$ such that $mD_i$ is Cartier. We call such divisor an \textit{exceptional} divisor. One can check that a divisor $D_i$ is exceptional if and only if $\mbox{cone}(v_1,...,\widehat{v_i},...,v_r)=\mathbb{R}^n$.
 
We define the $r$-dimensional closed convex polyhedral cone:
$$
\mathcal{P}_{\mathcal{H}}:= \mathcal{D}^{-1} \Big( Amp^1(X)\Big)  \ \subset \ \mathcal{PP}_{\mathcal{H}} \ \subset \ \mathbb{R}^r.
$$

Next we show that $\mathcal{PP}_{\mathcal{H}}\ /\sim$ can be identified with $\mathcal{P}_{\mathcal{H}}$.

\begin{pr}\label{prop:P_H<->amp^1} Let the assumptions and notation be as in \ref{escolha} above.
Then the correspondence $a\mapsto P_a$ induces a bijection between 
$\mathcal{P}_{\mathcal{H}}$ and the set of polytopes $\big\{P_a \ \big| \ a\in \mathcal{PP}_{\mathcal{H}}\big\}$.
\end{pr}
\begin{proof}
Let $a\in\mathcal{PP}_{\mathcal{H}}\cap\mathbb{Z}^n$, and set $D_a:=\displaystyle\sum_{i=1}^{r}a_iD_i$. Consider the set $$I_a:=\{v_i\in\Sigma_X(1)\mid \langle u,v_i\rangle > -a_i \mbox{ for every }u\in P_a\}.$$
It follows from Proposition \ref{p.global} that the divisor $D_i$ appears in the stable base locus of $\mathcal{D}_a$ if and only if $v_i\in I_a$.

We claim that there is an effective divisor  $D_b:=\displaystyle\sum_{i=1}^{r}b_iD_i\in Amp^{1}(X)$ such that $P_a=P_b$. 
For each $i\in I_a$ set $d_i:=\displaystyle min_{u\in P_a} \langle u,v_i\rangle +a_i$.\\
 We take $D_b:=\displaystyle\sum_{i\notin I_a}a_iD_i+\displaystyle\sum_{i\in I_a}(a_i-d_i)D_i\in Amp^1(X)$. 
To see the equality  $P_a=P_b$ consider $i\in I_a$ and $u\in P_a$. Since $\langle u,v_i\rangle +a_i\geq d_i$ we have $u\in P_b$. The converse is obvious since $d_i\geq 0$. So, we are reduced to proving that the correspondence $a\mapsto P_a$ is injective on $\mathcal{P}_{\mathcal{H}}$. But this is obvious since $D_a\in Amp^1(X)$ if and only if $I_a=\emptyset$. 
\end{proof}

\begin{pr}\label{compativel} Let $X$ be a complete $\mathbb{Q}$-factorial toric variety. Then, the cone $Amp^1(X)$ is compatible with the $GKZ$ decomposition of $X$. More precisely, $Amp^1(X)=\displaystyle\bigcup_ {(\Delta,\emptyset)\in\mathcal{S}}GKZ(\Delta,\emptyset)$. In particular, every maximal cone of $GKZ(X)$ intersects $Amp^1(X)$ non trivially. If $X$ is projective then $Amp^1(X)$ is a $\rho_X$-dimensional cone containing $Nef(X)$.

\end{pr}

\begin{proof} Let $D$ be a divisor in $Amp^1(X)$ and let $\Delta$ be the fan defined by $P_D$.
Let $m$ be a positive integer such that $mD$ is Cartier.
The complete linear system $\big| mD\big|$
defines a rational map $f:X\dashrightarrow X_{\Delta}$, which extends to a morphism
on an open subset whose complement has codimension $\geq 2$ in $X$.
Therefore one can define a pull-back map $f^*:N^1(X_{\Delta})\to N^1(X)$.
By construction, there is an ample divisor $A$ on $X_{\Delta}$, and an effective divisor 
$E$ on $X$ such that $mD= f^*A+E$.
The support of $E$ consists of the $T$-invariant prime divisors 
on the base locus of $\big| mD\big|$. Since $D\in Amp^ 1(X)$ we have $E=0$ and $mD=f^ *A$. Thus, $D\in GKZ(\Delta,\emptyset)$. 

Conversely, given $(\Delta,\emptyset)\in\mathcal{S}$, $f:X\dashrightarrow X_{\Delta}$ the associated rational map and $A$ an ample divisor on $X_{\Delta}$, then $D:=f^*(A)$ belongs to the interior of $GKZ(\Delta,\emptyset)$. Thus, the set $I=\{v_i\mid \langle u,v_i\rangle >-a_i \textit{ for every }u\in P_D\}$ is empty. Using the same argument as in the proof of Proposition \ref{prop:P_H<->amp^1} we conclude that $D\in Amp^1(X)$. If $GKZ(\Delta,I)$ is a maximal cone then it follows from the first propriety of $GKZ$ decomposition that $GKZ(\Delta,\emptyset)$ is a non trivial face of $GKZ(\Delta,I)$. For the last claim notice that in that case $Nef(X)=GKZ(\Sigma_X,\emptyset)$.
\end{proof} 

\vspace{0.5cm}

Let $\mathcal{H}=\big(v_1,...,v_r\big)$ and $X$ be as in \ref{escolha}.
In Section \ref{gkzdecomposition}, we described 
the $GKZ$ decomposition of $X$. 
We saw that each maximal cone in this decomposition is generated
by the nef cone of a $\mathbb{Q}$-factorial birational model of $X$ and 
some classes of exceptional invariant divisors $[D_i]$'s.  This decomposition induces, through the map $\mathcal{D}$, a degenerate fan supported on $\mathcal{PP}_{\mathcal{H}}\subset\mathbb{R}^r$ which we call the \emph{GKZ-decomposition} of $\mathcal{PP}_{\mathcal{H}}$ and denote by $GKZ(\mathcal{PP}_{\mathcal{H}})$. 

Polytopes associated to elements in the relative interior of the same cone of this fan 
are strictly combinatorially equivalent. 
This is a different construction of a cell decomposition of $\mathcal{PP}_{\mathcal{H}}$  that 
has already been considered in the literature.
See for instance \cite[Section1.3]{luca}. 
There, the maximal cones of this decomposition are called \emph{
big cells},
and it is shown that the volume of the polytope $P_b$ is a polynomial function of $b$
on each big cell.

It follows from the properties of the $GKZ$ decomposition of $X$ that given a point $a\in\mathcal{PP}_{\mathcal{H}}$ there is a linear subspace $S\ni a$ with maximal dimension among those satisfying the following property: there is an open neighbourhood $U$ of $a$ such that for each $b\in S\cap U$ the polytopes $P_a$ and $P_b$ are strictly combinatorially equivalent. The closure of the set consisting of points $b\in S$ such that $P_b$ is strictly combinatorially equivalent to $P_a$ defines a cone $GKZ(a,\mathcal{PP}_{\mathcal{H}})$ of $GKZ(\mathcal{PP}_{\mathcal{H}})$ whose dimension is dim $S$. This means that
$$GKZ(\mathcal{PP}_{\mathcal{H}})=\Big\{ GKZ(a,\mathcal{PP}_{\mathcal{H}}) \ \big | \  a\in\mathcal{PP}_{\mathcal{H}}\Big\},$$
and
 $$\mathcal{PP}_{\mathcal{H}}=\displaystyle\bigcup_{a\in\mathcal{PP}_{\mathcal{H}}}GKZ(a,\mathcal{PP}_{\mathcal{H}}).$$

 Moreover, for each point $a$ in the interior of a maximal cone of $GKZ(\mathcal{PP}_{\mathcal{H}})$, the polytope $P_a$ is $n$-dimensional and simple. It follows from Proposition \ref{compativel}
that the GKZ-decomposition of $\mathcal{PP}_{\mathcal{H}}$ is compatible with $\mathcal{P}_{\mathcal{H}}$. More precisely,
$$\mathcal{P}_{\mathcal{H}}=\displaystyle\bigcup_{a}GKZ(a,\mathcal{PP}_{\mathcal{H}}),$$
\noindent where the union is over all $a\in\mathcal{PP}_{\mathcal{H}}$ such that $P_a$ is a simple polytope having exactly $r$ facets.

 Note that $\mathcal{D}(a)=\mathcal{D}\big(a+\big(\langle u,v_i\rangle,\cdots ,\langle u,v_r\rangle\big)\big)$ for every $u\in\mathbb{Z}^n$. Therefore, the fan $GKZ(\mathcal{PP}_{\mathcal{H}})$ is degenerate. Each cone of this fan contains the $n$-dimensional linear subspace defined as the image of the linear map $\mathbb{R}^n\rightarrow\mathbb{R}$ given by the matrix whose rows are the vectors $v_1,...,v_r$. \\

\begin{rk}\label{obs} \em{ It follows from the properties of the GKZ decomposition of $X$ that $GKZ(a,\mathcal{PP}_{\mathcal{H}})$ is a maximal cone if and only if $P_a$ is a simple $n$-dimensional polytope with $r$ facets.}
\end{rk} 
 
 Proposition \ref{gkz} can be rewritten as follows. 
Consider $GKZ(a,\mathcal{PP}_{\mathcal{H}})$ and $GKZ(b,\mathcal{PP}_{\mathcal{H}})$  maximal cones whose intersection is a wall $GKZ(c,\mathcal{PP}_{\mathcal{H}})$. One of the following occurs:
 \begin{enumerate}
 \item  $P_b$ has one less facet than $P_a$, $P_c$ is strictly combinatorially equivalent to $P_b$ and there is a divisorial contraction $f:X_{P_a}\rightarrow X_{P_b}$.
 \item $P_a$ has one less facet than $P_b$, $P_c$ is strictly combinatorially equivalent to $P_a$  and there is a divisorial contraction $f:X_{P_b}\rightarrow X_{P_a}$.
 \item The polytope $P_c$ is $n$-dimensional but not simple. There is a small contraction $f:X_{P_a}\rightarrow X_{P_c}$ with $f^+:X_{P_b}\rightarrow X_{P_c}$ the associated flip. The polytope $P_c$ has one less $m$-dimensional face than $P_a$, where $m$  denotes the dimension of the exceptional locus of $f$.
 \end{enumerate}
 
 In all cases it follows from Theorem \ref{mori} that the lost face (i.e., the face corresponding to $Exp(f)$) is a Cayley-Mori polytope.
 
If $GKZ(a,\mathcal{PP}_{\mathcal{H}})$ is a maximal cone and $GKZ(c,\mathcal{PP}_{\mathcal{H}})$ is an exterior wall of $GKZ(a,\mathcal{PP}_{\mathcal{H}})$, then dim $P_c<$ dim $P_a$ and there is a Mori fiber space $f:X_{P_a}\rightarrow X_{P_c}$. In this case, by Theorem \ref{mori} $P_a$ is a Cayley-Mori polytope.

\begin{rk}\label{rk3.3} \em{ In \ref{escolha} we made a choice of $a_0\in \mathbb{Z}^r$ such that 
$P_{a_0}$ is an $n$-dimensional simple lattice polytope having exactly
$r$ distinct facets. 
Suppose we choose another element $a_1\in \mathbb{Z}^n$ such that 
$P_{a_1}$ is an $n$-dimensional simple lattice polytope having exactly
$r$ distinct facets.
Let $(X',L')$ be the polarized toric variety associated to $P_{a_1}$. Then
$X$ and $X'$ are $T$-isomorphic in codimension one.
So there is a natural identification $\varphi:N^1(X)\to N^1(X')$ 
sending the class of a $T$-invariant prime divisor on $X$ to the class of 
the corresponding  $T$-invariant prime divisors on $X'$.
In particular, $\varphi$ identifies the cones $Amp^1(X)$ and $Eff(X)$ with 
$Amp^1(X')$ and $Eff(X')$, respectively.
Moreover, $\varphi$ preserves the GKZ decomposition on $Eff(X)$ and $Eff(X')$.}
\end{rk}

\section{Adjoint Polytopes}
In this section we will explain the MMP in terms of polytope geometry. Given a polarized toric variety $(X,L)$ we will define an operation on the associated polytope that, under some assumptions, describes each
step of the MMP with scaling for $(X,L)$.

\begin{de}\label{def4}
Let $P\subset\mathbb{R}^n$ be an $n$-dimensional polytope, defined by the facet presentation:
$$P:=\{u\in \mathbb{R}^n\mid \langle u,v_i\rangle\geq -a_i, 1\leq i\leq r\},$$
where each $v_i\in\mathbb{Z}^n$ is primitive and $a_i\in\mathbb{R}$ for $1\leq i\leq r$. 
The polarized toric variety corresponding to $P$ is denoted by $(X,L)$, where $L$ is an ample $\mathbb{R}$-divisor on $X$. 
For each $s\in\mathbb{R}_{\geq 0}$, define:

$$P^{(s)}=\{u\in \mathbb{R}^n\mid \langle u,v_i\rangle\geq -a_i+s, 1\leq i\leq r\}.$$

This operation was introduced in \cite{alicia} and the polytopes $P^{(s)}$ are called adjoint polytopes in \cite{adjunction}. 
\end{de}

\begin{rk} \em{ We note that in order to define the adjoint polytopes $P^{(s)}$ we need $P$ to be given by it facet presentation. This means that $P$ has exactly $r$ facets. Otherwise $P^{(s)}$ would not be well defined, i.e., it could depend on the presentation.} 
\end{rk}
\begin{ex} \em{ Consider the polytope $P$ defined by the following facet presentation :

\begin{figure}[!ht]
\begin{minipage}[b]{0.45\linewidth}
$$
\left\{
\begin{array}{lll}
\langle (501,-1000),x\rangle &\geq &-1000 \\
\langle (-1000,501),x \rangle &\geq &-499  \\
\langle (0,1),x\rangle &\geq & 1  \\
\end{array}
\right.
$$
\vspace*{0.4cm}
\end{minipage}
\begin{minipage}[b]{0.45\linewidth}
\begin{center}
\includegraphics[scale=0.18]{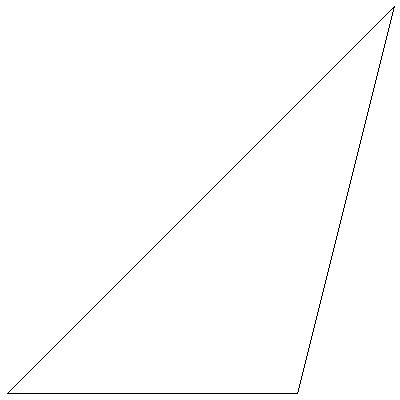}
\end{center}
\end{minipage}
\end{figure}
The polytopes $P^{(s)}$ have the same combinatorial type as $P$ until the moment in which the dimension of $P^{(s)}$ drops. In particular, $P^{(\frac{2}{5})}$ is still a triangle. On the other hand, if we add the inequality $\langle (0,-1),x\rangle\geq -2$, we still have the same polytope $P$. However, $P^{(\frac{2}{5})}$ is no longer a triangle.   

\begin{figure}[!ht]
\begin{minipage}[b]{0.45\linewidth}
$$
\left\{
\begin{array}{lll}
\langle (501,-1000),x\rangle &\geq &-1000 \\
\langle (-1000,501),x \rangle &\geq &-499  \\
\langle (0,1),x\rangle &\geq & 1  \\
\langle (0,-1),x\rangle &\geq & -2
\end{array}
\right.
$$

\vspace*{0.3cm}
\end{minipage}
\begin{minipage}[b]{0.45\linewidth}
\begin{center}
\includegraphics[scale=0.18]{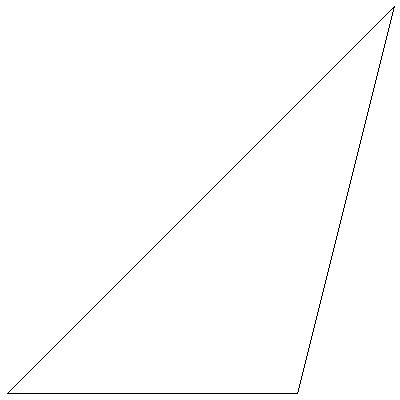}
\end{center}
\end{minipage}
\end{figure}

\newpage

\begin{figure}[!hnt]
\begin{minipage}[b]{0.45\linewidth}
$$
\left\{
\begin{array}{lll}
\langle (501,-1000),x\rangle &\geq &-1000+\displaystyle\frac{2}{5} \\
\langle (-1000,501),x \rangle &\geq &-499+\displaystyle\frac{2}{5}  \\
\langle (0,1),x\rangle &\geq & 1+\displaystyle\frac{2}{5}  \\
\langle (0,-1),x\rangle &\geq & -2+\displaystyle\frac{2}{5}
\end{array}
\right.
$$
\vspace*{0.3cm}
\end{minipage}
\begin{minipage}[b]{0.45\linewidth}
\hspace*{2cm}
\includegraphics[scale=0.16]{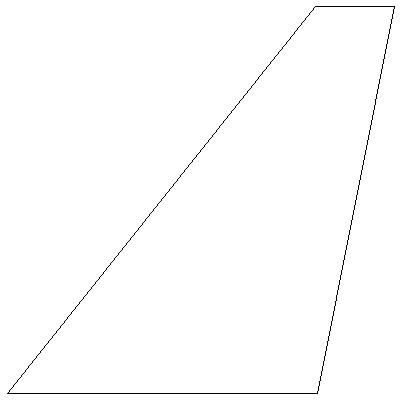}
\vspace{0.7cm}
\end{minipage}
\end{figure}

}
\end{ex}
\vspace{0,5cm}
Suppose that $P$ is given by it facet presentation as in Definition \ref{def4}. By \cite[Proposition 1.4]{adjunction}, $(P^{(s)})^{(t)}=P^{(s+t)}$.\\

The nef threshold of $P$ is the number $$\lambda(P):=sup\{s\in\mathbb{R}_{\geq 0}\mid P \mbox{ and } P^{(s)} \mbox{ have the same normal fan}\}$$ while the effective threshold of $P$ is the number $$\sigma(P):=sup\{s\in \mathbb{R}_{\geq 0} \mid P^{(s)}\neq\emptyset\}.$$ Equivalently, $\sigma(P)=\sigma(L)=sup\{s\in \mathbb{R}_{\geq 0} \mid [L+sK_X]\in Eff(X) \}$ and $\lambda(P)=\lambda(L)=sup\{s\in\mathbb{R}_{\geq 0}\mid [L+sK_X]\in Nef(X)\}$.
The polytope $P^{(\sigma(P))}$ is called the $core$ of $P$ and is denoted by $core(P)$.\\

By Lemma~\ref{lemma:PP_H<->Eff}, $\dim P^{(s)}=n$ for $0\leq s<\sigma(P)$, and 
$\dim P^{(\sigma(P))}<n$. \\

\begin{rk} \em{ By \cite[Lemma 1.13]{adjunction} if $P$ is a rational polytope, then the number $\lambda(P):=sup\{s\in\mathbb{R}_{\geq 0}\mid [L+sK_X]\in Nef(X)\}$ is positive if and only if $X$ is a $\mathbb{Q}$-Gorenstein variety. This happens for instance when $P$ is a simple polytope.}
\end{rk}

Suppose that $P=P_a$ is simple and set $\mathcal{H}=\big(v_1, \cdots, v_r\big)$. Consider the linear map:

$$
\begin{array}{cccl}
 \gamma: \  & [\ 0\ ,\ \sigma(P)\ ] & \rightarrow & \mathcal{PP}_{\mathcal{H}} \\
 & s & \mapsto & a-s\cdot \big(1,...,1\big) . \\
\end{array}
$$
In order to relate adjoint polytopes with birational geometry, we will compose
$\gamma$ with the map $\mathcal{D}: \mathbb{R}^r  \to N^1(X)$ defined in Section~\ref{Section:P_H}.
Recall that $\mathcal{D}(1, \cdots, 1) =  \big[\sum_{_{i=1}}^{^r}D_i\big] = [-K_X]$, 
which lies in the interior of $Eff(X)$.
Set $L=\sum_{_{i=1}}^{^r}a_iD_i$. Then we have 
$$
\begin{array}{cccl}
\pi=\mathcal{D}\circ\gamma: \  & [\ 0\ ,\ \sigma(P)\ ] & \rightarrow & Eff(X)  \\
 & s & \mapsto & [L+sK_X] . \\
\end{array}
$$

This means that the polytope associated to the divisor $L+sK_X$ is $P^{(s)}$.

 When we increase $s$ the adjoint polytope $P^{(s)}$ will change its combinatorial type at some critical values, the first one being $\lambda(P)$. Our aim is to describe the family of polytopes $P^{(s)}$ for values of $s$ between 0 and $\sigma(P)$.
Theorem \ref{MORI} below describes $P^{(s)}$ as we vary $s$ from 0 to $\sigma(P)$ under the assumption that $P$ is a \textit{general}
simple rational polytope. In this case we can say precisely what happens with the adjoint polytope $P^{(s)}$ each time reaches one of finitely many critical values. Moreover, we show that for $\varepsilon >0$ sufficiently small, $P^{(\sigma(P)-\varepsilon)}$ is a Cayley-Mori polytope.\\

We follow the notation introduced in Section \ref{gkzdecomposition}

\begin{lema} \label{lm} Let $(\Sigma,I)\in\mathcal{S}$ and let $f:X\dashrightarrow X_{\Sigma}$ be the associated rational map. If $A$ is a nef effective $\mathbb{R}$-$T$-divisor on $X_{\Sigma}$ and $\lambda=(\lambda_i)_i$ is in $\mathbb{R}_{\geq 0}^{|I|}$ then $P_A=P_D$ where $D=f^*(A)+\displaystyle\sum_{i\in I}\lambda_iD_i$.
\end{lema}

\begin{proof} It follows easily from \cite[Proposition 6.2.7]{cox}.
\end{proof}

\begin{de}\label{general} Let $P$ be a simple polytope and set $\mathcal{H}=\big(v_1,...,v_r\big)$, where $\{v_1,...,v_r\}$ are the primitive vectors of $\Sigma_P$.
Let the notations be as in \ref{escolha}. We say that the polytope $P$ is general if $P=P_b$ for some  
$$b\in \mathcal{D}^ {-1}\Big(Eff(X)\setminus\displaystyle\bigcup_{\Gamma}(\mbox{cone}(\Gamma,[-K_X]))\Big),$$
\noindent where $\Gamma$ runs over all $(\rho_X-2)$-dimensional cones in $GKZ(X)$ (see Remarks \ref{obs} and \ref{rk3.3}).
\end{de}

\begin{teo} \label{MORI} Let $(X,L)$ be a polarized $n$-dimensional $\mathbb{Q}$-factorial toric variety such that $P:=P_L\subset\mathbb{R}^n$ is a general rational  polytope. Then there exist sequences 
$$0=\lambda_0<\lambda_1<...<\lambda_k=\sigma(P), \ \ \ \ \ X=X_1 \stackrel{f_1}{\dashrightarrow} X_2\stackrel{f_2}{\dashrightarrow} ...\stackrel{f_k}{\dashrightarrow} X_{k+1}$$ 

of rational numbers and rational maps, such that: 
\begin{enumerate}
\item For $i\in\{1,...,k-1\}$, $f_i$ is either a divisorial contraction or a flip. 
\item For $\lambda_i <s,t<\lambda_{i+1}, \ P^{(s)}$ and $P^{(t)}$ are $n$-dimensional simple polytopes with the same normal fan. 
\item At $s=\lambda_i$, one of the following occurs. 
\begin{enumerate}

\item Either $P^{(\lambda_i)}$ is simple and $P^{(\lambda_i)}$ has one less facet than $P^{(\lambda_{i-1})}$ (equivalently, $f_i$ is a contraction of divisorial type), or 
\item $P^{(\lambda_i)}$ is not simple and $P^{(\lambda_i)}$ has the same number of facets as $P^{(\lambda_{i-1})}$ (equivalently, $f_i$ is a flip). 
\end{enumerate}
\item For $\lambda_{k-1}<s<\lambda_k=\sigma(P)$, $P^{(s)}$ is a Cayley-Mori polytope, (equivalently $f_k$ is a Mori fiber space, and $X_{k+1}$ is the toric variety associated to $P^{(\sigma(P))}$). 

\item For each $i\in\{1,...,k\}$, denote by $m_i$ the dimension of the locus where $f_i$ is not an isomorphism. Then, for $\lambda_i<s<\lambda_{i+1}$, the polytope $P^{(s)}$ has exactly one more $m_i$-dimensional face than $P^{(\lambda_{i+1})}$, and this face is a Cayley-Mori polytope.  

\item Let $K(P)$ be the linear space parallel to $Aff(Core(P))$ and consider the natural projection $\pi_P:\mathbb{R}^n\rightarrow \mathbb{R}^n/K(P)$ associated to $P$. The toric variety associated to the polytope $Q:=\pi_P(P)$ is the closure of the general fiber of the rational map $f:=f_k\circ...\circ f_1:X\dashrightarrow X_{k+1}$.
\end{enumerate}
Moreover, if $P^{(\lambda_i)}$ is simple then $X_{i+1}$ is the toric variety associated to it.
  Otherwise, if $P^{(\lambda_i)}$ is not simple, the toric variety associated to $P^{(\lambda_i)}$ is the image of the small contraction corresponding to the flip $f_i$ and $X_{i+1}$ is associated to $P^{(s)}$ for $\lambda_i<s<\lambda_{i+1}$.
\end{teo}
\begin{rk}\em{ The polytope $Q$ of item 6 was introduced in \cite{adjunction} but lacked this geometric interpretation.}
\end{rk}
 \begin{proof} Set $\lambda_1=\lambda(L,X)$. By definition, $L+sK_X\in Amp(X)$ for $0\leq s<\lambda_1$.  Then $P^{(s)}=P_{L+sK_X}$ has the same normal fan for $0\leq s<\lambda_1$. Since $P$ is general, $L+\lambda_1K_X$ belongs  to the interior of a wall $\tau$ of $Nef(X)$.  Since the nef cone is dual to the Mori cone, there is a unique extremal ray $R_1$ of $\mbox{NE}(X)$ such that $(L+\lambda_1K_X)\cdot R_1=0$. Since $L$ is ample, $K_X\cdot R_1<0$.  
 
 By Theorem \ref{contracao}, the complete linear system $|m(L+\lambda_1K_X)|$ defines the contraction $f:X\rightarrow Y$ of the extremal ray  $R$ for $m$ sufficiently large and divisible. Moreover, $Y$ is the projective toric variety defined by $P^{(\lambda_1)}$ and $\tau=f^*(Nef(Y))$.

We fall under one of the following 3 possibilities.
\begin{enumerate}
\item If dim $Y<$ dim $X$ then $Y$ is $\mathbb{Q}$-factorial. By Proposition \ref{gkz},  the cone $\tau=~GKZ(\Delta_{L+\lambda_1K_X},\emptyset)$ is an exterior wall of $Eff(X)$. Thus $\lambda_1=\sigma(P)$, $k=1$ and $f=f_1$.
It follows from Theorem \ref{mori} that $P^{(s)}$ is a Cayley-Mori polytope for $0\leq s<\lambda_1$.

\item If  $f:X\rightarrow Y$ is a divisorial contraction then $Y$ is also $\mathbb{Q}$-factorial. In this case, using again Proposition \ref{gkz} we conclude that $\tau$ is an interior wall and therefore $\lambda_1<\sigma(P)$.  We set $X_2=Y$ and $f_1=f$. Since $X_2$ is the toric variety associated to the polytope $P_{L+\lambda_1K_X}=P^{(\lambda_1)}$ we have that $P^{(\lambda_1)}$ has one less facet than $P$. \\

Now, consider $\lambda_2=\lambda((f_1)_*L,X_2)$. Then $(f_1)_*L+\lambda_2K_{X_2}\in\partial Nef(X_2)$. 

Notice that $$L+\lambda_1K_X=f_1^*((f_1)_*L+\lambda_1K_{X_2}).$$
This implies that $(f_1)_*L+\lambda_1K_{X_2}$ is nef since so is $L+\lambda_1K_X$. By Lemma \ref{lm} the divisor $(f_1)_*L+\lambda_1K_{X_2}$ determines the same variety as $L+\lambda_1K_X$, which is $X_2$. Thus $(f_1)_*L+\lambda_1K_{X_2}$ is ample and $\lambda_1<\lambda_2$. Moreover, 
$(f_1)_*L+sK_{X_2}\in Amp(X_2)$ for $\lambda_1\leq s<\lambda_2$.  

From Lemma \ref{lm} we conclude that for $\lambda_1\leq s<\lambda_2$, the varieties defined by $(f_1)_*L+sK_{X_2}$ and $L+sK_X$ coincide and are equal to $X_2$. Consequently, for $s$ in this interval, $L+sK_X$ is in a same cone of  $GKZ(X)$, which means that  the polytopes $P^{(s)}$ are $n$-dimensional, simple, and they have the same normal fan.\\

Since the product of the exceptional divisor $E$ of $f_1$ by the ray $R_1$ is negative we have that, for every $\lambda_1< s\leq\lambda_2$, there exists $a_s< 0$ such that $f_1^*((f_1)_*L+sK_{X_2})=L+sK_X+a_sE$. 

Since $(f_1)_*L+\lambda_2K_{X_2}\in\partial Nef(X_2)$, the variety associated to $L+\lambda_2K_X$ (which is the same as the one defined by  $(f_1)_*L+\lambda_2K_{X_2}$) is not $X_2$. So, since $P$ is general, $L+\lambda_2K_X$ is in the interior of the wall of $GKZ(X)$ which is given as the image of the map $$Nef(X_2)~\times~\mathbb{R}_{\geq 0}\rightarrow N^1(X).$$ It follows that 
$(f_1)_*L+\lambda_2K_{X_2}$ is in the interior of a wall of $Nef(X_2)$. Hence, there exists a unique extremal ray $R_2\in \mbox{NE}(X_2)$ such that the product $((f_1)_*L+\lambda_2K_{X_2})\cdot R_2$ is equal to zero. Then we have again a contraction  $g:X_2\rightarrow Z$ of an extremal ray.

\item If $f:X\rightarrow Y$ is a small contraction, then $Y$ is not $\mathbb{Q}$-factorial. Let $\psi:X\dashrightarrow X^+$ be the associated flip. Then $X^+$ is $\mathbb{Q}$-factorial. We set $X_2=X^+$ and $f_1=\psi$. 

The polytope $P^{(\lambda_1)}$  is not simple because $Y$ is not $\mathbb{Q}$-factorial. Since $f$ is an isomorphism in codimension one it follows from  Lemma \ref{picard} that $P^{(\lambda_1)}$   and $P$ have the same number of facets. 
Consider the flip diagram:

$$\xymatrix{
X\ar[rd]_f \ar@{-->}[rr]^{\psi} & & X^+ \ar[ld]^{f^+} \\ 
& Y &  
}$$ 

By Proposition \ref{gkz}, $GKZ(\Sigma^+,\emptyset)$ is a maximal cone, where $\Sigma^+$ denotes the fan of $X^+$. In addition, $\tau=GKZ(\Sigma,\emptyset)\cap  GKZ(\Sigma^+,\emptyset)$. It follows that for every $\varepsilon>0$ such that $L+(\lambda_1+\varepsilon)K_X\in GKZ(\Sigma^+,\emptyset)$, $P^{(\lambda_1+\varepsilon)}$ is associated to $X_2=X^+$. This means that if we set $\lambda_2=sup\{s\in \mathbb{R}_{\geq 0} \mid L+sK_X\in GKZ(\Sigma^+,\emptyset)\}$ then $P^{(s)}$ has the same normal fan for $\lambda_1<s<\lambda_2$.

 Since $\psi_1$ is an isomorphism in codimension one, the divisors $(\psi_1)_*L+sK_{X^+}$ and $L+sK_X=\psi_1^*((\psi_1)_*L+sK_{X^+})$ determine the same variety. Hence $\lambda_2=\lambda((\psi_1)_*L,X^+)$, and we have that $(\psi_1)_*L+\lambda_2K_{X^+}\in\partial Nef(X^+)$.

As before, $(\psi_1)_*L+\lambda_2K_{X^+}$ must belong to the interior of a wall of $Nef(X_2)$ because $P$ is a general polytope. This wall determines a contraction $g:X_2~\rightarrow Z$ of an extremal ray of $\mbox{NE}(X_2)$.
\end{enumerate}

For the cases 2 and 3, if dim $Z<$ dim $X_2$ then $\lambda_2=\sigma(P)$. We set $X_3=Z$ and $f_2=g$. From Theorem \ref{mori} we have that $P^{(s)}$ is a Cayley-Mori polytope for $\lambda_1< s<\lambda_1$. Now, we conclude the items 1 to 4 by induction. \\

To prove item $5.$ note that the exceptional locus of the extremal contraction $f_{R_i}$ corresponding to $f_i$ is, by \ref{fibrapicard}, an irreducible variety which corresponding the $m_i$-dimensional lost face of $P^{(\lambda_{i+1})}$. Since the restriction of $f_{R_i}$ to $Exp(f_{R_i})$ is a Mori fiber space, item $5$ follows from item $4$. \\

To prove $6.$ notice that the rational map $f:X\dashrightarrow X_{k+1}$ is induced by the exact sequence $$0\longrightarrow K(P)^{\perp}\stackrel{j}{\longrightarrow} (\mathbb{R}^n)^*\longrightarrow (\mathbb{R}^n)^*/K(P)^{\perp}\longrightarrow 0,$$
whose the dual exact sequence is $$0\longrightarrow K(P)\longrightarrow\mathbb{R}^n\stackrel{\pi_P}{\longrightarrow} \mathbb{R}^n/K(P)\longrightarrow 0.$$ 
Given a vertex $\bar{v}$ of $Q$ there is a vertex $v$ of $P$ such that $\pi_P(v)=\bar{v}$. It follows that the image of the cone $C_v:=\mbox{cone}(P-v)$ by $\pi_P$ is the cone $C_{\bar{v}}:=\mbox{cone}(Q-\bar{v})$. Since each maximal cone of the fan $\Sigma_Q$ determined by $Q$ is of the form $(C_{\bar{v}})^\vee$ and $j$ is the dual map to $\pi_P$, we have that $j$ induces an inclusion $X_Q\hookrightarrow X$.

Consider the toric variety $U\subset X$ given by the open subset $\displaystyle\bigcup_{v\in\Sigma_X(1)}U_{v}$ of $X$. The restriction $f\mid_U:U\rightarrow X_{k+1}$ is a morphism of toric varieties. We will see in Section \ref{fibrageral} that the general fiber of this morphism is the toric variety $F\subset U$ given by the fan $\{v\mid v\in K(P)^{\perp}\cap\Sigma_X(1)\}$, that is, $F$ is the open subset of $X_Q$ given by $\displaystyle\bigcup_{\bar{v}\in\Sigma_Q(1)}U_{\bar{v}}$. As $X_Q$ is a closed subset of $X$, we have $\overline{F}^X=X_Q$. 
\end{proof}

\begin{rk} \em{It follows from the proof of Theorem \ref{MORI} and from the description given in Section \ref{mmp} that if $P$ is general then each point of intersection between the segment $\{[L+sK_X]\mid 0\leq s\leq \sigma(P)\}$ and the walls of the $GKZ$ decomposition of $X$ corresponds to a step of the MMP of $X$ with scaling of $L$. 

We can also reformulate the definition of general polytope as follows: $P$ is general if when we vary $s$ from $0$ to $\sigma(P)$, the polytopes $P^{(s)}$ lose exactly one face at each critical value. More precisely, if $\lambda_i$ is a critical value, then all lost faces from $p^{(\lambda_i-\varepsilon)}$ to $P^{(\lambda_i)}$, for $\varepsilon$ sufficiently small, are contained in a unique face.}
\end{rk}

\begin{ex} \em{Consider $X:=Bl_p(\mathbb{P}^1\times\mathbb{P}^1)$ the toric variety associated to the fan defined by the following maximal cones: \\
$\big\{\mbox{cone}\big((1,0),(0,1)\big),$ $\mbox{cone}\big((0,1),(-1,0)\big),$ $\mbox{cone}\big((-1,0),(-1,-1)\big),$\\ $\mbox{cone}\big((-1,-1),(0,-1)\big),$ $\mbox{cone}\big((0,-1),(1,0)\big)\big\}.$

The invariant divisors corresponding to the primitive vectors are denoted by $D_1,D_2,D_3,$ $D_4,E$. The effective cone of $X$ is generated by $D_3,D_4,E$ and we have the relations $D_1\sim D_3+E$ and $D_2\sim D_4+E$. 

Let $L:=2D_1+D_2+2D_3+D_4+\frac{5}{2}E$ be an ample divisor on $X$ and $P$ the corresponding polytope.

\begin{figure}[!ht]
\begin{minipage}[b]{0.45\linewidth}
\vspace{0pt}
\includegraphics[scale=0.35]{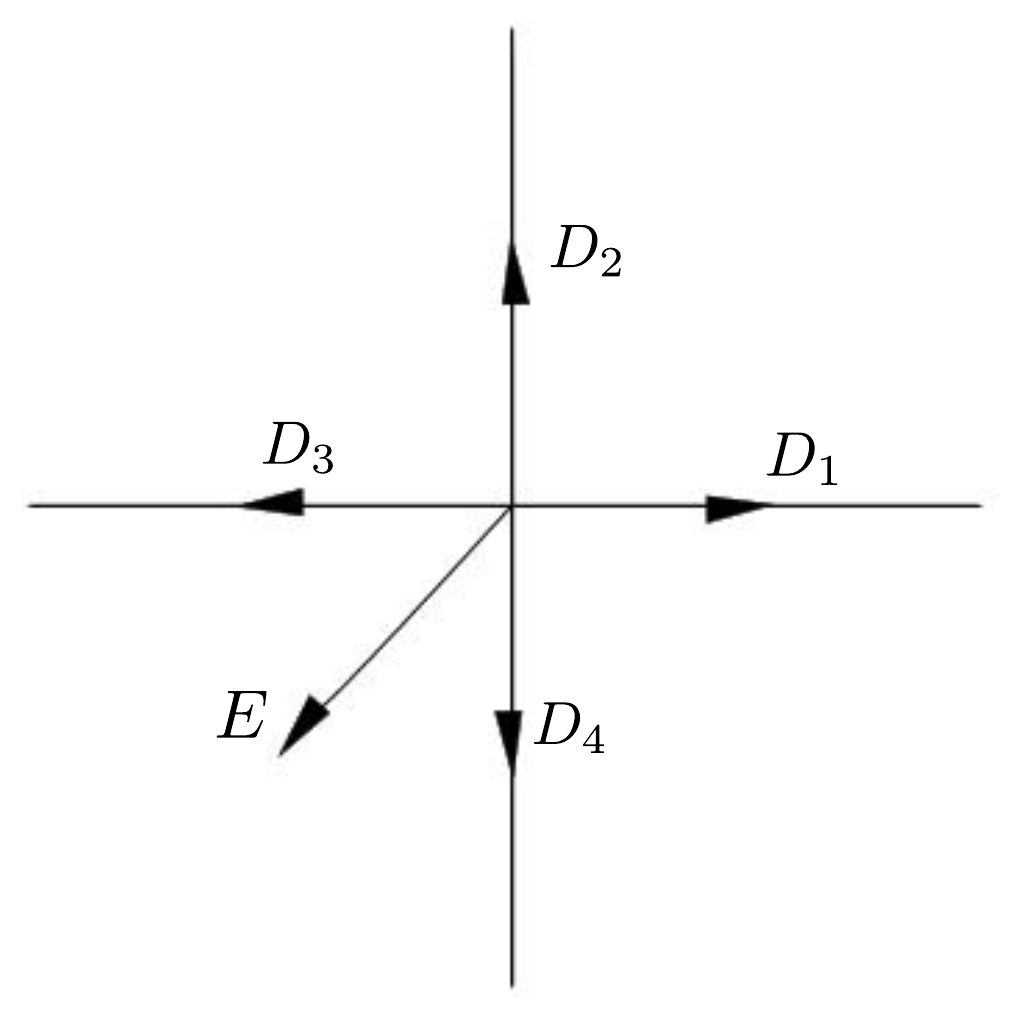} 
\end{minipage}
\begin{minipage}[b]{0.45\linewidth}
\vspace{0cm}
 \includegraphics[scale=0.40]{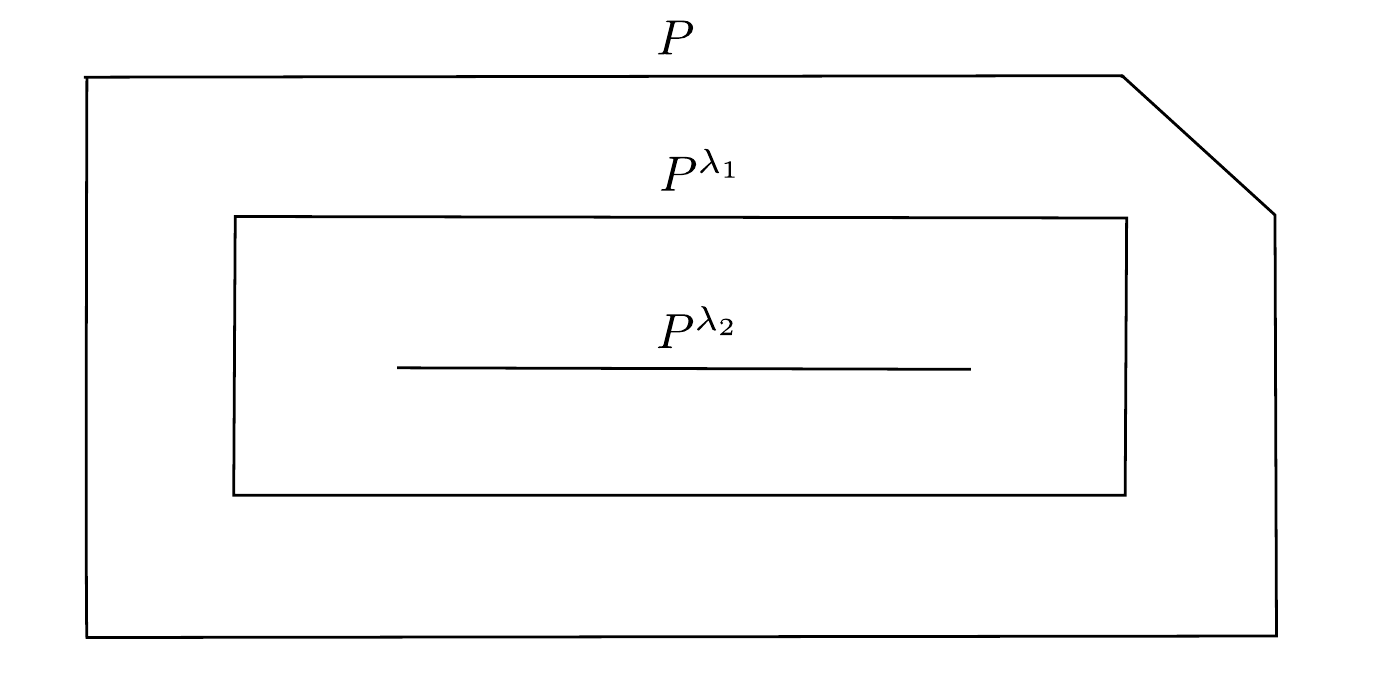}
\end{minipage}
\end{figure}

We have 2 critical values $\lambda_1=1/2$ and $\lambda_2=1$ that correspond to the maps
$$\xymatrix{
Bl_p(\mathbb{P}^1\times\mathbb{P}^1)\ar[r]^{f_1} \ar@{<->}[d] & \mathbb{P}^1\times\mathbb{P}^1  \ar[r]^{f_2} \ar@{<->}[d] & \mathbb{P}^1  \ar@{<->}[d] \\ 
P & P^{(\lambda_1)}  &  P^{(\lambda_2)}
}$$ 

The sequence of polytopes $P,P^{(\lambda_1)}$ and $P^{(\lambda_2)}$ correspond to the sequence of varieties that appear in the MMP of $X$ with scaling of $L$.

The picture that corresponds to this process in the $GKZ$ decomposition of $X$ is given below.
\begin{center}

\includegraphics[scale=0.45]{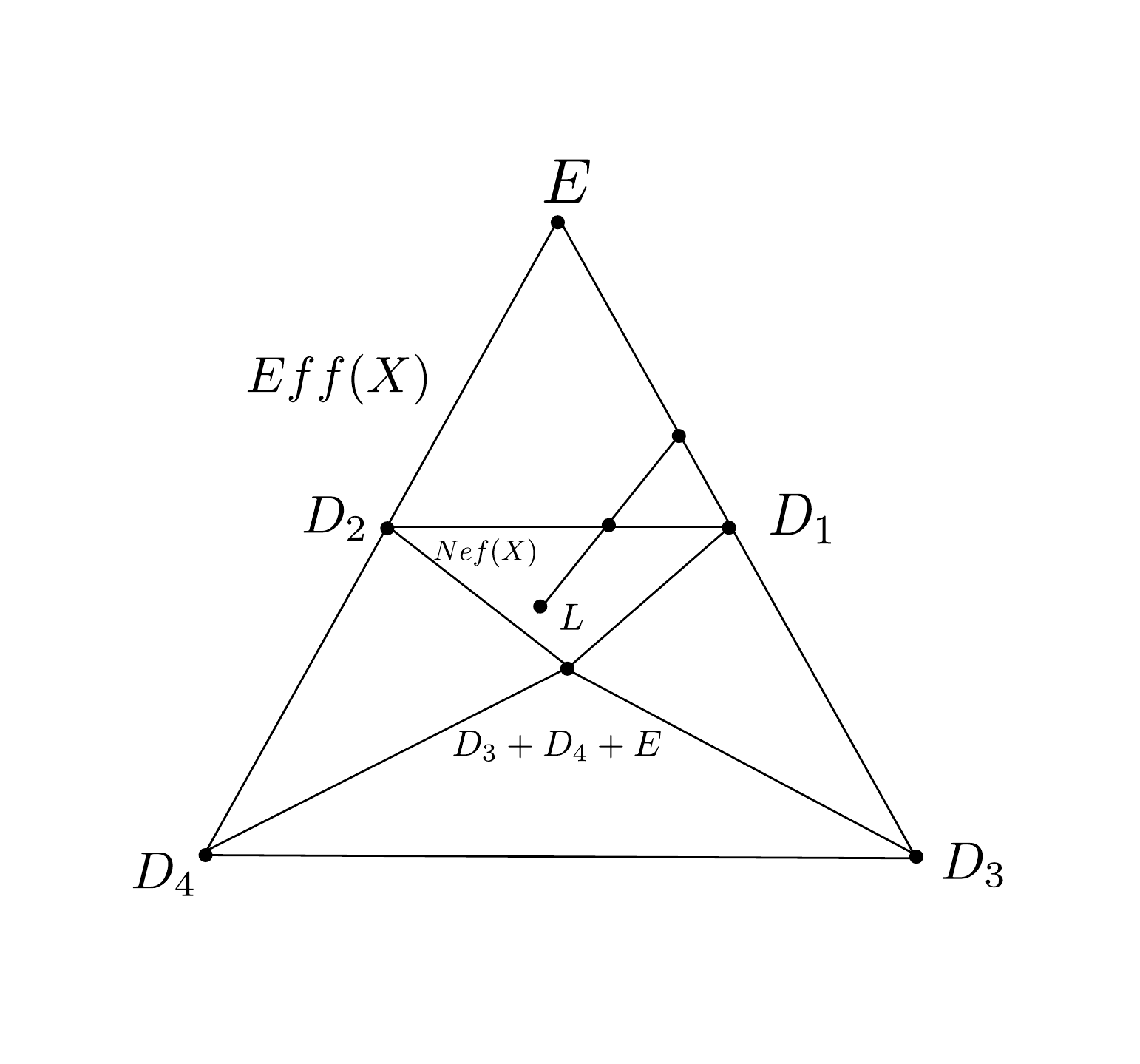}
\end{center}

If we replace $L$ by the ample divisor $6D_1+5D_2+6D_3+5D_4+2E$ then we will have 3 critical values $\lambda_1=1/2$, $\lambda_2=3/2$ and $\lambda_3=5/2$,

\begin{figure}[!ht]
\begin{minipage}[b]{0.45\linewidth}
\includegraphics[width=3.5cm]{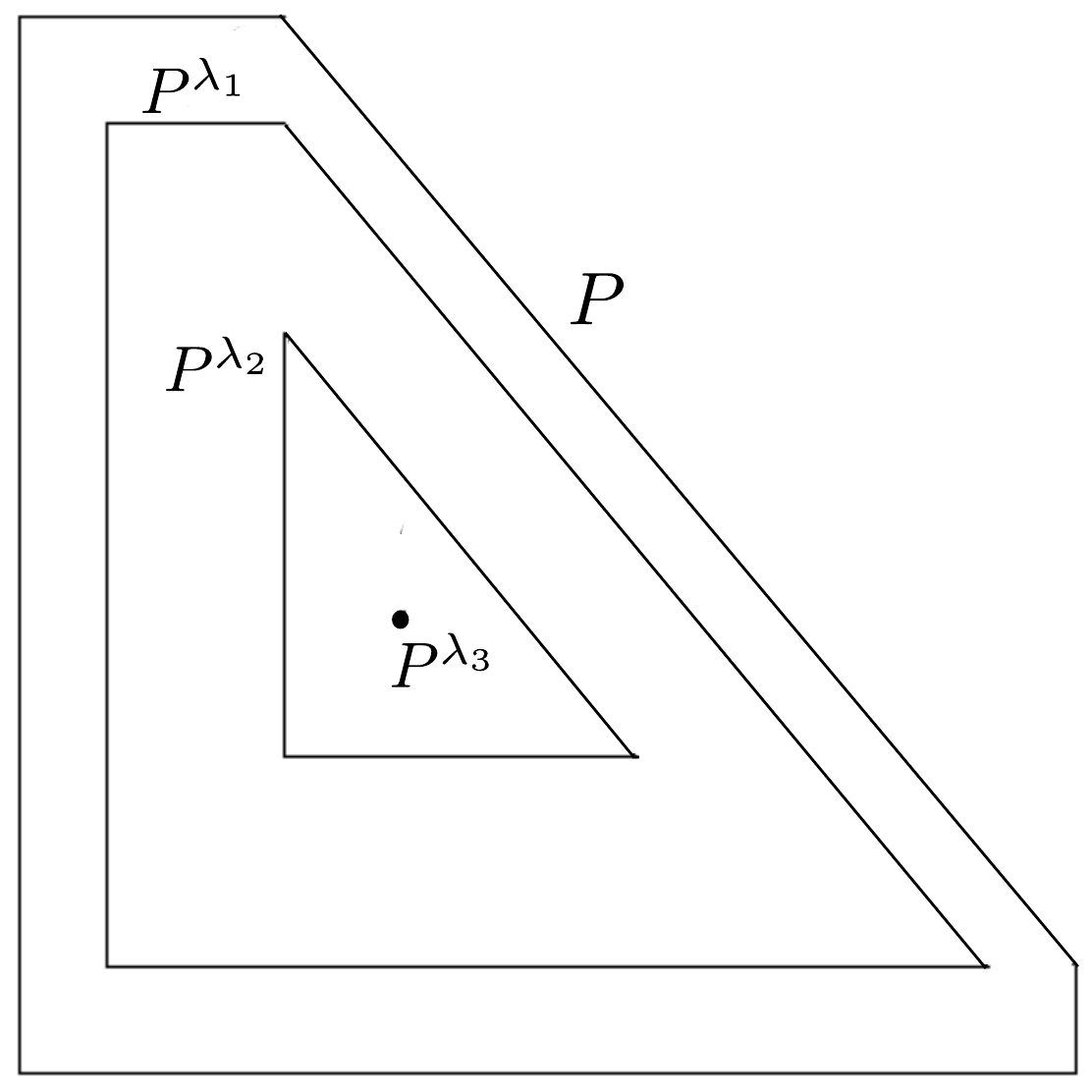}
 \vspace{1cm}
\end{minipage}
\begin{minipage}[b]{0.45\linewidth}
\vspace{0cm}
\includegraphics[scale=0.40]{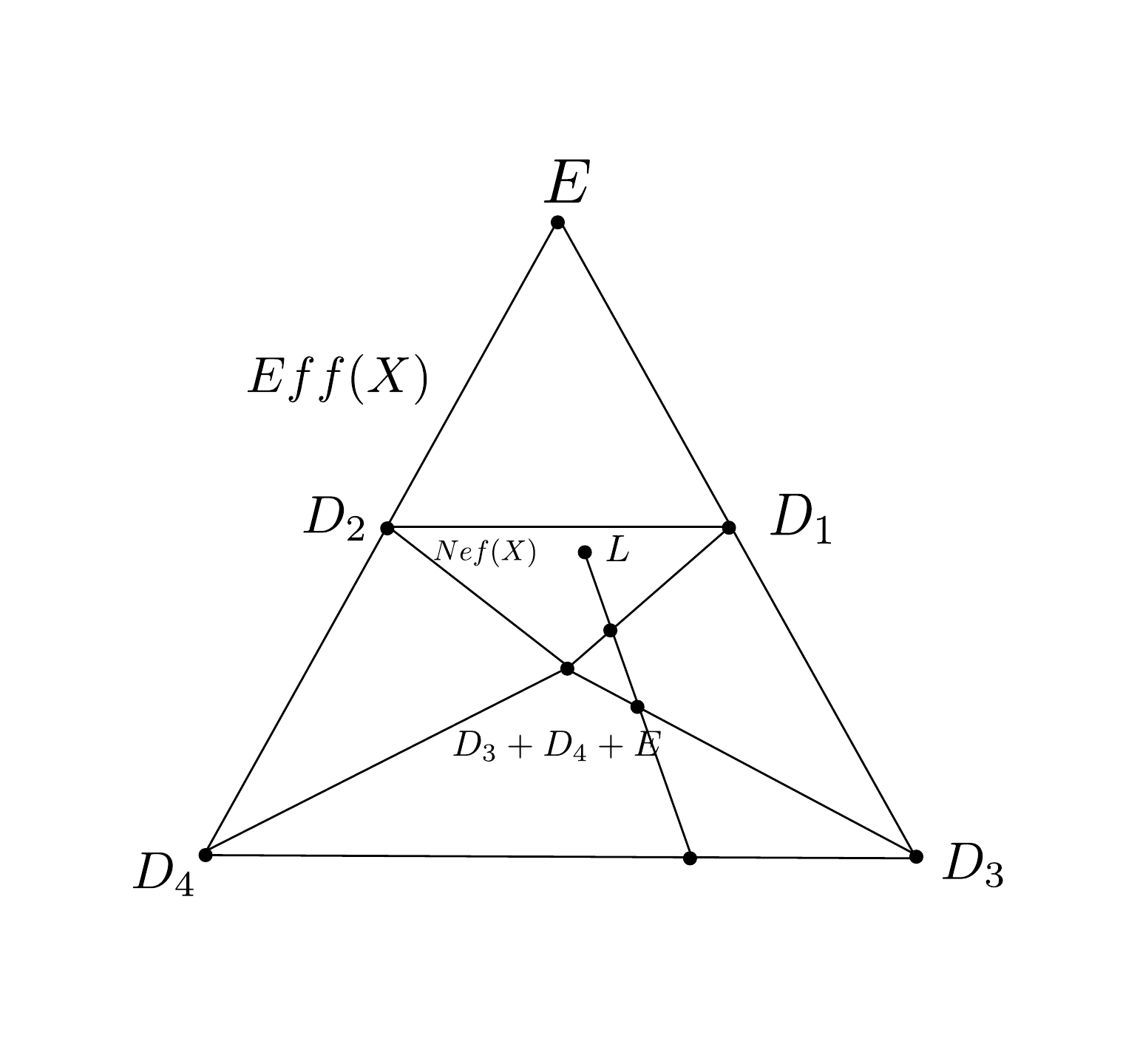}
\end{minipage}
\end{figure}

that correspond to the maps:}
$$\xymatrix{
Bl_p(\mathbb{P}^1\times\mathbb{P}^1)\ar[r]^{f_1} \ar@{<->}[d] & \mathbb{F}_1  \ar[r]^{f_2} \ar@{<->}[d]& \mathbb{P}^2 \ar[r]^{f_3} \ar@{<->}[d]   &  \{q\} \ar@{<->}[d] \\ 
P & P^{(\lambda_1)}  &  P^{(\lambda_2)} & P^{(\lambda_3)}
}$$ 

\end{ex}
 
\begin{rk}\em{ If the polarized toric variety $(X,L)$ corresponds to a non general polytope $P$, then the adjoint polytopes $P^{(s)}$ do not give complete information about the sequence of varieties that appear in the MMP of $X$ with scaling of $L$.}
 \end{rk} 
 
 \newpage
\begin{ex}\em{ Let $(X,L)$ be the polarized toric variety associated to the polytope $P$ defined by the following inequalities:}
\end{ex}
\begin{figure}[hnt!]
\begin{minipage}[b]{0.45\linewidth}
$$\left\{
\begin{array}{lll}
\langle (-1,-1),x\rangle &\geq &-1 \\
\langle (0,-1),x \rangle &\geq &-1  \\
\langle (1,0),x\rangle &\geq & -1  \\
\langle (1,1),x\rangle &\geq & -1  \\
\langle (0,1),x\rangle &\geq & -1 \\
\langle (-1,0),x\rangle &\geq & -1 \\
\end{array}
\right.$$
\vspace*{0.3cm}
\end{minipage}
\begin{minipage}[b]{0.45\linewidth}
\begin{center}
\includegraphics[scale=0.14]{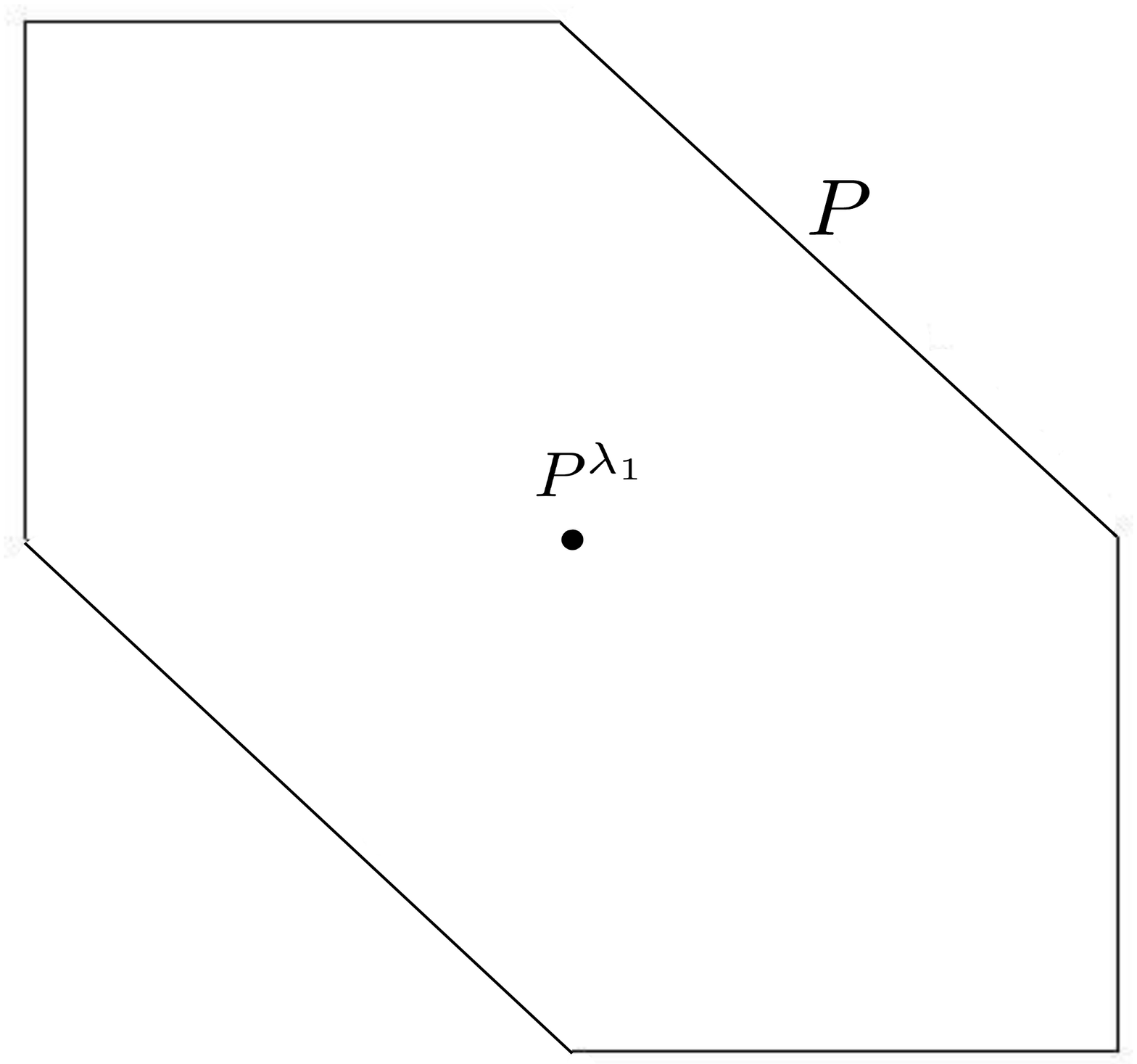}
\end{center}
\end{minipage}
\end{figure}

Notice that $L=-K_X$ and thus $\lambda(P)=\sigma(P)=1$, but the polytope $P$ is not a Cayley-Mori.

\begin{rk} \em{ Let $P:=\{u\in \mathbb{R}^n\mid \langle u,v_i\rangle\geq -a_i, 1\leq i\leq r\}$ be a polytope as in Definition \ref{def4}.
Let $(X,L)$ be the polarized toric variety corresponding to $P$.
Let $\alpha:=\big(\alpha_1,... ,\alpha_r\big)\in (\mathbb{Q}_{\geq 0})^r$.
For each $s\in\mathbb{R}_{\geq 0}$, define:

$$P^{(s\alpha)}=\{u\in \mathbb{R}^n\mid \langle u,v_i\rangle\geq -a_i+s\alpha_i, 1\leq i\leq r\}.$$

Consider the divisor $D:=-\displaystyle\sum_{i=1}^{r}\alpha_iD_i$ on $X$. We can run the analogous steps of MMP replacing $K_X$ with $D$. As in the usual MMP, the divisor $D$ cannot be made nef. So, we still end with a Mori fiber space. Then, all results of this section can be generalized replacing $K_X$ by $D$ and $P^{(s)}$ by $P^{(s\alpha)}$.} 

\end{rk}

\chapter{On the Classification of 2-Fano Toric Varieties}

 A smooth projective variety $X$ is said to be Fano if it has ample anti-canonical divisor. These varieties have been studied by several authors and play an important role in birational algebraic geometry. Fano varieties are quite rare. 
It was proved by Kollár, Miyaoka and Mori that, fixed the dimension, there exist only finitely many smooth Fano varieties up to deformation  (see \cite{kollar}, \cite{kol}). 
Further, in the toric case, there exist finitely many isomorphism classes of them. In dimension 2, smooth Fano varieties are classically known as Del Pezzo surfaces. 
There are five isomorphism classes of  toric Del Pezzo surfaces: $\mathbb{P}^2$, $\mathbb{P}^1 \times \mathbb{P}^1$ and $\mathbb{P}^2$ blown up in 1, 2 or 3 invariant points. 
In dimensions 3, 4, 5 and 6 there are 18, 124, 866 and 7622 isomorphism classes of smooth toric varieties respectively (see \cite{bat3}, \cite{bat}, \cite{nill} and \cite{obro}).\\
 
 We are interested in toric varieties known as 2-Fano varieties. A Fano variety $X$ is said to be 2-Fano if its second Chern character is positive (i.e., $ch_2(T_X)\cdot S >0$ for every surface $S\subset X$). These varieties were introduced by de Jong and Starr in \cite{starr} and \cite{jason} in connection with rationally simply connected varieties, which in turn are linked with the problem of finding rational sections for fibrations over surfaces. 2-Fano varieties are even rarer than Fano varieties. One can check from the classification of Del Pezzo surfaces that the only 2-Fano surface is $\mathbb{P}^2$. In \cite{carol} it is proved that the only 2-Fano threefolds are $\mathbb{P}^3$ and the smooth hyperquadric in $\mathbb{P}^4$. In higher dimensions, few examples are known. First de Jong and Starr gave some examples in \cite{jason}, then Araujo and Castravet found some more examples (see \cite{ana}, Section 5). Among all examples known, the only smooth toric 2-Fano varieties are projective spaces. \\

\textbf{Question 1:} Is $\mathbb{P}^n$ the only n-dimensional smooth projective toric 2-Fano variety? \\

 We will go through the classification of toric Fano 4-folds, given by Batyrev in \cite{bat}, and we will check that the only one with positive second Chern character is $\mathbb{P}^4$. 
This result was published in \cite{eu}. Then, we use a database provided by Øbro in \cite{obro} to answer Question 1 positively in dimension 5 and 6.  
We remark that in  \cite{sato2} Sato considers a similar problem. In particular, he classifies smooth toric Fano varieties with Picard number 2 whose second Chern character is nef (i.e., $ch_2(T_X)\cdot S \geq 0$ for every surface $S\subset X$). We can use Proposition \ref{hi} and Remark \ref{sato} to recover this result. 
 In Section \ref{dimensaomaior} we study the case of dimension bigger than 6 and give some partial results.\\

\section{2-Fano in Low Dimension}
\subsection{Batyrev Classification of Toric Fano 4-folds}

Let $\Sigma\subset N_{\mathbb{R}}\simeq\mathbb{R}^n$ be a simplicial complete fan. Write $\Sigma(1)=\{v_1,...,v_r\}$ and set $X:=X_{\Sigma}$.

\begin{de} A subset $\mathcal{P}=\{v_{i_1},v_{i_2},...,v_{i_d}\}$ of $\Sigma(1)$ is a \textbf{primitive collection} for $\Sigma$ if the following conditions are satisfied:\\
\begin{enumerate}
 \item 
$\mathcal{P}$ is not contained in any cone from $\Sigma$.
\item Any proper subset of $\mathcal{P}$ is contained in some cone from $\Sigma$.
 \end{enumerate}
 \end{de}

\begin{de}  Let $\mathcal{P}=\{v_{i_1},v_{i_2},...,v_{i_d}\}\subseteq\Sigma(1)$ be a primitive collection for $\Sigma$ and $\sigma_{\mathcal{P}}=\langle{v_{j_1},...,v_{j_k}}\rangle$ be the cone of minimal dimension in $\Sigma$ such that $v_{i_1}+v_{i_2}+...+v_{i_d}\in\sigma_{\mathcal{P}}$. Then, there is a unique linear relation $$v_{i_1}+...+v_{i_d}=c_1v_{j_1}+...+c_kv_{j_k}, \ \ \ c_i\in\mathbb{Q}_{>0}.$$
We call  $\mathcal{R(P)}:=v_{i_1}+...+v_{i_d}-c_1v_{j_1}-...-c_kv_{j_k}$ the \textbf{primitive relation} associated to $\mathcal{P}$.
If $X$ is smooth then $c_i\in\mathbb{Z}_{>0}$, and we define the \textbf{degree} of the primitive collection $\mathcal{P}$ by:
$$\Delta(\mathcal{P}):=d-c_1-...-c_k.$$
\end{de}
By Proposition \ref{seq}, we may interpret $N_1(X)$ as the space of linear relations among the minimal generators of $\Sigma$. Under this identification, we have that $N_1(X)$ is generated by primitive relations. Moreover, if $V(\sigma)$ is an invariant curve on $X$ then its class $[V(\sigma)]$ in $N_1(X)$ is a positive linear combination of primitive relations (see \cite{cox}, Theorem 6.3.10).\\

Hence, \  $\mbox{NE}(X)=\displaystyle\sum_{\mathcal{P} \ \substack{ primitive \\  collection}}\mathbb{R}_{\geq0}\mathcal{R(P)}$.

Note that a relation $\displaystyle\sum_{i=1}^{r}a_iv_i=0, a_i \in \mathbb{R}$, corresponds to an element $\xi\in N_1(X)$ that has intersection $a_i$ with $V(\langle v_i \rangle)$ for all $i\in \{1,...,r\}$.\\
Since $c_1(T_X)=\displaystyle \sum_{i=1}^{r}V(\langle v_i \rangle)$ (see for instance, \cite{cox} Theorem 8.2.3), if $X$ is smooth and $\mathcal{R(P)}=\displaystyle\sum_{i=1}^{r}a_iv_i$ is a primitive relation, then $$\Delta(\mathcal{P})=\displaystyle\sum_{i=1}^{r}a_i=\displaystyle\sum_{i=1}^{r}V(\langle v_i \rangle)\cdot\mathcal{R(P)}=-K_X\cdot\mathcal{R(P)}.$$
Hence, using Kleiman's Criterion of ampleness (Theorem \ref{kleiman}), we can give a characterization of smooth toric Fano varieties in terms of primitive relations:\\

A smooth toric variety $X_{\Sigma}$ is a Fano variety  if and only if $\Delta(\mathcal{P})>0$ for every primitive relation $\mathcal{P}$ of $\Sigma$.\\

From now on, $X:=X_{\Sigma}$ will denote a smooth projective toric variety.\\
In (\cite{bat}, 2.2.4) we see that Fano toric varieties can be recovered from the set of primitive relations. In that paper Batyrev gives a classification of toric Fano 4-folds by describing the possible sets of primitive relations in dimension 4. He also gives  a geometric description for these varieties. He found 123 isomorphism classes of smooth toric Fano 4-fold. Then in \cite{sato} Sato noticed one missing isomorphism class in Batyrev's classification and he described the primitive relations of this missing class, completing the classification of toric Fano 4-folds. \\
 Note that if $X_{\Sigma}$ is a smooth toric variety then any set of primitive vectors that generate a maximal cone of $\Sigma$ can be chosen to be the canonical basis of $\mathbb{Z}^n$. By definition of primitive collection, a simplicial cone $\sigma$ generated by vectors from $\Sigma(1)$ belongs to $\Sigma$ if and only if $\sigma$ does not contain any primitive collection. In the next example, we illustrate how to recover a smooth toric Fano variety from its set of primitive relations.
\\

  \begin{ex} \em{
Let $X_{\Sigma}$ be the toric Fano 4-fold given by the following primitive relations:\\
$v_1+v_2=v_8, \ v_7+v_8=v_1, \ v_1+v_6=v_7,  \ v_2+v_7=0, \ v_6+v_8=0, \ v_3+v_4+v_5=2v_1$. 

Then, the primitive collections are: $\{v_1,v_2\}$, $\{v_7,v_8\}$, $\{v_1,v_6\}$, $\{v_2,v_7\}$, $\{v_6,v_8\}$, $\{v_3,v_4,v_5\}$.
Thus, by definition of primitive collection, the fan $\Sigma$ obtained from primitive relations, satisfies:\\
$$\sigma=\langle v_i,v_j,v_k,v_l\rangle \in \Sigma \Leftrightarrow
\left\{
\begin{array}{lll}
\langle v_1,v_2\rangle & \nsubseteq & \sigma \\
\langle v_7,v_8\rangle & \nsubseteq & \sigma \\
\langle v_1,v_6\rangle & \nsubseteq & \sigma \\
\langle v_2,v_7\rangle & \nsubseteq & \sigma \\
\langle v_6,v_8\rangle & \nsubseteq & \sigma \\
\langle v_3,v_4,v_5\rangle & \nsubseteq & \sigma \\
\end{array}
\right.
$$
Since $X_{\Sigma}$ is smooth, every maximal cone in $\Sigma$ provides a basis to $N\simeq \mathbb{Z}^4$.
The cone $\langle v_1,v_2,v_3,v_4\rangle$ is maximal in $\Sigma$, so we can take $v_1=(1,0,0,0), v_2=(0,1,0,0), v_3=(0,0,1,0), v_4=(0,0,0,1)$. Thus, from the primitive relations, we get $v_5=(2,0,-1,-1), v_6=(-1,-1,0,0), v_7=(0,-1,0,0), v_8=(1,1,0,0)$.}\\
\end{ex}

 In the table below we list all smooth toric 4-folds and its primitive collections or geometric description. The last variety in our table follows Sato's notation. For the others, our notation differs from Batyrev's notation used in \cite{bat} only in the enumeration of  minimal vectors. Whenever he enumerates the vectors from $0$ to $k$, we will enumerate them from $1$ to $k+1$. We denote by $S_i$ the Del Pezzo surface obtained by the blow up of $i$ points in general position on $\mathbb{P}^2$ for $i=1,2$ and 3. It is clear that primitive collections are not enough to describe the variety. They describe only its combinatorial type.  \\ \\

\noindent
 \begin{tabular}{|p{2cm}| p{9,3cm}|}
\hline
Notation & \hspace{0,85cm} Primitive Collections or Geometric Description  \\ 

\hline
& $\mathbb{P}^4$\\ \hline

$B_1,...,B_5$ & $\{v_5,v_6\}$, $\{v_1, v_2, v_3, v_4\}$ \\ \hline

$C_1,...,C_4$ & $\{v_1, v_2, v_3\}$, $\{v_4, v_5, v_6\}$\\ \hline

$D_1,...,D_{19}$ &  $\{v_4,v_5\}$, $\{v_6, v_7\}$, $\{v_1, v_2, v_3\}$ \\ 
\hline

$E_1,E_2,E_3$ & $\{v_1,v_7\}$, $\{v_1, v_2\}$, $\{v_6, v_7\}$, $\{v_2, v_3, v_4, v_5\}$, $\{v_3, v_4, v_5, v_6\}$  \\
\hline

 $G_1,...,G_6$ & $\{v_1, v_7\}$, $\{v_2, v_3, v_4\}$, $\{v_4, v_5, v_6\}$, $\{v_5, v_6, v_7\}$, $\{v_1, v_2, v_3\}$ \\
\hline

$H_1,...,H_{10}$ & $\{v_1, v_2\}$, $\{v_7, v_8\}$, $\{v_1, v_6\}$, $\{v_2, v_7\}$, $\{v_6, v_8\}$, $\{v_3, v_4, v_5\}$\\
\hline

 $I_1,...,I_{15}$ & $\{v_1,v_2\}$, $\{v_7, v_8\}$, $\{v_3,v_6\}$, $\{v_6,v_8\}$, $\{v_3, v_4, v_5\}$, $\{v_4, v_5, v_7\}$ \\
\hline
 
$J_1,J_2$ & $\{v_3,v_6\}$, $\{v_6, v_8\}$, $\{v_7,v_8\}$, $\{v_1, v_2, v_3\}$, $\{v_1, v_2,v_7\}$, $\{v_1,v_2,v_8\}$, $\{v_3, v_4, v_5\}$, $\{v_4, v_5, v_6\}$, $\{v_4, v_5, v_7\}$ \\
\hline

$K_1,...,K_4$ & $\{v_7, v_9\}$, $\{v_1, v_8\}$, $\{v_8, v_9\}$, $\{v_2, v_8\}$, $\{v_6, v_7\}$, $\{v_1, v_6\}$, $\{v_6, v_9\}$, $\{v_1, v_2\}$, $\{ v_2, v_7\}$, $\{v_3, v_4, v_5\}$ \\
\hline

$L_1,...,L_{13}$ & $ \{v_1, v_8\}$, $\{v_2, v_3\}$, $\{v_4, v_5\}$, $\{v_6, v_7\}$  \\ \hline

$M_1,...,M_4$ & $\{v_1, v_8\}$, $\{v_4, v_5\}$, $\{v_6, v_7\}$, $\{v_1, v_2, v_3\}$, $\{v_4, v_6, v_8\}$, $\{v_2, v_3, v_5\}$, $\{v_2, v_3, v_7\}$  \\
\hline

$Q_1,...,Q_{17}$ & $\{v_1,v_2\}$, $\{v_1, v_8\}$, $\{v_2,v_7\}$, $\{v_3, v_5\}$, $\{v_4, v_6\}$, $\{v_8,v_9\}$, $\{v_7, v_9\}$ \\
\hline

$R_1,R_2,R_3$ & $\{v_7,v_9\}$, $\{v_4,v_8\}$, $\{v_8,v_9\}$, $\{v_6,v_7\}$, $\{v_3,v_5\}$, $\{v_4,v_6\}$, $\{v_1, v_2, v_9\}$, $\{v_3, v_6, v_8\}$, $\{v_1, v_2, v_5\}$, $\{v_1, v_2, v_7\}$, $\{v_1, v_2, v_4\}$ \\
\hline

$108$ & $\{v_7,v_9\}$, $\{v_8, v_9\}$, $\{v_3,v_5\}$, $\{v_4, v_6\}$, $\{v_1, v_7\}$, $\{v_3,v_6\}$, $\{v_1, v_2, v_5\}$, $\{v_1, v_2, v_4\}$, $\{v_2, v_5, v_8\}$, $\{v_2, v_4, v_8\}$ \\ \hline 

$U_1,...,U_8$ & $\{v_1, v_3\}$, $\{v_2, v_4\}$, $\{v_1, v_4\}$, $\{v_3, v_5\}$, $\{v_4, v_6\}$, $\{v_2, v_5\}$, $\{v_1, v_5\}$, $\{v_2, v_6\}$, $\{ v_3, v_6\}$,  $\{ v_7, v_8\}$,  $\{ v_9, v_{10}\}$ \\
\hline

$Z_1,Z_2$ & $\{v_1, v_8\}$, $\{v_5, v_7\}$, $\{v_1, v_2, v_5\}$, $\{v_1, v_2, v_6\}$, $\{v_2, v_4, v_5\}$, $\{v_2, v_4, v_6\}$, $\{v_3, v_7, v_8\}$, $\{v_3, v_4, v_6\}$, $\{v_3, v_4, v_7\}$, $\{v_3, v_6, v_8\}$   \\
\hline
\end{tabular}

\noindent
 \begin{tabular}{|p{2cm}| p{9,3cm}|}

\hline

$117$ & $\{v_4, v_{10}\}$, $\{v_1, v_5\}$, $\{v_2, v_6\}$, $\{v_3, v_7\}$, $\{v_8, v_9\},$ $\{v_1, v_2, v_{10}\}$, $\{v_1, v_3, v_{10}\}$, $\{v_2, v_3, v_{10}\}$, $\{v_1, v_2, v_3\}$, $\{v_1, v_9, v_{10}\}$, $\{v_2, v_9, v_{10}\}$, $\{v_3, v_9, v_{10}\},$  $\{v_1, v_2, v_9\}$, $\{v_1, v_3, v_9\}$, $\{v_2, v_3, v_9\}$, $\{v_4, v_5, v_6\}$, $\{v_4, v_5, v_7\}$, $\{v_4, v_6, v_7\}$, $\{v_5, v_6, v_7\}$, $\{v_4, v_5, v_8\}$, $\{v_4, v_6, v_8\}$, $\{v_4, v_7, v_8\}$, $\{v_5, v_6, v_8\}$, $\{v_5, v_7, v_8\}$,  $\{v_6, v_7, v_8\}$           \\
\hline

$118$ & $\{v_4, v_9\}$,  $\{v_1, v_5\}$, $\{v_2, v_6\}$, $\{v_3, v_7\}$,  $\{v_1,v_2,v_9\}$, $\{v_1, v_3, v_9\}$, $\{v_2, v_3, v_9\}$,  $\{v_1, v_2, v_3\}$, $\{v_4, v_5, v_8\}$, $\{v_4, v_6, v_8\}$, $\{v_4, v_7, v_8\}$, $\{v_5, v_6, v_8\}$, $\{v_5, v_7, v_8\}$, $\{v_6, v_7, v_8\}$   \\ \hline

$119, 120, 121$ & $S_2\times S_2, \ S_2\times S_3, \ S_3\times S_3$ \\ \hline

$124$ & $\{v_1, v_4\}$,  $\{v_2, v_5\}$, $\{v_3, v_6\}$, $\{v_1, v_2, v_3\}$,  $\{v_4,v_5,v_6\}$, $\{v_7, v_8, v_9\}$,  $\{v_1, v_2, v_9\}$, $\{v_4, v_5, v_9\}$, $\{v_1, v_3, v_8\}$, $\{v_4, v_6, v_8\}$, $\{v_2, v_3, v_7\}$, $\{v_5, v_6, v_7\}$, $\{v_1, v_8, v_9\},$ $\{v_4, v_8, v_9\}$, $\{v_2, v_7, v_9\}$, $\{v_5, v_7, v_9\}$, $\{v_3, v_7, v_8\}$, $\{v_6, v_7, v_8\}$ \\ \hline

\end{tabular}\\ \\

\subsection{Second Chern Class Computation} \label{secondchern}

In this section, we will compute $ch_2(T_X)$ in terms of the invariant divisors $D_i:=V(\langle v_i\rangle)$ and we will give an analogue of the Toric Cone Theorem \ref{teo cone}. We also give a formula for the second Chern character of a variety obtained by a blow up.\\

\begin{pr} \label{2chern} For a smooth toric variety $X$ we have: $$ ch_2(T_X)=\frac{1}{2}\left(\displaystyle \sum_{i=1}^{r}D_i^2\right).$$ 
\end{pr}

\begin{proof} There are exact sequences (see \cite[4.0.28, 8.1.1]{cox}):\\
$$0\rightarrow\Omega_X^1\rightarrow \mathcal{O}_X^n\rightarrow \displaystyle \oplus_{i=1}^{r}\mathcal{O}_{D_i}\rightarrow0$$
$$0\rightarrow\mathcal{O}(-D_i)\rightarrow \mathcal{O}_X\rightarrow \mathcal{O}_{D_i}\rightarrow0$$ 
Where $\mathcal{O}_{D_i}$ is the structure sheaf on $D_i$ extended by zero to X.\\

Using Whitney sum we have:\\

$0=ch_2(\mathcal{O}_{X{_\Sigma}}^n)=ch_2(\Omega_X^1)+ch_2(\displaystyle \oplus_{i=1}^{r}\mathcal{O}_{D_i})$\\

$ch_2(\mathcal{O}_{D_i})=-ch_2(\mathcal{O}(-D_i)=-\frac{1}{2}D_i^2$ \ for all $i\in\{1,...,r\}$.\\

$\Rightarrow ch_2(T_X)=ch_2(\Omega_X^1)=-ch_2(\displaystyle \oplus_{i=1}^{r}\mathcal{O}_{D_i})=\frac{1}{2}\left(\displaystyle \sum_{i=1}^{r}D_i^2\right)$. 
 \end{proof}
  
By definition, a variety $X$ has positive second Chern character if for any surface $S\in X$ we have $ch_2(X)\cdot S>0$. However, in the toric case, we only need to check this inequalities for invariant surfaces, because of the following result. The proof sketched below is due to D. Monsôres.

\begin{pr} \label{dm} Let $X:= X_{\Sigma}$ be a complete toric variety of dimension $n \geq 3$. If $S$ is a surface
contained in $X$, then we have a numerical equivalence: 

\begin{center}     $S \equiv \displaystyle\sum_{\sigma \in \Sigma(n-2)} a_{\sigma} \cdot
[V(\sigma)]$

\end{center}

 \noindent with $a_{\sigma} \geq 0$, \  $\forall \sigma \in \Sigma(n-2)$.

\end{pr}

\begin{proof} The proof is by induction on the dimension of $X$. If $n=3$ then $S$ is an efective divisor on $X$ and by \cite[Section 5.1]{fulton}, $S\sim \displaystyle\sum_{\sigma\in \Sigma(1)}a_{\sigma}\cdot [V(\sigma)]$ with $a_{\sigma}\geq 0, \forall \sigma\in \Sigma(1).$ By induction hypothesis we can suppose that $S$ intersects the torus $T$ of $X$. Consider the action $\mathbb{C}^*\times X\rightarrow X$ given by $(t,x)\mapsto t^{\lambda_1}\cdot x$, where $\lambda_1\in N$. This action induces a rational map $f:\mathbb{C}\times S\dashrightarrow X$. Consider a toric resolution on indeterminacy for this map:\\

$$\xymatrix{
Y \ar@/_0.8cm/[dd]_{\pi}
\ar[d]_p \ar[dr]^{\psi} & \\ 
\mathbb{C}\times S \ar[d]_q \ar@{-->}[r]^f  & X  \\
\mathbb{C} & & 
}$$
 By \cite[III 9.6,9.7]{hartshorne} $\pi$ is a flat morphism whose fibers have pure dimension two. Hence the cycles $\psi_*(\pi^*(0))$ and $\psi_*(\pi^*(1))$ are rationally equivalent. Since $S=\psi_*(\pi^*(1))$ we get a numerical equivalence $S\equiv \displaystyle\sum_{i=1}^ {k}a_iS_i$, where $a_i\geq 0$ and $S_i\subset \psi(\pi^{-1}(0)) \ \forall i=1,...,k$. Note that by construction $\psi_*(\pi^*(0))$ and therefore every $S_i$ is invariant by the action of $\lambda_1$. If each surface $S_i$ has empty intersection with the torus then we conclude the proposition by induction. If $S_i$ intersects the torus we take $\lambda_2\in N$ such that $\{\lambda_1, \lambda_2\}$ are linearly independent and repeat the construction above to $S_i$ and $\lambda_2$. We get $S_i\equiv \displaystyle\sum_{j=1}^{r}b_jS'_j$, where $b_j\geq 0$ and each $S'_j$ is an invariant surface by the actions of $\lambda_1$ and $\lambda_2$. If $S'_j\cap T=\emptyset \ \forall j=1,...,r$ we are done. If for some $S'_j$ it fails, we repeat the process using a parameter $\lambda_3$ such that $\{\lambda_1,\lambda_2,\lambda_3\}$ are linearly independent and obtain $S'_j\equiv \displaystyle\sum_{k=1}^{s}c_kS''_k$ where $c_k\geq 0$ and $S''_k$ are invariant surfaces by the actions of $\lambda_1, \lambda_2$ and $\lambda_3$. To finish the proof we observe that $S''_k\cap T=\emptyset \ \forall k=1,...,s$. Since $\{\lambda_1, \lambda_2,\lambda_3\}$ are linearly independent, if there were $t\in S''_k\cap T$ we would have an injective map $(\mathbb{C}^*)^3\rightarrow S''_k$ given by $(t_1,t_2,t_3)\mapsto t_1^{\lambda_1}\cdot t_2^{\lambda_2}\cdot t_3^{\lambda_3}\cdot t$. But this is absurd since $S''_k$ is a surface. 
 \end{proof}
 \vspace{0.4cm}

So, in order to check whether a smooth toric Fano variety is 2-Fano, we need to compute $\displaystyle \sum_{i=1}^{d}D_i^2 \cdot S$ for invariant surfaces $V(\sigma)$. 

If $V(\sigma)$ is not contained in the support of $D_i$ we can use  the Remark \ref{rkinter}  to compute $D_i\cdot V(\sigma)$.
 
Suppose, otherwise, that $V(\sigma)\subseteq supp(D_i)$. Since $D_i$ is a Cartier divisor, there exists $u\in M$ such that $(D_i)_{\mid_{U_{\sigma}}}=div (\chi^ u)_{\mid_{U_{\sigma}}}$. Hence,  $D_i-div(\chi^u)$ is linearly equivalent to  $D_i$ and the support of $D_i - div(\chi^u)=D_i - \displaystyle \sum_{i=1}^{r}\langle u,v_j\rangle D_j$ does not contain $V(\sigma)$. Thus, if we find an element $u\in M$ satisfying $(D_i)_{\mid_{U_{\sigma}}}=div (\chi^ u)_{\mid_{U_{\sigma}}}$ then we can use again the Remark \ref{rkinter} to compute $D_i \cdot V(\sigma)=(D_i-div(\chi^ u)) \cdot V(\sigma)$.\\

By the cone-orbit correspondence, we have:
$$U_{\sigma}\cap D_j \neq \emptyset \Leftrightarrow \langle v_j\rangle \subseteq \sigma.$$
Since $div (\chi^ u)_{\mid_{U_{\sigma}}}=\displaystyle \sum_{v_j \in \sigma}\langle u,v_j\rangle (D_j)_{\mid_{U_{\sigma}}}$, in order to have $(D_i)\mid_{U_{\sigma}}=div(\chi^u)\mid_{U_{\sigma}}$ we can take $u$ to be any element in $M$ such that $\langle u,v_i\rangle =1$ and $\langle u,v_j\rangle =0 \ \forall j \neq i$ such that $v_j \in \sigma$. With this, we are ready to compute the product of $ch_2(T_X)$ with $V(\sigma)$.\\

Next, we give a formula of the second Chern character of a variety obtained by a blow up. This formula can be also found in \cite{jason}. For the convenience of the reader we sketch the proof below.

\begin{lema}\label{blowup} Let $X$ be a smooth projective variety and $Z\subset X$ a (smooth) invariant subvariety of $X$ of codimension $c$. Denote by $\pi:\tilde{X}\rightarrow X$ the blowing up of $X$ along of $Z$, $j:E\hookrightarrow \tilde{X}$ the natural inclusion of the exceptional divisor $E$ and $f:=\pi{|_E}:E\rightarrow Z$. Then $$ch_2(T_{\tilde{X}})=\pi^*ch_2(T_X)+\frac{c+1}{2}E^2-j_*\Big(f^*\big(c_1(N_{Z|X})\big)\Big).$$
In particular, the blow up of $\mathbb{P}^n$ along of a subvariety of codimension $2$ is not $2$-Fano.  
\end{lema}
\begin{proof}
Using the exact sequence $$0\longrightarrow\pi^*\Omega_X\longrightarrow\Omega_{\tilde{X}}\longrightarrow j_*\Omega_{f}\longrightarrow 0$$ we get $$ch(\Omega_{\tilde{X}})=\pi^*ch(\Omega_X)+ch(j_*\Omega_f). \ \ \ \ \ \ \ (1)$$
Grothendieck-Riemann-Roch gives $$ch(j_*\Omega_f)\cdot td(T_{\tilde{X}})=j_*(ch(\Omega_f)\cdot td(T_E)).\ \ \ \ \ \ \ (2)$$
From the exact sequence $$0\longrightarrow T_E\longrightarrow j^*T_{\tilde{X}}\longrightarrow j^* \mathcal{O}_{\tilde{X}}(E) \longrightarrow 0$$
\noindent we have  $$td(T_E)=td(j^*T_{\tilde{X}})\cdot td(j^* \mathcal{O}_{\tilde{X}}(E))^{-1}=j^*\Big(\displaystyle\frac{E}{1-e^{-E}}\Big)^{-1}.\ \ \ \ \ \ \ (3)$$
  
Putting together $(2)$ and $(3)$ we obtain
$$ch(j_*\Omega_f)=j_*(ch(\Omega_f))\cdot \Big(\displaystyle\frac{1-e^{-E}}{E}\Big).\ \ \ \ \ \ \ (4)$$
Using the isomorphism $\mathcal{O}_E(-1)\simeq j^*\mathcal{O}_{\tilde{X}}(E)$ and the Euler sequence for $\Omega_f$,
$$0\longrightarrow\Omega_f\longleftrightarrow f^*N_{Z|X}^{\vee}\otimes\mathcal{O}_E(-1)\longrightarrow \mathcal{O}_E\longrightarrow 0,$$
we obtain
$$ch(\Omega_f)=f^*ch(N_{Z|X}^{\vee})\cdot j^*(1+E)-1.\ \ \ \ \ \ \ (5)$$

Putting together $(1),(4)$ and $(5)$ we arrive at $$ch(\Omega_{\tilde{X}})=\pi^*ch(\Omega_X)+j_*\Big(f^*ch(N_{Z|X}^{\vee})\cdot j^*(1+E)-1\Big)\cdot \left(\frac{1-e^{-E}}{E}\right).$$
Taking the degree $2$ piece we prove the first part of the lemma. To prove the last statement we take a curve $C\subset Z$ and consider the surface 
$S:=f^ {-1}(C)\simeq\mathbb{P}(N^{\vee})$ where $N$ denotes the normal bundle $N_{Z|X}|_C$. \\
Then,
$ch_2(T_{\tilde{X}})\cdot S=\displaystyle\frac{3}{2}E^2\cdot S-E\cdot f^*c_1(N_{Z|X})\cdot S=-\displaystyle\frac{1}{2}\mbox{deg}(N)=-\displaystyle\frac{1}{2}\mbox{det}(N)\cdot C$.

 If $X$ is $\mathbb{P}^n$ then $T_X$ is ample, and thus so is $N_{Z|X}$. So $\mbox{det}(N)$ is ample and $ch_2(T_{\tilde{X}})\cdot S<0$.
\end{proof}

\subsection{The Main Result}

\begin{teo} For $n\leq 6$ the only toric Fano $n$-fold with positive second Chern character is $\mathbb{P}^n$.
\end{teo}

\textit{Proof.} For $n\leq 2$ the theorem follows from the classification of toric Del Pezzo surfaces and Lemma \ref{blowup}. For $n=3$ the result follows from \cite{carol}.

 We claim that in the following cases $ch_2(T_X)$ is not positive: 

\begin{enumerate}
\item $X=Z\times Y$ is a product of positive-dimensional Fano manifolds.
\item $X=\mathbb{P}_Y(E)$ is a projective bundle  over a positive-dimensional Fano manifold $Y$.
\end{enumerate}
In the first case recall we have an isomorphism $T_{Z\times Y}\simeq \pi_Z^*T_Z\times \pi_Y^*T_Y$ and therefore $ch(T_X)=\pi_Z^*(ch(T_Z))+\pi_Y^*(ch(T_Y))$. If $A$ and $B$ are curves in $Z$ and $Y$ respectively, then $ch_2(T_X)\cdot (A\times B)=0$. 
The second item follows from \cite[4.1]{jason} and can also be obtained as a special case of Theorem \ref{ch}.

As consequence, the toric Fano 4-folds listed in the Batyrev's classification that are of type 1. or 2. do not have positive second Chern Character. They are (see Batyrev's description of these varieties in \cite{bat}):\\
 
 $B_1,...,B_5$,$C_1,...,C_4,D_1,...,D_{19},H_8, L_1,...,L_{13},I_7,I_{11},I_{13},Q_6,Q_8,Q_{10}$,
 $Q_{11}$, $Q_{15}$,$K_4,U_4,U_5,U_6,119,120,121$. \\
 
In the remaining 4-dimensional cases, we computed $ch_2(T_X)\cdot S$ for all invariant surfaces $S\subset X$, as described in Section \ref{secondchern}. To make the computation we used the program Maple. For all smooth toric Fano 4-folds $X\neq \mathbb{P}^4$ in Batyrev's list we found a surface $S\subset X$ such that $ch_2(T_X)\cdot S \leq 0$.\\

The next table summarizes our results. The first column lists toric Fano 4-folds according Batyrev's notation. The second column lists its primitive vectors explicity. The third column gives an invariant surface $S$ for which the intersection number $ch_2(T_X)\cdot S$ (listed on the last column) is non positive.\\ \\

\noindent
\begin{tabular}{  |c|  p{7cm} | c| c| }
\hline
 & \centering{Primitive \ Vectors} & Surface & $ch_2(T_X)\cdot S$  \\ 
\hline
$E_1$ & $v_1=e_1, v_2=e_2,v_3=e_3,v_4=e_4,v_5=2e_1-e_2-e_3-e_4,v_6=e_1+e_2,v_7=-e_1$ & $V(v_2,v_3)$ &  -2 \\ \hline

$E_2$ & $v_1=e_1, v_2=e_2,v_3=e_3,v_4=e_4,v_5=e_1-e_2-e_3-e_4,v_6=e_1+e_2,v_7=-e_1$ & $V(v_2,v_3)$ & $-\displaystyle\frac{3}{2}$ \\ \hline

$E_3$ & $v_1=e_1, v_2=e_2,v_3=e_3,v_4=e_4,v_5=-e_2-e_3-e_4,v_6=e_1+e_2,v_7=-e_1$ & $V(v_2,v_3)$ & -1 \\ \hline

$G_1$ & $v_1=e_1, v_2=e_2,v_3=e_3,v_4=e_1-e_2-e_3,v_5=e_4,v_6=e_1+e_2+e_3-e_4,v_7=-e_1$ & $V(v_1,v_5)$ & $-\displaystyle\frac{1}{2}$ \\ \hline

$G_2$ & $v_1=e_1, v_2=e_2,v_3=-e_1-e_2,v_4=e_4,v_5=e_3,v_6=2e_1-e_3-e_4,v_7=-e_1+e_4$ & $V(v_1,v_5)$ & -2 \\ \hline

$G_3$ & $v_1=e_1, v_2=e_2,v_3=e_3,v_4=-e_2-e_3,v_5=e_4,v_6=e_1+e_2+e_3-e_4,v_7=-e_1$ & $V(v_1,v_5)$ & -1 \\ \hline

$G_4$ & $v_1=e_1, v_2=e_2,v_3=-e_1-e_2,v_4=e_4,v_5=e_3,v_6=e_1+e_2-e_3-e_4,v_7=-e_1+e_4$ & $V(v_1,v_5)$ & $-\displaystyle\frac{1}{2}$ \\ \hline

$G_5$ & $v_1=e_1, v_2=e_2,v_3=-e_1-e_2,v_4=e_4,v_5=e_3,v_6=-e_3-e_4,v_7=-e_1+e_4$ & $V(v_2,v_5)$ & -2 \\ \hline

$G_6$ & $v_1=e_1, v_2=e_2,v_3=-e_1-e_2,v_4=e_4,v_5=e_3,v_6=e_1-e_3-e_4,v_7=-e_1+e_4$ & $V(v_2,v_5)$ & $-\displaystyle\frac{3}{2}$ \\ \hline
\end{tabular}
\noindent
\begin{tabular}{  |c| p{7cm} | c| c| }
\hline
& \centering{Primitive \ Vectors} & Surface & $ch_2(T_X)\cdot S$  \\ 
\hline
$H_1$ & $v_1=e_1, v_2=e_2,v_3=e_3,v_4=e_4,v_5=2e_1-e_3-e_4,v_6=-e_1-e_2,v_7=-e_2,v_8=e_1+e_2$ & $V(v_3,v_4)$ & $-\displaystyle\frac{3}{2}$ \\ \hline

$H_2$ & $v_1=e_1, v_2=e_2,v_3=e_3,v_4=e_4,v_5=2e_1+e_2-e_3-e_4,v_6=-e_1-e_2,v_7=-e_2,v_8=e_1+e_2$ & $V(v_3,v_4)$ & -1 \\ \hline

$H_3$ & $v_1=e_1, v_2=e_2,v_3=e_3,v_4=e_4,v_5=2e_1+2e_2-e_3-e_4,v_6=-e_1-e_2,v_7=-e_2,v_8=e_1+e_2$ & $V(v_3,v_4)$ & $-\displaystyle\frac{3}{2}$ \\ \hline

$H_4$ & $v_1=e_1, v_2=e_2,v_3=e_3,v_4=e_4,v_5=e_1-e_3-e_4,v_6=-e_1-e_2,v_7=-e_2,v_8=e_1+e_2$ & $V(v_3,v_4)$ & $-\displaystyle\frac{3}{2}$ \\ \hline

$H_5$ & $v_1=e_1, v_2=e_2,v_3=e_3,v_4=e_4,v_5=e_1+e_2-e_3-e_4,v_6=-e_1-e_2,v_7=-e_2,v_8=e_1+e_2$ & $V(v_3,v_4)$ & $-\displaystyle\frac{3}{2}$\\ \hline

$H_6$ & $v_1=e_1, v_2=e_2,v_3=e_3,v_4=e_4,v_5=e_1+2e_2-e_3-e_4,v_6=-e_1-e_2,v_7=-e_2,v_8=e_1+e_2$ & $V(v_3,v_4)$ & $-\displaystyle\frac{3}{2}$ \\ \hline

$H_7$ & $v_1=e_1$, $v_2=e_2$, $v_3=e_3$, $v_4=e_4$, $v_5=2e_2-e_3-e_4,$ $v_6=-e_1-e_2,$ $v_7=-e_2,$ $v_8=e_1+e_2$ & $V(v_3,v_4)$ & $-\displaystyle\frac{3}{2}$ \\ \hline

$H_9$ & $v_1=e_1, v_2=e_2,v_3=e_3,v_4=e_4,v_5=e_2-e_3-e_4,v_6=-e_1-e_2,v_7=-e_2,v_8=e_1+e_2$ & $V(v_3,v_4)$ & $-\displaystyle\frac{3}{2}$\\ \hline

$H_{10}$ & $v_1=e_1, v_2=e_2,v_3=e_3,v_4=e_4,v_5=-e_1-e_3-e_4,v_6=-e_1-e_2,v_7=-e_2,v_8=e_1+e_2$ & $V(v_3,v_4)$ & $-\displaystyle\frac{3}{2}$ \\ \hline

$I_1$ & $v_1=e_1, v_2=-e_1+e_3,v_3=e_3,v_4=e_4,v_5=-2e_2+e_3-e_4,v_6=e_2-e_3,v_7=e_2,v_8=-e_2+e_3$ & $V(v_1,v_4)$ & $-\displaystyle\frac{3}{2}$ \\ \hline

$I_2$ & $v_1=e_1, v_2=e_2,v_3=e_3,v_4=e_4,v_5=2e_1+2e_2-e_3-e_4,v_6=-e_1-e_2,v_7=-e_1-e_2+e_3,v_8=e_1+e_2$ & $V(v_1,v_4)$ & $-\displaystyle\frac{3}{2}$\\ 
\hline
$I_3$ & $v_1=e_1, v_2=e_2,v_3=e_3,v_4=e_4,v_5=2e_1+e_2-e_3-e_4,v_6=-e_1-e_2,v_7=-e_1-e_2+e_3,v_8=e_1+e_2$ & $V(v_1,v_4)$ & $-\displaystyle\frac{3}{2}$ \\ \hline

$I_4$ & $v_1=e_1, v_2=e_2,v_3=e_3,v_4=e_4,v_5=-2e_1-2e_2+e_3-e_4,v_6=e_1+e_2-e_3,v_7=e_1+e_2,v_8=-e_1-e_2+e_3$ & $V(v_1,v_4)$ & $-\displaystyle\frac{3}{2}$ \\ \hline
$I_5$ & $v_1=e_1, v_2=e_2,v_3=e_3,v_4=e_1+e_2,v_5=-e_1-e_2-e_3+2e_4,v_6=-e_4,v_7=e_3-e_4,v_8=e_4$ & $V(v_1,v_4)$ & $-\displaystyle\frac{3}{2}$ \\ \hline

$I_6$ & $v_1=e_1, v_2=e_2,v_3=e_1+e_2,v_4=e_4,v_5=e_3,v_6=-e_1-e_2-e_3-e_4,v_7=-e_3-e_4,v_8=e_1+e_2+e_3+e_4$ & $V(v_1,v_4)$ & $-\displaystyle\frac{3}{2}$ \\ \hline
\end{tabular}
\noindent
\begin{tabular}{  |c| p{7cm} | c| c| }
\hline
& \centering{Primitive \ Vectors} & Surface & $ch_2(T_X)\cdot S$  \\ 
\hline
$I_8$ & $v_1=e_1, v_2=e_2,v_3=e_3,v_4=e_4,v_5=-e_2-e_3-e_4,v_6=e_1+e_2,v_7=e_1+e_2+e_3,v_8=-e_1-e_2$ & $V(v_1,v_4)$ & $-\displaystyle\frac{3}{2}$ \\ \hline

$I_9$ & $v_1=e_1, v_2=e_2,v_3=e_3,v_4=e_4,v_5=-e_1-e_2-e_4,v_6=e_1+e_2-e_3,v_7=e_1+e_2,v_8=-e_1-e_2+e_3$ & $V(v_1,v_4)$ & $-\displaystyle\frac{3}{2}$\\ \hline

$I_{10}$ & $v_1=e_1, v_2=e_2,v_3=e_3,v_4=e_4,v_5=e_1+e_2-e_3-e_4,v_6=-e_1-e_2,v_7=-e_1-e_2+e_3,v_8=e_1+e_2$ & $V(v_1,v_4)$ & $-\displaystyle\frac{3}{2}$ \\ \hline

$I_{12}$ & $v_1=e_1, v_2=e_2,v_3=e_3,v_4=e_4,v_5=-e_1-e_2-e_3-e_4,v_6=e_1+e_2,v_7=e_1+e_2+e_3,v_8=-e_1-e_2$ & $V(v_1,v_4)$ & $-\displaystyle\frac{3}{2}$ \\ \hline

$I_{14}$ & $v_1=e_1, v_2=e_2,v_3=e_3,v_4=e_1+e_2,v_5=e_4,v_6=-e_1-e_2-e_3-e_4,v_7=-e_1-e_2-e_4,v_8=e_1+e_2+e_3+e_4$ & $V(v_1,v_4)$ & $-\displaystyle\frac{3}{2}$ \\ \hline

$I_{15}$ & $v_1=e_1, v_2=e_2,v_3=e_3,v_4=e_4,v_5=-2e_1-2e_2-e_3-e_4,v_6=e_1+e_2,v_7=e_1+e_2+e_3,v_8=-e_1-e_2$ & $V(v_1,v_4)$ & $-\displaystyle\frac{3}{2}$\\ \hline
$J_1$ & $v_1=e_1, v_2=e_2,v_3=e_3,v_4=e_4,v_5=-e_3-e_4,v_6=e_1+e_2+e_3,v_7=e_1+e_2+2e_3,v_8=-e_1-e_2-e_3$ & $V(v_1,v_3)$ & -1 \\ \hline

$J_2$ & $v_1=e_1, v_2=e_2,v_3=-e_1-e_2,v_4=e_4,v_5=e_3,v_6=e_1+e_2-e_3-e_4,v_7=-e_3-e_4,v_8=-e_1-e_2+e_3+e_4$ & $V(v_1,v_3)$ & $-\displaystyle\frac{1}{2}$ \\ \hline
$K_1$ & $v_1=e_1, v_2=e_2,v_3=e_3,v_4=e_4,v_5=2e_1+2e_2-e_3-e_4,v_6=-e_1,v_7=-e_2,v_8=-e_1-e_2,v_9=e_1+e_2$ & $V(v_3,v_4)$ & -3 \\ \hline

$K_2$ & $v_1=e_1, v_2=e_2,v_3=e_3,v_4=e_4,v_5=2e_1+e_2-e_3-e_4,v_6=-e_1,v_7=-e_2,v_8=-e_1-e_2,v_9=e_1+e_2$ & $V(v_3,v_4)$ & -3 \\ \hline

$K_3$ & $v_1=e_1, v_2=e_2,v_3=e_3,v_4=e_4,v_5=e_1+e_2-e_3-e_4,v_6=-e_1,v_7=-e_2,v_8=-e_1-e_2,v_9=e_1+e_2$ & $V(v_3,v_4)$ & -3 \\ \hline

$M_1$ & $v_1=e_1, v_2=e_2,v_3=e_3,v_4=e_4,v_5=-e_4,v_6=e_1+e_2+e_3-e_4,v_7=-e_1-e_2-e_3+e_4,v_8=-e_1$ & $V(v_2,v_4)$ & $-\displaystyle\frac{5}{2}$ \\ \hline

$M_2$ & $v_1=e_1, v_2=e_2,v_3=e_3,v_4=e_4,v_5=e_1-e_4,v_6=e_1+e_2+e_3-e_4,v_7=-e_2-e_3+e_4,v_8=-e_1$ & $V(v_2,v_4)$ & $-\displaystyle\frac{5}{2}$ \\ \hline
\end{tabular}
\noindent
\begin{tabular}{  |c| p{7cm} | c| c| }
\hline
& \centering{Primitive \ Vectors} & Surface & $ch_2(T_X)\cdot S$  \\ 
\hline
$M_3$ & $v_1=e_1, v_2=e_2,v_3=e_3,v_4=e_4,v_5=e_1-e_4,v_6=e_1+e_2+e_3-e_4,v_7=-e_2-e_3,v_8=-e_1$ & $V(v_2,v_4)$ & $-\displaystyle\frac{5}{2}$ \\ \hline
$M_4$ & $v_1=e_1, v_2=e_2,v_3=e_3,v_4=e_4,v_5=e_1-e_4,v_6=e_1+e_2+e_3-e_4,v_7=-e_1-e_2-e_3+e_4,v_8=-e_1$ & $V(v_2,v_4)$ & $-\displaystyle\frac{5}{2}$\\ \hline

$M_5$ & $v_1=e_1,$ $v_2=e_2,$ $v_3=-e_1-e_2+e_4,$ $v_4=e_3,$ $v_5=e_1-e_3-e_4,$ $v_6=e_4,$ $v_7=e_1-e_4,$ $v_8=-e_3-e_4$ & $V(v_2,v_4)$ & $-\displaystyle\frac{3}{2}$ \\ \hline 

$Q_1$ & $v_1=e_1, v_2=e_2,v_3=e_3,v_4=e_4,v_5=e_1-e_3,v_6=e_1-e_4,v_7=-e_2,v_8=-e_1-e_2,v_9=e_1+e_2$ & $V(v_3,v_4)$ & $-\displaystyle\frac{3}{2}$ \\ \hline

$Q_2$ & $v_1=e_1, v_2=e_2,v_3=e_3,v_4=e_4,v_5=e_1-e_3,v_6=e_3-e_4,v_7=-e_2,v_8=-e_1-e_2,v_9=e_1+e_2$ & $V(v_3,v_4)$ & $-\displaystyle\frac{3}{2}$ \\ \hline

$Q_3$ & $v_1=e_1, v_2=e_2,v_3=e_3,v_4=e_4,v_5=e_1+e_2-e_3,v_6=e_1+e_2-e_4,v_7=-e_2,v_8=-e_1-e_2,v_9=e_1+e_2$ & $V(v_3,v_4)$ & $-\displaystyle\frac{3}{2}$ \\ \hline

$Q_4$ & $v_1=e_1, v_2=e_2,v_3=e_3,v_4=e_4,v_5=e_1-e_3,v_6=e_1+e_2-e_4,v_7=-e_2,v_8=-e_1-e_2,v_9=e_1+e_2$ & $V(v_3,v_4)$ & $-\displaystyle\frac{3}{2}$ \\ \hline

$Q_5$ & $v_1=e_1, v_2=e_2,v_3=e_3,v_4=e_4,v_5=e_1+e_2-e_3,v_6=e_3-e_4,v_7=-e_2,v_8=-e_1-e_2,v_9=e_1+e_2$ & $V(v_3,v_4)$ & $-\displaystyle\frac{3}{2}$ \\ \hline

$Q_7$ & $v_1=e_1, v_2=e_2,v_3=e_3,v_4=e_4,v_5=e_1+e_2-e_3,v_6=-e_2-e_4,v_7=-e_2,v_8=-e_1-e_2,v_9=e_1+e_2$ & $V(v_3,v_4)$ & $-\displaystyle\frac{3}{2}$ \\ \hline

$Q_9$ & $v_1=e_1, v_2=e_2,v_3=e_3,v_4=e_4,v_5=e_1+e_2-e_3,v_6=e_2-e_4,v_7=-e_2,v_8=-e_1-e_2,v_9=e_1+e_2$ & $V(v_3,v_4)$ & $-\displaystyle\frac{3}{2}$ \\ \hline

$Q_{12}$ & $v_1=e_1, v_2=e_2,v_3=e_3,v_4=e_4,v_5=e_1-e_3,v_6=e_2-e_4,v_7=-e_2,v_8=-e_1-e_2,v_9=e_1+e_2$ & $V(v_3,v_4)$ & $-\displaystyle\frac{3}{2}$ \\ \hline

$Q_{13}$ & $v_1=e_1, v_2=e_2,v_3=e_3,v_4=e_4,v_5=e_2-e_3,v_6=e_2-e_4,v_7=-e_2,v_8=-e_1-e_2,v_9=e_1+e_2$ & $V(v_3,v_4)$ & $-\displaystyle\frac{3}{2}$ \\ \hline

$Q_{14}$ & $v_1=e_1, v_2=e_2,v_3=e_3,v_4=e_4,v_5=e_2-e_3,v_6=e_3-e_4,v_7=-e_2,v_8=-e_1-e_2,v_9=e_1+e_2$ & $V(v_3,v_4)$ & $-\displaystyle\frac{3}{2}$ \\ \hline
\end{tabular}
\noindent
\begin{tabular}{  |c| p{7cm} | c| c| }
\hline
& \centering{Primitive \ Vectors} & Surface & $ch_2(T_X)\cdot S$  \\ 
\hline
$Q_{16}$ & $v_1=e_1, v_2=e_2,v_3=e_3,v_4=e_4,v_5=e_1+e_2-e_3,v_6=-e_1-e_2-e_4,v_7=-e_2,v_8=-e_1-e_2,v_9=e_1+e_2$ & $V(v_3,v_4)$ & $-\displaystyle\frac{3}{2}$ \\ 
\hline

$Q_{17}$ & $v_1=e_1, v_2=e_2,v_3=e_3,v_4=e_4,v_5=e_2-e_3,v_6=-e_1-e_2-e_4,v_7=-e_2,v_8=-e_1-e_2,v_9=e_1+e_2$ & $V(v_3,v_4)$ & $-\displaystyle\frac{3}{2}$ \\ \hline

$R_1$ & $v_1=e_1, v_2=e_2,v_3=e_3,v_4=-e_1-e_2+e_3,v_5=-e_1-e_2,v_6=e_4,v_7=-e_4,v_8=e_1+e_2-e_3-e_4,v_9=-e_1-e_2+e_3+e_4$ & $V(v_1,v_3)$ & -4 \\ \hline

$R_2$ & $v_1=e_1, v_2=e_2,v_3=e_3,v_4=-e_1-e_2+e_3,v_5=-e_1-e_2+e_4,v_6=e_4,v_7=-e_4,v_8=e_1+e_2-e_3-e_4,v_9=-e_1-e_2+e_3+e_4$ & $V(v_1,v_3)$ & -4 \\ \hline
$R_3$ & $v_1=e_1, v_2=e_2,v_3=e_3,v_4=-e_1-e_2+e_3, v_5=-e_3, v_6=e_4, v_7=-e_4,v_8=e_1+e_2-e_3-e_4,v_9=-e_1-e_2+e_3+e_4$ & $V(v_1,v_3)$ & -4 \\ \hline

$108$ & $v_1=e_1, v_2=e_2,v_3=e_3, v_4=-e_1-e_2+e_4,v_5=-e_1-e_2-e_3+e_4,v_6=-e_3,v_7=-e_4,v_8=e_1-e_4,v_9=e_4$ & $V(v_4,v_9)$ & -1 \\ \hline

$U_1$ & $v_1=e_1, v_2=e_1+e_3,v_3=e_3,v_4=-e_1,v_5=-e_1-e_3,v_6=-e_3,v_7=e_2,v_8=e_1-e_2,v_9=e_4,v_{10}=e_1-e_4$ & $V(v_3,v_7)$ & $-\displaystyle\frac{1}{2}$ \\ \hline

$U_2$ & $v_1=e_1, v_2=e_1+e_3,v_3=e_3,v_4=-e_1,v_5=-e_1-e_3,v_6=-e_3,v_7=e_2,v_8=e_1-e_2,v_9=e_4,v_{10}=e_1-e_2-e_4$ & $V(v_3,v_7)$ & $-\displaystyle\frac{1}{2}$ \\ \hline
$U_3$ & $v_1=e_1, v_2=e_1+e_3,v_3=e_3,v_4=-e_1,v_5=-e_1-e_3,v_6=-e_3,v_7=e_2,v_8=e_1-e_2,v_9=e_4,v_{10}=e_1+e_3-e_4$ & $V(v_3,v_9)$ & $-\displaystyle\frac{1}{2}$ \\ \hline
$U_7$ & $v_1=e_1, v_2=e_1+e_3,v_3=e_3,v_4=-e_1,v_5=-e_1-e_3,v_6=-e_3,v_7=e_2,v_8=e_1-e_2,v_9=e_4,v_{10}=e_3-e_4$ & $V(v_3,v_9)$ & $-\displaystyle\frac{1}{2}$ \\ \hline
$U_8$ & $v_1=e_1, v_2=e_1+e_3,v_3=e_3,v_4=-e_1,v_5=-e_1-e_3,v_6=-e_3,v_7=e_2,v_8=e_1-e_2,v_9=e_4,v_{10}=-e_1-e_4$ & $V(v_3,v_9)$ & $-\displaystyle\frac{1}{2}$ \\ \hline

$Z_1$ & $v_1=e_1, v_2=e_2,v_3=e_3,v_4=e_4,v_5=-e_1-e_2,v_6=-e_2-e_3-e_4,v_7=e_1-e_3-e_4,v_8=-e_1+e_4$ & $V(v_1,v_3)$ & $-\displaystyle\frac{5}{2}$ \\ \hline

$Z_2$ & $v_1=e_1, v_2=e_2,v_3=e_3,v_4=e_4,v_5=-e_1-e_2,v_6=-e_3-e_4,v_7=e_1+e_2-e_3-e_4,v_8=-e_1+e_4$ & $V(v_1,v_3)$ & -2 \\ \hline

$117$ & $v_1=e_2, v_2=e_3,v_3=e_4,v_4=-e_1,v_5=-e_2,v_6=-e_3,v_7=-e_4,v_8=e_1+e_2+e_3+e_4,v_9=-e_1-e_2-e_3-e_4,v_{10}=e_1$ & $V(v_1,v_4)$ & -5 \\ \hline
\end{tabular}
\noindent
\begin{tabular}{  |c| p{7cm} | c| c| }
\hline
& \centering{Primitive \ Vectors} & Surface & $ch_2(T_X)\cdot S$  \\ 
\hline
$118$ & $v_1=e_2, v_2=e_3,v_3=e_4,v_4=-e_1,v_5=-e_2,v_6=-e_3,v_7=-e_4,v_8=e_1+e_2+e_3+e_4,v_9=e_1,$ & $V(v_1,v_4)$ & $-\displaystyle\frac{5}{2}$ \\ \hline

$124$ & $v_1=e_1, v_2=e_2, v_3=-e_1-e_2,v_4=-e_1+e_4, v_5=e_1-e_3-e_4, v_6=e_3, v_7=e_4, v_8=e_1+e_2-e_3-e_4, v_9=-e_1-e_2+e_3,$ & $V(v_1,v_7)$ & $-4$ \\ \hline
\end{tabular}\\ \\

\begin{rk}\em{There is a misprint in \cite[Proposition 3.4.1]{bat} concerning the primitive relations for the toric Fano 4-fold 108.}
\end{rk}
      
Finally, we use a database provided by Øbro in \cite{obro} and the program Maple to prove the theorem in dimension 5 and 6. 
We implement a function which, given a smooth toric variety  computes the product of the second Chern character with an invariant surface. The code is provided in the Chapter \ref{apendix}.

\begin{rk} \em{ We have concluded that there exists only one toric 2-Fano 4-fold. However, using proposition 3.2 and computations with the program Maple we see that there exist toric Fano 4-folds that have nef second Chern character (i.e., $ch_2(T_X)\cdot S\geq 0$ for every surface $S\subset X$). They are:} \\
$\mathbb{P}^4, B_1, B_2, B_3, B_4, C_4, D_1, D_2, D_3, D_5, D_6, D_8, D_9, D_{12}, D_{13}, D_{15}, L_1, L_2, L_3, L_4,$ \\
$L_5, L_6, L_7, L_8, L_9$.
\end{rk}

\section{Higher Dimensions}\label{dimensaomaior}

\subsection{Strategy}

Our goal is to classify smooth toric 2-Fano $n$-folds. Our strategy is to investigate what happens with the second Chern character of a toric variety $X$ when we run the Minimal Model Program for $X$. 
Since in the toric case  the MMP ends with a Mori fiber space, we start investigating the second Chern character of a Mori fiber space. There are examples of singular 2-Fano varieties with Picard number bigger than 1 with structure of 
Mori fiber space, see for instance \ref{ex}. However, we will show that is possible to find a birational model $X'$ of $X$ with structure of Mori fiber space, such that the general fibers are 
projective spaces. Then, in Section \ref{fibrageral}, we will see that a such variety  
 $X'$ cannot be 2-Fano.
 
\begin{pr} Let $X$ be a smooth projective toric variety. Then, making suitable choices of extremal rays, we can run MMP for $X$ to obtain a birationally equivalent variety $X'$ and a Mori fiber space 
$\pi:X'\rightarrow Y$ whose general fiber is $\mathbb{P}^k$.
\end{pr}
\begin{proof} By \cite[Corollary 6]{fu} there is an open subset $X_0\subset X$ containing the torus, a smooth variety $Y_0$ and a $\mathbb{P}^k$-bundle structure $\pi_0:X_0\rightarrow Y_0$.
Let $Y$ be the closure of $Y_0$ in Chow$(X)$. Since $\pi_0$ is flat and proper, there is a variety $\mathcal{U}$, called universal cycle over $Y$, and universal morphisms $e:\mathcal{U}\rightarrow X$, $\pi:\mathcal{U}\rightarrow Y$ such that $e^{-1}|_{X_0}$ is an isomorphism and the diagram below is commutative.
$$\xymatrix{
X_0 \ar[d]_{\pi_0} \ar@{^{(}->}[r] & \mathcal{U}\ar[d]_{\pi} \ar[r]^e & X\\ 
Y_0 \ar@{^{(}->}[r] 
& Y \ar@{^{(}->}[r]& \mbox{Chow}(X) 
}$$ 
Consider a very ample effective divisor $A_Y$ on $Y$. Set $A:=e_*(\pi^*(A_Y))$. Since $A_Y$ is very ample, the linear system $|A|$ has no base points on $X_0$. Let $l$ be a curve in a fiber of $\pi_0$. Then, $A\cdot l=0$ and $K_X\cdot l=K_{\mathbb{P}^k}\cdot l=-k-1<0$. It follows that there is an extremal ray $R$ of $\mbox{NE}(X)$ such that $K_X\cdot R<0$ and $A\cdot R\leq 0$. We consider $f_R:X\rightarrow X_1$ the contraction of $R$.
If $f_R$ is a Mori fiber space, then there is a curve $C\subset X$ such that its class is in the ray $R$ and $C\cap X_0\neq\emptyset$. Since we choose $R$ satisfying $A\cdot R\leq 0$ and $A_Y$ very ample, we must have that $C$ is a curve in a fiber of $\pi_0$.  Thus, $R=\mathbb{R}_{\geq 0}[l]$. Since a fiber space depends only on its contracted curves, we have $f_R|_{X_0}=\pi_0$ and 
it proves the theorem in this case. Suppose that $f_R$ is a birational morphism and let $E$ be the exceptional locus of $f_R$. Then, $E\cap X_0=\emptyset$. Otherwise, as before, $R$ would be generated by $[l]$ and therefore $X_0$ would be contained in $E$. Since $f_R$ is an isomorphism outside $E$, the $\mathbb{P}^k$-bundle structure of $X_0$ is preserved in $X_1$. Continuing this process we prove the theorem.
\end{proof}

\subsection{Second Chern Computation for Mori Fiber Spaces} \label{fibrageral}

Now, we study the second Chern character of a variety associated to a Cayley-Mori polytope. Throughout this section we will say that a $\mathbb{Q}$-factorial toric variety 
(non necessarily smooth) is Fano (resp. 2-Fano) if $-K_X$ (resp. $-K_X$ and $ch_2(X)$) is positive.

\begin{lema}\label{fano} Let $X$ be the toric variety given by a  Cayley-Mori lattice polytope $P=P_0*...*P_k\subset\mathbb{R}^{1+k}$. Suppose that the polytopes $P_i's$ are
intervals in $\mathbb{R}$ and that the natural projection $\pi:\mathbb{R}^{1+k}\rightarrow \mathbb{R}^{k}$ project $P$ on a polytope $Q:=conv(0,se_1,...,se_k)$ for some positive integer $s$. Then, $ch_2(T_X)$ is not positive.
\end{lema}
\begin{rk}\em{
 It follows from Theorem \ref{mori} that the hypothesis imposed in the lemma above is equivalent to saying that $X$ admits a structure of Mori fiber space $f:X\rightarrow\mathbb{P}^1$ whose the general 
fiber is  $\mathbb{P}^k$.} 
\end{rk}

\begin{proof} Recall from Section \ref{Cayley} that the fan $\Sigma$ determined by $P$ is of the following form:\\

Write $P_i=[-b_{iu},b_{iv}]$. Then $D_{P_i}=b_{iu}V(1)+b_{iv}V(-1)=\mathcal{O}_{\mathbb{P}^ 1}(b_i)$ where $b_i:=b_{iu}+b_{iv}$. We may assume that $0<b_0\leq b_1\leq ...\leq b_k$. 
The primitive vectors of $\Sigma$ are: \\

$u:=d_u(1,b_{1u}-b_{0u},...,b_{ku}-b_{0u}), v:=d_v(-1,b_{1v}-b_{0v},...,b_{kv}-b_{0v}), e_0:=(0,-1,...,-1), e_1:=(0,1,0,...,0),...,e_k:=(0,0,...,1)$, 
where $$d_u=\displaystyle\frac{s}{gcd(s,b_{1u}-b_{0u},...,b_{ku}-b_{0u})}$$ and 
$$d_v=\displaystyle\frac{s}{gcd(s,b_{1v}-b_{0v},...,b_{kv}-b_{0v})}.$$\\

The maximal cones of $\Sigma$ are: for $i\in\{0,1,...,k\}$, $\sigma_i:=\langle u,e_0,...,\widehat{e_i},...,e_k\rangle$ and $\tau_i:=\langle v,e_0,...,\widehat{e_i},...,e_k\rangle$.   

Consider the invariant surface $S:=V(e_2,...,e_k)\subset X$. We will use Propositions \ref{inter} and \ref{2chern} to compute $ch_2(X)\cdot S$.\\

Set $D_u=V(u)$, $D_v=V(v)$ and $D_i=V(e_i)$, for $i\in\{0,1,...,k\}$.\\
Since $div(\chi^{(1,0,...,0)})=d_uD_u-d_vD_v$, we have that $D_u^2\cdot S=D_v^2\cdot S=0$.\\ 

 We want to compute $D_i\cdot S$ for $i\in\{1,...,k\}$.\\

$D_0^2\cdot S=D_0\cdot V(e_0,e_2,...,e_k)=(D_0+div(\chi^{e_1}))\cdot V(e_0,e_2,...,e_k)=$\\ 
$[(b_{1u}-b_{0u})D_u+(b_{1v}-b_{0v})D_v]\cdot V(e_0,e_2,...,e_k)=\displaystyle\frac{(b_{1u}-b_{0u})}{d_u}+\displaystyle\frac{(b_{1v}-b_{0v})}{d_v}$. \\

$D_1^2\cdot S=D_1\cdot V(e_1,e_2,...,e_k)=(D_1+div(\chi^{-e_1})\cdot V(e_1,e_2,...,e_k)=$\\
$[-(b_{1u}-b_{0u})D_u-(b_{1v}-b_{0v})D_v]\cdot V(e_1,e_2,...,e_k)=-\displaystyle\frac{(b_{1u}-b_{0u})}{d_u}-\displaystyle\frac{(b_{1v}-b_{0v})}{d_v}$.\\

If $k=1$ then $2ch_2(T_X)\cdot S=(D_0^2+D_1^2+D_u^2+D_v^2)\cdot S=0$. Suppose $k>1$. 
Since $$D_2\simeq (D_2+div(\chi^{(-(b_{2v}-b_{0v}),0,-1,0,...,0)})=D_0-d_u[(b_{2v}-b_{0v})+(b_{2u}-b_{0u})]D_u,$$ we have:\\

$D_2^2\cdot S= [D_0^2-2d_u((b_{2v}-b_{0v})+(b_{2u}-b_{0u}))D_u\cdot D_0]                                                                       
\cdot V(e_2,...,e_k)=\displaystyle\frac{(b_{1u}-b_{0u})}{d_u}+\displaystyle\frac{(b_{1v}-b_{0v})}{d_v}-2[(b_{2v}-b_{0v})+(b_{2u}-b_{0u})]=$\\
$=\displaystyle\frac{(b_{1u}-b_{0u})}{d_u}+\displaystyle\frac{(b_{1v}-b_{0v})}{d_v}-2(b_{2}-b_{0}).$\\

Similarly, $D_i^2\cdot S=\displaystyle\frac{(b_{1u}-b_{0u})}{d_u}+\displaystyle\frac{(b_{1v}-b_{0v})}{d_v}-2(b_{i}-b_{0})
, i=2,...,k.$ \\

It follows that  $$2ch_2(T_X)\cdot S=\displaystyle\sum_{i=2}^{k}\Big[\displaystyle\frac{(b_{1u}-b_{0u})}{d_u}+\displaystyle\frac{(b_{1v}-b_{0v})}{d_v}-2(b_{i}-b_{0})\Big]\leq\displaystyle\sum_{i=2}^{k}\big[(b_1-b_0)-2(b_{i}-b_{0})\big] $$  \\

and then $ch_2(T_X)\cdot S\leq 0$.
\end{proof}
\vspace{0.4cm}

\begin{rk}\label{sato} \em{ Notice that if $k>1$ then $ch_2(T_X)\cdot S=0$ if and only if $b_0=b_1=...=b_k$. In other words, $ch_2(T_X)\cdot S=0$  if and only if the polytopes $P_0,...,P_k$ have the same size. In this case, it follows from Remark \ref{nef} that $ch_2(T_X)$ is nef. More generally, making a similar computation as in the previous lemma, we conclude that if  $P_i$ is an $n$-dimensional simplex (i.e., the variety defined by $P_i$ is $\mathbb{P}^n$), then $ch_2(T_X)$ is nef if and only if $P_0,...,P_k$ have the same size. This means that, up to translations, there is an integer $b$ such that $P_i=b\Delta_n$ for every $i\in\{0,...,k\}$. }
\end{rk}

\begin{lema} \label{reducao} Let $X$ be the toric variety given by a  Cayley-Mori lattice polytope $P=P_0*...*P_k\subset\mathbb{R}^{n+k}$  and let $f:X\rightarrow Y$ be the associated  fibration. If $C$ is an invariant curve on $Y$ and $X_C:=f^{-1}(C)$ then $ch_2(T_X)\cdot X_C=ch_2(T_{X_C})$.
\end{lema}

\textit{Proof}. Let $C:=V(\tau)$ with $\tau=\langle v_1,...,v_{n-1}\rangle\subset\Sigma_Y$. 

It follows from Proposition \ref{pr1} and the description of the fan $\Sigma_P$ that the invariant variety $X_C:=f^{-1}(C)$ is given by $V(\tilde{\tau})$ with $\tilde{\tau}=\langle \tilde{v_1},...,\tilde{v_{n-1}}\rangle$
and $f^*(V(v_i))=V(\tilde{v_i})$.

Our aim is to compute the restriction of $ch_2(T_X)$ to $X_C$. Let $D$ be a prime $T$-invariant divisor on $X$ corresponding to the primitive vector $v\in\Sigma_X(1)$. If we restrict $D$ to $X_C$ 
we have three possibilities:

\begin{enumerate}

 \item $D\cdot X_C=0$. This happens exactly when $\langle v,\tilde{\tau}\rangle$ is not contained in any cone of $\Sigma_X$.
 
 \item $D\cdot X_C\neq 0$ and $X_C$ is not contained in $D$. This happens exactly when $\langle v,\tilde{\tau}\rangle$ is  contained in some cone of $\Sigma_X$ but $v\notin \tilde{\tau}$. In this case, $D\cdot X_C$ is the prime invariant divisor on $X_C$ corresponding to the image of $v$ in the fan $Star(\tau)$ of $X_C$. 
Conversely, every prime invariant divisor on $X_C$ appears in this way.
 
 \item $D\cdot X_C\neq 0$ and $X_C$ is contained in $D$. This happens when $v\in\tilde{\tau}$. In this case, $v=\tilde{v_i}$ for some $i\in\{1,...,n-1\}$. Thus, $D$ is the pull-back of $V(v_i)$. Since $C$ is a curve $D^2\cdot X_C=0$.
\end{enumerate}
Now, the lemma follows from Proposition \ref{2chern}. \qed \\

\begin{teo}\label{ch}
Let $X$ be a $\mathbb{Q}$-factorial projective toric variety and suppose that $f_R:X\rightarrow Y$ is a contraction of  fibering type of a ray $R$ of the Mori cone $\mbox{NE}(X)$. If the general 
fiber is a projective space, then $X$ is not 2-Fano  unless $X$ is a projective space.
\end{teo}

\textit{Proof}. From Theorem \ref{mori} we conclude that $X$ is the toric variety associated to a Cayley-Mori lattice polytope $P=P_0*...*P_k\subset\mathbb{R}^{n+k}$ and $f_R$ is induced by 
the natural projection $\pi:\mathbb{Z}^{n+k}\rightarrow \mathbb{Z}^{k}$.  If the general 
fiber of $f_R$ is $\mathbb{P}^k$ then the projection of $P$ by $\pi_{\mathbb{R}}$ is a polytope of the form $Q:=conv(0,se_1,...,se_k)$ for some positive integer $s$ and $\{e_1,...,e_k\}$ a basis for $\mathbb{Z}^k$.
Suppose that $X$ is not a projective space. Then $Y$ is positive dimensional. If $C$ is an invariant curve in $Y$ then $X_C=f_R^{-1}(C)$ is given by a Cayley-Mori polytope $l_0*...*l_k$ where $l_i\subset P_i$ is an edge of $P_i$. So, the theorem follows from Lemmas \ref{fano} and \ref{reducao}. \qed

\begin{cor}\label{pic2} If $X$ is a smooth projective toric variety with $\rho(X)\leq 2$ then $X$ is not 2-Fano unless $X$ is $\mathbb{P}^n$. 
\end{cor}

\begin{proof} By Corollary \ref{picard1} $\rho(X)=1$ implies $X$ is a projective space. Otherwise, by \cite{classificar} $X$ is a projective bundle over a positive dimensional base and by the previous theorem  $X$ is not 2-Fano. \end{proof}

\begin{rk} \em{ If $X$ is a $\mathbb{Q}$-factorial projective toric variety with $\rho(X)=1$ then $\mbox{NE}(X)$ and $\mbox{NE}_2(X)$ are unidimensional. Thus, $X$ is 2-Fano.}
\end{rk}

 If we do not require smoothness there are toric varieties with Picard number 2 that are 2-Fano as in example \ref{ex}. This example also shows that the hypothesis on the general fiber of $f_R$
made in Theorem \ref{ch} is necessary. In order to explain this example we need some preliminaries on Mori fiber spaces of Picard number $2$. 

\begin{pr} \label{hi}
Let $X$ be a toric variety corresponding to a Cayley-Mori polytope $P=P_0*...*P_k$ where each $P_i$ is an n-dimensional simplex. Then the dimension of $NE_2(X)$ is at most 3. More precisely,

$$\mbox{dim } NE_2(X)=
\left\{
\begin{array}{ll}
  1  & \text{if } n=k=1,\\
    2  & \text{if } n=1 \mbox{and } k>1, \mbox{or } n>1 \mbox{and } k=1 ,\\
    3  & \text{otherwise}.\\
\end{array}
\right.
$$
\end{pr}
\begin{proof} As discussed in Section \ref{Cayley}, writing $D_{P_i}:=\displaystyle\sum_{j=0}^{n}a_{i_j}D_j$ for $i=0,...,k$, the primitive vectors of the fan $\Sigma_P$ are:

 $u_i:=s_i(A^*)^{-1}(e_i)$ for $i\in\{0,...,k\}$ and $\displaystyle \tilde{v}_j:=d_j[v_j+\sum_{i=1}^k(a_{i_j}-a_{0_j})(A^*)^{-1}(e_i)],$ for every $j\in\{0,1,...,n\}$.
 
Denote the divisors $V(u_i)$ and $V(\tilde{v}_j)$ by $D_{u_i}$ and $D_{\tilde{v}_j}$ respectively. 

For each $i\in\{0,...,k\}$ and $j\in\{0,1,...,n\}$  there exist positive integers $c_i$ and $c_j$ such that $d_ic_iD_{\tilde{v}_i}\simeq d_jc_jD_{\tilde{v}_j}$ 
(to see it compute $div(\chi^u)$ for $u\in\mathbb{Z}^n\cap (\langle v_0,...,\widehat{v_i},...,\widehat{v_j},...,v_n\rangle)^{\perp}$). 

After renumbering the polytopes $P_i's$ we may assume that  $$\displaystyle\sum_{j=0}^{n}\frac{a_{0_j}}{c_j} \leq \displaystyle\sum_{j=0}^{n}\frac{a_{1_j}}{c_j}\leq ...\leq \displaystyle\sum_{j=0}^{n}\frac{a_{k_j}}{c_j}.$$

For each $i\in\{0,...,k\}$ we have $$div(\chi^{A(e_i)})=-s_0D_{u_0}+s_iD_{u_i}+\displaystyle\sum_{j=0}^{n}d_j(a_{i_j}-a_{0_j}) D_{\tilde{v_j}}.$$

From this it follows that for every  $i\in\{0,...,k\}$ we have:

$$
\left\{
\begin{array}{ll}
 D_{u_0}= & \frac{s_i}{s_0}D_{u_i}+\frac{d_1c_1}{s_0}\displaystyle\sum_{j=0}^{n}\frac{(a_{i_j}-a_{0_j})}{c_j} D_{\tilde{v_1}},\\
   D_{u_1}= & \frac{s_i}{s_1}D_{u_i}+\frac{d_1c_1}{s_1}\displaystyle\sum_{j=0}^{n}\frac{(a_{i_j}-a_{1_j})}{c_j} D_{\tilde{v_1}}\mbox{ and}\\
    D_{u_2}=  & =\frac{s_i}{s_2}D_{u_i}+\frac{d_1c_1}{s_2}\displaystyle\sum_{j=0}^{n}\frac{(a_{i_j}-a_{2_j})}{c_j} D_{\tilde{v_1}}.\\
\end{array}
\right.
$$

This means that for every $j\in\{0,...,n\}$:
 $$D_{u_0}\in \mbox{cone}(D_i, D_{\tilde{v_j}} )\mbox{ for } i\in\{0,...,k\},$$
  $$D_{u_1}\in \mbox{cone}(D_i, D_{\tilde{v_j}} )\mbox{ for } i\in\{1,...,k\}\mbox{ and}$$
 $$D_{u_2}\in \mbox{cone}(D_i, D_{\tilde{v_j}} )\mbox{ for } i\in\{2,...,k\}.$$  
  
  We conclude that every invariant surface can be written as a non negative linear combination of 
$S_1:=V(u_3,...,u_k,\tilde{v_0},...,\tilde{v}_{n-1})$, $S_2:=V(u_2,...,u_k,\tilde{v_0},...,\tilde{v}_{n-2})$ and $S_3:=V(u_1,...,u_k,\tilde{v_0},...,\tilde{v}_{n-3})$. 
 \\

 If $n=k=1$ then $X$ is a surface. If $n=1$ and $k>1$ then  $S_3$ does not exists and $NE_2(X)$ has $\{S_1,S_2\}$ as a basis. If $n>1$ and $k=1$ then $S_1$ does not exists and $\{S_2,S_3\}$ form a basis of $NE_2(X)$. Otherwise, $NE_2(X)$ has dimension 3 and $\{S_1,S_2,S_3\}$ is a basis.
 \end{proof}
\begin{rk}\label{nef} \em{ Let $X$ be as in Proposition \ref{hi}. Using the same notation as the proof of Proposition \ref{hi}, we conclude that the Mori cone of $X$ is generated by the curves  $C_1:=V(\tilde{v_1}, ...,\tilde{v_n},u_2,...,u_k)$ and $C_2:=V(\tilde{v_2},...,\tilde{v_n},u_1,...,u_k)$.

It is straightforward to check,  using \ref{inter}, \ref{2chern} and the relations above, that $ch_1(X)\cdot C_1=-K_X\cdot C_1$, $ch_2(X)\cdot S_1$ and  $ch_2(X)\cdot S_3$ are 
 always positive. So, the positivity of $ch_1(X)$ and $ch_2(X)$ depends only on the values $-K_X\cdot C_2$ and $ch_2(X)\cdot S_2$.}
\end{rk}

Now, we give an example of a singular 2-Fano variety with Picard number 2.
  
\begin{ex} \label{ex} \em{Let $X$ be the singular toric variety defined by the Cayley-Mori polytope $$Q:=conv\big((0,0,0),(1,0,0),(0,0,1),(2,0,1),(0,-2,1),(2,-2,1)\big).$$ 
The primitive vectors of $\Sigma_Q$ are: $u_0=(0,0,-1), u_1=(0,1,2), u_2=(0,-1,0),$  $\tilde{v}_0=~(1,0,0)$ and $\tilde{v}_1= (-1,0,1)$. 

Consider $C_2:=V(u_1,u_2)$ and $S_2=V(u_2)=D_{u_2}$. Using the relations  $D_{u_1}+div(\chi^{-e_2})=D_{u_2}$, $D_{u_0}+div(\chi^{e_3})=2D_{u_1}+D_{\tilde{v_1}}$ and the Proposition \ref{inter}, we get:
$$-K_X\cdot C_2=(D_{u_0}+D_{u_1}+D_{u_2}+D_{\tilde{v_0}}+D_{\tilde{v_1}}V(u_1,u_2))=1.$$

$$ch_2(X)\cdot S_2=\frac{1}{2}((D_{u_0})^2+(D_{u_1})^2+(D_{u_2})^2+(D_{\tilde{v_0}})^2+(D_{\tilde{v_1}})^2)V(u_2))=\displaystyle\frac{1}{4}.$$ }\\
\end{ex}

We have studied the second Chern character of a Mori fiber space, which is the variety obtained in the last step of the MMP.
The next step is to understand what happens to it when we apply a step of the MMP which is a divisorial contraction or a flip. 
This seems to be a hard problem. Next, we give some partial results in this direction.

\begin{teo} Let $X$ be a smooth toric variety and $f_R: X \rightarrow Y$ an extremal contraction of a ray $R$ of $\mbox{NE}(X)$ of divisorial type. 
It follows from Corollary \ref{exc}  that $Exc(f_R)=\mathbb{P}_{Z}(\mathcal{E})$ where $Z\subset Y$ is an invariant subvariety and $\mathcal{E}$ is a decomposable vector bundle on $Z$. 
Suppose that $Z=\mathbb{P}^1$ and $\mathcal{E}=\mathcal{O}_{\mathbb{P}^1}\oplus \mathcal {O}_{\mathbb{P}^1}(b_1)\oplus...\oplus \mathcal{O}_{\mathbb{P}^1}(b_k)$ with 
$0\leq b_1\leq ...\leq b_k$. Then $X$ is not 2-Fano.
\end{teo}

\begin{proof} Suppose that $X$ is 2-Fano. The minimal vectors of the fan of $E:=Exc(f_R)$ are: \\

 $u=(1,b_1,...,b_k), v=(-1,0,...,0)$, $e_0=(0,-1,-1,...,-1),e_1=(0,1,0,...,0), ..,$ $e_k=(0,0,...,1)$.\\

It is well known that $K_E=\pi^*(K_{\mathbb{P}^1}+\mathcal {O}_{\mathbb{P}^1}(b_1)+...+ \mathcal{O}_{\mathbb{P}^1}(b_k) )-(k+1)\xi,$ 
where $\pi:\mathbb{P}_{\mathbb{P}^1}(E)\rightarrow \mathbb{P}^1$ is the canonical projection and $\xi=\mathcal{O}_{\mathbb{P}(\mathcal{E})}(1)$ 
is the tautological line bundle.
By adjunction formula we have $-K_X|_E=-K_E+E|_E$.
Since~$\xi=V(e_0)$ (see Lema 3,\cite{sandra}), we obtain 
$$-K_X|_E=-\pi^*(\mathcal{O}_{\mathbb{P}^1}(-2)+\mathcal{O}_{\mathbb{P}^1}(b_1)+...+\mathcal{O}_{\mathbb{P}^1}(b_k))+(k+1)V(e_0)+E|_E.$$
Since $E$ is the exceptional locus of $f_R$, we can write $E|_E=a\xi +\pi^*(\mathcal{O}_{\mathbb{P}^1}(s))$, for some integer $a<0$ and $s \in \mathbb{Z}$ (see \cite{matsuki}. lemma 14.1.7).\\

If we take $C:=V(u,e_2,...,e_k)\subseteq E$ a curve contained in a fiber of $f_R|_E$, we conclude:

$$-K_X|_E\cdot C=\big[(k+1)V(e_0) + E|_E\big]\cdot C=\big[(k+1)V(e_0)+aV(e_0)\big]\cdot C >0$$ 
$$\Rightarrow a>-(k+1).  \ \ \ \ \ (1)$$

If we take $C':=V(e_1,...,e_k)\subseteq E$, we have:

$$-K_X|_E\cdot C'=2-b_1-...-b_k+\big[(k+1)V(e_0)+aV(e_0)\big]\cdot C' +\pi^*(\mathcal{O}_{\mathbb{P}^1}(s))\cdot C'$$
$$\Rightarrow -K_X|_E\cdot C'=2-b_1-...-b_k+s>0$$
$$\Rightarrow s\geq b_1+...+b_k-1.\ \ \ \ \ (2)$$

Now, we will compute the intersection number between $ch_2(T_X)$ and the surface $S:=V(e_2,...,e_k)\subseteq E$.\\

 From exact sequence $0\rightarrow T_E \rightarrow T_X|_E\rightarrow E|_E\rightarrow 0$ it follows that:
 $$ch_2(T_X|_E)=ch_2(T_E)+(E|_E)^2.$$

 Using the formula $ch_2(T_E)=\displaystyle\frac{(k+1)}{2}\xi^2-\pi^*c_1(E)\xi+\pi^*(ch_2(T_{\mathbb{P}^1})-ch_2(E))$, given in \cite[4.1]{jason}, we conclude:\\

$$ch_2(T_E)\cdot S=\Big[\frac{(k+1)}{2} V(e_0)^2-\pi^*\big( \mathcal{O}_{\mathbb{P}^1}(b_1)+...+\mathcal{O}_{\mathbb{P}^1}(b_k)\\big)\cdot V(e_0)\Big]\cdot S=$$
 $$=\Big[\frac{(k+1)}{2} V(e_0)-\pi^*( \mathcal{O}_{\mathbb{P}^1}(b_1)+...+\mathcal{O}_{\mathbb{P}^1}(b_k))\Big]\cdot V(e_0,e_2,...,e_k)=$$
$$=\displaystyle\frac{(k+1)}{2} \big(V(e_0)+div(\chi^{e_1})\big)\cdot V(e_0,e_2,...,e_k)-b_1-...-b_k=$$
$$=\displaystyle\frac{(k+1)}{2}\big(V(e_1)+b_1V(u)\big)\cdot V(e_0,e_2,...,e_k)-b_1-...-b_k=$$
$$=\displaystyle\frac{(k+1)}{2}b_1-b_1-...-b_k.$$

If $X$ is 2-Fano then, $ch_2(T_X|_E)\cdot S=ch_2(T_E)\cdot S+(E|_E)^2\cdot S>0$\\$ \Rightarrow (E|_E)^2\cdot S>-\displaystyle\frac{(k+1)}{2}b_1+b_1+...+b_k>0$.\\

On the other hand,\\
 $(E|_E)^2\cdot S=[a^2\xi^2+2a\pi^*(\mathcal{O}_{\mathbb{P}^n}(s))\cdot \xi]\cdot V(e_2,...,e_k)=a^2b_1+2as$

$$\Rightarrow a(ab_1+2s)>0\Rightarrow s<-\displaystyle\frac{a}{2}b_1\stackrel {\textrm{(1)}}{<}
\displaystyle\frac{(k+1)}{2}b_1.$$

It follows that $$b_1+...+b_k-1\stackrel {\textrm{(2)}}{\leq}s\leq \displaystyle\frac{(k+1)}{2}b_1-1 \ \ \ \ \ (3)$$
This is possible only if $k=1$. But this implies that dim $X=3$. On the other hand we know that the only toric 2-Fano 3-fold is $\mathbb{P}^3$. 
\end{proof}

\begin{pr} Let $X$ be a smooth projective toric variety and $Z\subset X$ a (smooth) invariant subvariety of $X$ of codimension $c$. Denote by $\pi:\tilde{X}\rightarrow X$ the 
blowing up of $X$ along of $Z$. If $\tilde{X}$ is 2-Fano then $Z$ is Fano.
\end{pr}
\begin{proof} Denote by $E$ the exceptional divisor and $f:=\pi{|_E}:E\rightarrow Z$.Then $E=\mathbb{P}(N^{\vee})$, where $N=N_{Z|X}$. Given an invariant (rational) curve $C\subset Z$ there are integers $a_1\leq\cdots\leq a_c$ such that $N^{\vee}{|_C}=\mathcal{O}_C(a_1)\oplus\cdots\oplus\mathcal{O}_C(a_c)$. 
Let $l\subset E$ be the minimal ``pull back" curve of $E$, that is, $f(l)=C$ and $\xi\cdot l=a_1$, where $\xi=\mathcal{O}_E(1)$. It is well known that $K_{\tilde{X}}=\pi^*K_X+(c-1)E$. This implies that $$-K_X\cdot C=-K_{\tilde{X}}\cdot l+(c-1)E\cdot l \ \ \ \ \ \ \ (1)$$
From the exact sequence $0\rightarrow T_Z\rightarrow T_{X|_Z}\rightarrow N_{Z|X}\rightarrow 0$ we conclude that $$-K_Z\cdot C=-K_X|_Z-\mbox{det}(N)\ \ \ \ \ \ \ (2)$$
Putting (1) and (2) together we have that
$$-K_Z\cdot C=-K_{\tilde{X}}\cdot l+(c-1)E\cdot l+a_1+\cdots +a_c=-K_{\tilde{X}}\cdot l-(c-1)a_1+a_1+\cdots +a_c$$
$$\Rightarrow -K_Z\cdot C\geq -K_{\tilde{X}}\cdot l+a_c.$$
In order to prove that $Z$ is Fano we will show that $a_c\geq 0$ if $\tilde{X}$ is 2-Fano. We can write $ch_2(T_{\tilde{X}})=\pi^*ch_2(T_X)+\displaystyle\frac{c+1}{2}E^2-j_*\Big(f^*\big(c_1(N)\big)\Big)$, where $j:E\hookrightarrow \tilde{X}$ is the natural inclusion (see Lemma \ref{blowup}).

Consider the invariant surface $S:=\mathbb{P}_C(\mathcal{O}_C(a_1)\oplus\mathcal{O}_C(a_2))\subset E$. By projection formula $\pi^*ch_2(T_X)\cdot S=0$. Since $E^2=\xi$ we have that $E^2\cdot S=a_1+a_2$.
It follows that:

$$
\begin{array}{ll}
 ch_2(T_{\tilde{X}})\cdot S & = \ \displaystyle\frac{(c+1)}{2}(a_1+a_2)-E\cdot \Big(f^*\big(c_1(N)\big)\Big) \\
  & = \ \displaystyle\frac{(c+1)}{2}(a_1+a_2)-c_1(N)\cdot f_*(E\cdot S) \\
  & = \ \displaystyle\frac{(c+1)}{2}(a_1+a_2)+c_1(N)\cdot C \\
  & = \ \displaystyle\frac{(c+1)}{2}(a_1+a_2)-(a_1+\cdots +a_z).
\end{array}
$$
If $\tilde{X}$ is 2-Fano then we have $0\leq\displaystyle\frac{(c+1)}{2}(a_1+a_2)-(a_1+\cdots +a_c)\leq a_2\leq\cdots\leq a_c.$
\end{proof}

\chapter{Appendix: Maple Code}\label{apendix}

In this appendix we provide the code used in the program Maple to compute $ch_2(T_X)\cdot S$ for $X$ a smooth toric varieties and $S$ an invariant surface. The code is based in the theory given in \ref{secondchern}.
The reader who wishes to obtain the file in Maple extension can access the webpage \url{http://w3.impa.br/~edilaine/ }\\

Input: The primitive vectors and maximal cones which determine the toric variety $X$; an invariant surface of $X$.\\

Output: $ch_2(T_X)\cdot S$. 

\begin{verbatim}

restart:
with(LinearAlgebra):
with(RandomTools):
with(RegularChains):
with(combinat):
with(Statistics): 
with(ArrayTools):

A:=proc(L,m)
local i,j;
for j from 1 to Count(L) do i:=union(choose(convert(L[j],set),m),i) 
end do;
return i;  
end proc:

f:=(n,E,x)->if(convert(x,set) in A(E,n-2),convert(Concatenate(2,
Vector_row(x),Vector_row([0])),list),[seq(0, i = 1 .. (n-1))]):

g:= (n, E, y)->if(convert(y, set) in A(E, n-1),y,  
[seq(0, i = 1 .. (n-1))]): 

k:= (v,z,m)->GramSchmidt([seq(v[z[i]], i = 2 .. m), v[z[1]]]):

h:= (n,v,z,r)->if(Count(k(v,z,r)) < r, Vector[column](n,i->0), 
-k(v,z, r)[r]*DotProduct(v[z[1]],
k(v,z,r)[r])^-1): 

e:=z->convert(convert(z, set)\{z[1]},list): 

z:=(n,E,y)->if(is(y[1] in [seq(y[r],r=2..(n-1))])=false or
f(n,E,convert(Vector_column(n-2,i->y[i+1]),list))=
[seq(0,i=1..(n-1))],g(n,E,y),[seq(0,i=1..(n-1))]):

q:=(n,v,E,y)->if(is(y[1] in [seq(y[r],r=2..(n-1))])=false
or f(n,E,convert(Vector_column(n-2,i->y[i+1]),list))=
[seq(0,i=1..(n-1))],g(n,E,y),seq([z(n,E,convert(Concatenate(1,r,
Vector_column(n-2,i->y[i+1])),list)),DotProduct(v_r,h(n,v,
convert(Concatenate(2,y_1,convert(e(y),vector)),list),n-2))],
r=1..Count(v))):

u:=(n,E,w)->if(convert(w,set) in A(E,n),1,0):

t:=(n,v,E,w)->if(is(w[1] in [seq(w[r],r=2..n)])=false or
g(n,E,convert(Vector_column(n-1,i->w[i+1]),list))=
[seq(0,i=1..(n-1))],u(n,E,w),add(DotProduct(v_r,h(n,v,
convert(Concatenate(2,w_1,convert(e(w),vector)),list),n-1))*
u(n,E,convert(Concatenate(1,r,Vector_column(n-1,
i->w[i+1])),list)),r=1..Count(v))):

d:=(n,v,E,y)->if(q(n,v,E,y)=g(n,E,y),t(n,v,E,
convert(Concatenate(1,y_1,Vector_column(n-1,i->y_i)),list)),
add(t(n,v,E,convert(Concatenate(1,y_1,convert(q(n,v,E,y)[p][1],
Vector)),list))*q(n,v,E,y)[p][2],p=1..Count(v))):

c:=(v,E,x)->add(d(Count(v[1]),v,E,convert(Concatenate(2,s,
convert(x,vector)),list)),s=1..Count(v)):
\end{verbatim}


\begin{thebibliography}{10}
\bibitem{alicia} A.~Dickenstein, S.~Di~Rocco, R.~Piene, 
 Classifying smooth lattice polytopes via toric fibrations, 
Adv. Math. \textbf{222} (2009), 240-254.
\bibitem{jason} A. J. de Jong and J. Starr, A note on Fano manifolds whose second Chern~character is positive, pre-print  math.AG/0602644v1, (2006).
\bibitem{starr} A. J. de Jong and J. Starr, Higher Fano manifolds and rational surfaces, Duke Math. J. \textbf{139} (2007), no. 1, 173–183.
\bibitem{fu} B.~Fu, J.-M. Hwang,  Minimal rational curves on complete toric manifolds,  arXiv: 0912.1638, (2009).
\bibitem{carol} C. Araujo and A-M. Castravet, 2-Fano 3-folds, preprint.
\bibitem{ana} C. Araujo and A-M. Castravet, Polarized minimal families of rational curves and higher Fano manifolds, to appear in American Journal of Mathematics, preprint math.AG/0906.5388v1, (2009).
\bibitem{carolina} C. Araujo, The cone of pseudo-effective divisors of log varieties after Batyrev, Math. Z. \textbf{264} (2010), no. 1,179-193.
\bibitem{scaling} C. Birkar, P. Cascini, C. Hacon, and J. McKernan, Existence of minimal models for
varieties of log general type, J. Amer. Math. Soc. 23 (2010), no. 2, 405–468. 
\bibitem{luca} C. De Concini, C. Procesi, Topics in Hyperplane Arrangements, Polytopes and
Box-Splines. Springer, 2010. 
\bibitem{cox} D.A. Cox, J. B. Little, H. Schenck, Toric varieties, American Mathematical Society,~ Graduate Studies in Math., Vol.124 (2011).
\bibitem{cox2} E. Cattani, D. Cox and A. Dickenstein, Residues in toric varieties, Compositio Math. \textbf{108} (1997), 35-76.
\bibitem{eu} E. Nobili, Classification of toric 2-Fano 4-folds, Bull of the Braz Math Soc, 42(3),(2011)  399-414. 
\bibitem{sato} H. Sato, Toward the classification of higher-dimensional toric Fano varieties, Tôhoku Math. J. (2) \textbf{52} (2000), no. 3,383-413.
\bibitem{sato2} H. Sato, The numerical class of a surface on a toric manifold, pre-print math.AG/11065949v1, (2011).
\bibitem{kollar} J. Kollár, Y. Miyaoka, S. Mori, Rational connectedness and boundedness of
Fano manifolds, J. Differ. Geom. 36, No.3, 765-779 (1992).
\bibitem{kol} J. Kollár, Y. Miyaoka, S. Mori, Rational curves on Fano varieties, in Classification of irregular varieties, minimal models and abelian varieties, Proc. Conf., Trento/Italy 1990, Lect. Notes Math. 1515, 100-105 (1992).
\bibitem{matsuki} K.~Matsuki, Introduction to the Mori program, Universitext, Springer-Verlag, New York, (2002). 
\bibitem{mauro} M.~Beltramete, A.~Sommese, The adjunction theory of complex projective varieties, Expositions in Mathematics, \textbf{16}, Berlin, New York, W. de Gruyter, (1995).
\bibitem{hering_et_all_GKZ} M. Hering, A. K{\"u}ronya, and S. Payne, Asymptotic cohomological
  functions of toric divisors, Advances in Mathematics, \textbf{207} (2006), no. 2, 634-645.
 \bibitem{nill} M. Kreuzer and B. Nill, Classification of toric Fano 5-folds, Advances in geometry, vol. 9, no. 1, 85–97,(2009).
\bibitem{mustata} M. Mustaţă, Lecture notes on toric varieties, available at \url{http://www.math.lsa.umich.edu/~mmustata/toric_var.html}. 
\bibitem{obro} M. Øbro, An algorithm for the classification of smooth Fano polytopes,
arXiv:0704.0049.
\bibitem{reid} M.~Reid, Decomposition of toric morphisms, 
Arithmetic and geometry, Progr. Math., Vol.II, 395-418,  {\textbf{36}}, Birkh\"auser 
Boston, MA, (1983). 
\bibitem{classificar} P. Kleinschmidt, A classification of toric varieties with few generators, Aequationes Math. 35
(1988), 254–266.
\bibitem{hartshorne} R. Hartshorne, Algebraic Geometry, Graduate texts in Math., Springer-Verlag, 1977.
\bibitem{robert} R. Lazarsfeld, Positivity in Algebraic
    Geometry I, Ergebnisse der Mathematik \textbf{48}, 
    Springer-Verlag 2004.
\bibitem{sandra} S. Di Rocco, Projective duality of toric manifolds and defect polytopes, Proc. LMS, vol.3, \textbf{93} (2006), no.1, pp 85-104.
\bibitem{adjunction} S. Di Rocco, B.Nill, C. Haase and A. Paffenholz, Polyhedral Adjunction, arXiv:1105.2415.
\bibitem{mori} S. Mori, Flip theorem and the existence of minimal models for 3-folds, \textbf{J.} Amer. Math. Soc. \textbf{1} (1988), no. 1, 117-253.
\bibitem{payne} S.~Payne, Stable base loci, movable curves, and small modifications,
for toric varieties,  Math. Z. \textbf{253} (2006), 421-431.
\bibitem{oda} T.~Oda, Torus embeding and applications. Based on joint work with Katsuya Miyake, Tata Institute of Fundamental Research , Bombay, Springer, Berlin (1978).
\bibitem{bat3} V. V. Batyrev, Toroidal Fano 3-folds. Mathematics of the USSR Izvestiya, 19:13–25, (1982).
\bibitem{bat} V. V. Batyrev, On the classification of toric Fano 4-folds, J. Math. Sci. (New York) \textbf{94} (1999), no. 1, 1021-1050.  
\bibitem{fulton} W. Fulton, Introduction to toric varieties, Annals of Mathematics Studies, no. \textbf{131}, Princeton University Press, 1993.
\end{thebibliography}
\end{document}